\documentclass[12pt]{article}
\usepackage{eurosym}
\usepackage{latexsym} \usepackage{graphicx} \usepackage{enumerate}
\usepackage{amsthm} \usepackage{cite} \usepackage{amsmath}
\usepackage{amsfonts} \usepackage{amssymb} \usepackage{graphicx}
\usepackage{subfig} \usepackage[normalem]{ulem} \usepackage{hyperref}

\providecommand{\U}[1]{\protect\rule{.1in}{.1in}}
\newtheorem{theorem}{Theorem}[section]

\newtheorem{corollary}[theorem]{Corollary}

\newtheorem{definition}[theorem]{Definition}

\newtheorem{lemma}[theorem]{Lemma}

\newtheorem{remark}[theorem]{Remark}

\newcommand{\BIGOP}[1]{\mathop{\mathchoice{\raise-0.22em\hbox{\huge
        $#1$}} {\raise-0.05em\hbox{\Large $#1$}}{\hbox{\large
        $#1$}}{#1}}}
\makeatletter\@addtoreset{equation}{section}\makeatother
\evensidemargin=\oddsidemargin 
\newdimen\dummy \dummy=\oddsidemargin 
\newcommand{\llbr}{[\![} \newcommand{\rrbr}{]\!]}
 
\makeatletter \g@addto@macro\@floatboxreset\centering \makeatother
 
\begin{document}

\title{A Posteriori Error Estimation of $hp$-$\operatorname*{dG}$ Finite Element
Methods for Highly Indefinite Helmholtz Problems (extended version)\thanks{This paper is based on the
master's thesis \cite{ZechMaster}, which has been worked out during a visit of the
second author at the Institut f\"{u}r Mathematik, Universit\"{a}t Z\"{u}rich.}} 

\author{S. Sauter\thanks{Institut f\"{u}r Mathematik, Universit\"{a}t Z\"{u}rich,
Winterthurerstrasse 190, CH-8057 Z\"{u}rich, Switzerland, e-mail:\href{mailto:stas@math.uzh.ch}{stas@math.uzh.ch}}
\and J. Zech\thanks{Seminar for Applied Mathematics, ETH Z\"{u}rich, CH-8092
Z\"{u}rich, Switzerland, e-mail: \href{mailto:jakob.zech@sam.math.ethz.ch}{jakob.zech@sam.math.ethz.ch}}}

\date{}
\maketitle

\begin{abstract}
In this paper, we will consider an $hp$-finite elements discretization of a
highly indefinite Helmholtz problem by some $\operatorname{dG}$ formulation
which is based on the \textit{ultra-weak variational formulation} by Cessenat
and Depr\'{e}s.

We will introduce an a posteriori error estimator and derive reliability and
efficiency estimates which are explicit with respect to the wavenumber and the
discretization parameters $h$ and $p$. In contrast to the conventional
conforming finite element method for indefinite problems, the
$\operatorname{dG}$ formulation is unconditionally stable and the adaptive
discretization process may start from a very coarse initial mesh.

Numerical experiments will illustrate the efficiency and robustness of the method.

\end{abstract}


\section{Introduction\label{SecIntro}}

High frequency scattering problems are ubiquitous in many fields of
science and engineering and their reliable and efficient numerical
simulation pervades numerous engineering applications such as
detection (e.g., radar), communication (e.g., wireless), and medicine
(e.g., sonic imaging) (\cite{Jackson02}, \cite{Achenbach2005}). These
phenomena are governed by systems of linear partial differential
equations (PDEs); the wave equation for elastic waves and the Maxwell
equations for electromagnetic scattering. We are here interested in
time-harmonic problems where the equation can be reduced to purely
spatial problems; for high frequencies these PDEs become highly
indefinite and the development of accurate numerical solution methods
is far from being in a mature state.

In this paper we will consider the Helmholtz problem with high
wavenumber as our model problem. Although the continuous problem with
appropriate boundary conditions has a unique solution, conventional
$hp$-finite element methods require a minimal resolution condition
such that existence and uniqueness is guaranteed on the discrete level
(see, e.g., \cite{Frankp}, \cite{Ihlenburgbook},
\cite{MelenkSauterMathComp}, \cite{mm_stas_helm2},
\cite{Graham_ActaNumerica_2012}). However, this condition, typically,
contains a generic constant $C$ which is either unknown for specific
problems or only very pessimistic estimates are available. This is one
of the major motivations for the development of \textit{stabilized}
formulations such that the discrete system is always solvable --
well-known examples include \textit{least square techniques
}\cite{monk-wang99,IH06,IH92,ACDG10} and discontinuous Galerkin
($\operatorname{dG}$) methods \cite{FW09,FW,FX,wu12a,zhu-wu12b}. These
formulations lead to discrete systems which are unconditionally
stable, i.e., no resolution condition is required. Although
convergence starts for these methods only after a resolution condition
is reached, the stability of the discrete system is considerably
improved. The \textit{Ultra Weak Variational Formulation} (UWVF) of
Cessenat and Despr\'{e}s
\cite{CessenatDespres1998,CessenatDespres2003,D94} can be understood
as a $\operatorname*{dG}$-method that permits the use of non-standard,
discontinuous local discretization spaces such as plane waves (see
\cite{GittelsonHiptmairPerugia2007,HuttungenMonk2007,hiptmair-moiola-perugia09b,BuffaMonk2007}%
). In this paper we will employ a $hp$-$\operatorname*{dG}$-finite
element method based on the UWVF which was developed in
\cite{GittelsonHiptmairPerugia2007} and generalized in \cite{MPS13}.

Our focus here is on the development of an a posteriori error
estimator for this formulation and its analysis which is explicit with
respect to the discretization parameters $h$, $p$, and the
wavenumber. In contrast to definite elliptic problems, there exist
only relatively few publications in the literature on a posteriori
estimation for highly indefinite problems (cf.  \cite{IrimieBuillard},
\cite{Strouboulis1}, \cite{Strouboulis2}, \cite{Schatz74},
\cite{doerfler_sauter}). The papers which are closely related to our
work are \cite{HoppeHelmholtz2013} and \cite{doerfler_sauter}: a) In
\cite{HoppeHelmholtz2013}, an a posteriori error estimator for the
Helmholtz problem has been developed for the interior penalty
discontinuous Galerkin (IPDG) method and reliability, efficiency,
\textit{and} convergence of the resulting adaptive method is
proved. In contrast, we do not prove the convergence of the resulting
adaptive method for our $\operatorname{dG}%
$-formulation. On the other hand, our estimators are properly weighted
with the polynomial degree and the estimates are explicit with respect
to the wavenumber $k$, the mesh width $h$, \textit{and} the polynomial
degree $p$. In addition, the dependence of the constants in the
estimates on the wavenumber $k$ are milder in our approach compared to
\cite{HoppeHelmholtz2013}; b) In \cite{doerfler_sauter}, a
\textit{residual} a posteriori error estimator (cf.  \cite{BaRe1},
\cite{BaRe2}, \cite{AiOd3}, \cite{verfuerth_book_neu}) has been
developed for the conventional $hp$-finite element method. Although
efficiency and reliability estimates have been proved, a strict
minimal resolution condition is required for the initial finite
element space and this is a severe drawback in the context of adaptive
discretization.

We will prove in this paper, that our a posteriori error estimator for
the $hp$-$\operatorname*{dG}$-finite element method does not require
this strict condition and allows to start the adaptive discretization
process from very coarse finite element meshes and no a priori
information is required.

The paper is organized as follows. In Section \ref{SecdG}, we will
introduce the model problem and its $\operatorname{dG}$-discretization
by $hp$-finite elements. We will recall its unconditional stability
and state the quasi-optimal convergence.

Section \ref{SecApostAna} is devoted to the definition of the residual
a posteriori error estimator and we will prove its reliability and
efficiency.

In Section \ref{SecNumExp} we will present an adaptive discretization
process and report on numerical experiments which illustrate the
behavior of the method for specific model problems such as smooth
problems, problems with singularities, problems with constant,
varying, and discontinuous wavenumber, and the dependence on the
polynomial degree of approximation.

The proof of reliability employs a new $hp$-$C^{1}$ Cl\'{e}ment-type
interpolation operator which will be defined in Appendix \ref{AApprox}
and $hp$-explicit approximation results are proved.

\section{Discontinuous Galerkin (dG)-Discretization\label{SecdG}}

\subsection{Helmholtz Equation with Robin Boundary Conditions}

Let $\Omega\subset\mathbb{R}^{2}$ be a bounded Lipschitz domain with
boundary $\partial\Omega$. The scalar product in $L^{2}\left(
  \Omega\right) $ is denoted by $\left( u,v\right)
:=\int_{\Omega}u\overline{v}$ and the norm by $\left\Vert
  \cdot\right\Vert $.

For $s>0$, the space $H^{s}\left( \Omega\right) $ is the usual Sobolev
space with norm $\left\Vert \cdot\right\Vert _{H^{s}\left(
    \Omega\right) }$. The dual space is denoted by $\left( H^{s}\left(
    \Omega\right) \right) ^{\prime}$ and the trace spaces by
$H^{\sigma}\left( \partial\Omega\right) $ with norm $\left\Vert
  \cdot\right\Vert _{H^{\sigma}\left( \partial \Omega\right) }$. For
$\sigma=0$, we write $\left\Vert \cdot\right\Vert _{\partial\Omega}$
short for $\left\Vert \cdot\right\Vert _{L^{2}\left(
    \partial\Omega\right) } $. The seminorms containing only the
highest derivatives are denoted by $\left\vert \cdot\right\vert
_{H^{s}\left( \Omega\right) }$ and $\left\vert \cdot\right\vert
_{H^{\sigma}\left(
    \partial\Omega\right) }$.

For given $f\in L^2 ( \Omega) $, $g\in L^2(
\partial\Omega)$ we consider the Helmholtz equation with Robin
boundary condition
\begin{equation*}%
  \begin{aligned}
    -\Delta u-k^2 u&=f &&\text{in~}\Omega,\\
    \partial_{\mathbf{n}}u + iku &=g &&\text{on~}\partial\Omega,
  \end{aligned}
\end{equation*}
where $\partial_{\mathbf{n}}u$ denotes the outer normal derivative of
$u$ on the boundary.  In most parts of this paper we assume that $k$
is a positive constant. This is a simplification compared to the
following more general case: There exist positive constants $\kappa$
and $k_{\max}$ such that%
\begin{equation}%
  \begin{array}
    [c]{cll}%
    k\in L^{\infty}\left(  \Omega,\mathbb{R}\right)  , &  & 1<\kappa\leq k\left(
      x\right)  \leq k_{\max}<\infty,\\[3pt]%
    k=\kappa &  & \text{in a neighborhood of }\partial\Omega.
  \end{array}
  \label{min_conditions_on_k}%
\end{equation}
We define the method for, possibly, variable wavenumbers $k$ which
satisfy (\ref{min_conditions_on_k}) while the error analysis is
restricted to the constant case. In the section on numerical
experiments, we will again consider variable wavenumbers $k$.

The weak formulation reads: Find $u\in H^{1}\left(\Omega\right) $ such
that%
\begin{subequations}
  \label{He}
\end{subequations}%
\begin{equation}
  a\left(u,v\right)  =F\left(v\right)  \qquad\forall v\in H^{1}\left(
    \Omega\right)  \tag{%
    \ref{He}%
    a}\label{Hea}%
\end{equation}
with the sesquilinear form $a:H^{1}\left( \Omega\right) \times
H^{1}\left( \Omega\right) \rightarrow\mathbb{C}$ and linear form
$F:H^{1}\left( \Omega\right) \rightarrow\mathbb{C}$ defined by%
\begin{equation}
  a\left(  u,v\right)  :=\int_{\Omega}\left(  \left\langle \nabla u,\overline
      {\nabla v}\right\rangle -k^{2}u\overline{v}\right)  +\operatorname*{i}%
  \int_{\partial\Omega}ku\overline{v}\quad\text{and\quad}F\left(  v\right)
  :=\int_{\Omega}f\overline{v}+\int_{\partial\Omega}g\overline{v}. \tag{%
    \ref{He}%
    b}\label{Heb}%
\end{equation}
The assumptions on the data can be weakened to $f\in
(H^1(\Omega))^{\prime}$ and $g\in H^{-1/2}(\partial\Omega)$. In this
case the integrals in (\ref{Heb}) are understood as dual pairings.

It is well-known that this problem has a unique solution which depends
continuously on the data.

\begin{definition}
  Let $k$ satisfy (\ref{min_conditions_on_k}). On $H^{1}\left(
    \Omega\right) $, we introduce the norm%
  \[
  \left\Vert u\right\Vert _{\mathcal{H}}:=\left\Vert \nabla
    u\right\Vert +\left\Vert ku\right\Vert .
  \]
\end{definition}

\begin{theorem}
  \label{TheoExUni}Let $\Omega\subseteq\mathbb{R}^2$ be a bounded
  Lipschitz domain and let $k=\kappa>1$ be constant.
  \begin{enumerate}
  \item[a.] There exists a constant $C\left( \Omega,\kappa\right) >0$
    such that for every $f\in\left( H^{1}\left( \Omega\right) \right)
    ^{\prime}$ and $g\in H^{-1/2}\left( \partial\Omega\right) $, there
    exists a unique solution $u\in H^{1}\left( \Omega\right) $ of
    problem (\ref{He}) which satisfies%
    \[
    \left\Vert u\right\Vert _{\mathcal{H}}\leq C\left(
      \Omega,\kappa\right) \left( \left\Vert F\right\Vert _{\left(
          H^{1}\left( \Omega\right) \right) ^{\prime}}+\left\Vert
        g\right\Vert _{H^{-1/2}\left( \partial\Omega\right) }\right) .
    \]
  \item[b.] Let $\Omega\subseteq\mathbb{R}^2$ be a bounded star-shaped
    domain with smooth boundary or a bounded convex domain. There
    exists a constant $C\left( \Omega%
    \right) >0$ (depending only on $\Omega$) such that for any $f\in
    L^{2}\left( \Omega\right) $, $g\in
    H^{1/2}\left( \partial\Omega\right) $, the solution of (\ref{He})
    satisfies%
    \begin{align*}
      \left\Vert u\right\Vert _{\mathcal{H}} & \leq C\left(
        \Omega\right) \left( \left\Vert f\right\Vert +\left\Vert
          g\right\Vert _{\partial\Omega
        }\right)  ,\\
      \left\vert u\right\vert _{H^{2}\left( \Omega\right) } & \leq
      C\left( \Omega\right) \left( 1+\kappa\right) \left( \left\Vert
          f\right\Vert +\left\Vert g\right\Vert _{\partial
          \Omega}+\left\Vert g\right\Vert _{H^{1/2}\left(
            \partial \Omega\right) }\right) .
    \end{align*}
  \end{enumerate}
\end{theorem}

For a proof we refer to \cite[Prop. 8.1.3 and .4]{MelenkDiss}.

\begin{remark}\label{rmk:regularity}
  Let $\Omega\subseteq\mathbb{R}^{2}$ be a polygonal Lipschitz domain
  and let $k=\kappa>1$ be a constant.
  For $f\in L^{2}\left( \Omega\right) $ and $g\in
  H_{\operatorname*{pw}}%
  ^{1/2}\left( \partial\Omega\right) :=\left\{ g\in L^{2}\left(
      \partial\Omega\right) :g\text{ is edgewise in }H^{1/2}\right\}
  $, the classical elliptic regularity theory shows that the unique
  solution $u$ of (\ref{He}) is in $H^{3/2+\varepsilon}(\Omega)$ for
  some $\varepsilon>0$ depending on $\Omega$ and we briefly sketch the
  argument:
  We write (\ref{He}) in the following strong form%
  \begin{equation*}
    \begin{aligned}
      -\Delta u&=\tilde{f}:=f+k^{2}u && \text{in }\Omega,\\
      \partial_{\mathbf{n}}u&=\tilde{g}:=g-\operatorname*{i}ku &&
      \text{on }\partial \Omega.
    \end{aligned}
  \end{equation*}
  Since the solution $u$ of (\ref{He}) is in $H^{1}\left(
    \Omega\right) $, we have $\tilde{f}\in L^{2}\left( \Omega\right) $
  and $\tilde{g}\in
  H_{\operatorname*{pw}}^{1/2}\left( \partial\Omega\right) $. From
  \cite[Lemma A1]{mm_stas_helm2}, we conclude that there exists a
  lifting operator
  $\mathcal{L}:H_{\operatorname*{pw}}^{1/2}\left( \partial\Omega\right)
  \rightarrow H^{2}\left( \Omega\right) $ such that
  $G:=\mathcal{L}\left( g\right) $ satisfies $\partial_{\mathbf{n}}
  G=g$ and $\left\Vert G\right\Vert _{H^{2}\left( \Omega\right) }\leq
  C\left\Vert g\right\Vert
  _{H_{\operatorname*{pw}}^{1/2}\left( \partial\Omega\right) }$. Thus,
  the ansatz $u=u_{0}+\tilde{G}$ with $\tilde{G}:=\mathcal{L}\left(
    \tilde {g}\right) $ leads to%
  \begin{equation*}
    \begin{aligned}
      -\Delta u_{0}&=\check{f}:=\tilde{f}+\Delta\tilde{G} && \text{in }\Omega,\\
      \partial_{\mathbf{n}}u_{0}&=0 && \text{on }\partial \Omega
    \end{aligned}
  \end{equation*}
  with $\check{f}\in L^{2}\left( \Omega\right) $. From \cite[(7.22)]%
  {KellogReg} we obtain that the solution $u_{0}$, and thus also $u$,
  then is in $H^{3/2+\varepsilon}\left( \Omega\right) $ for some
  $\varepsilon>0$.
\end{remark}

\subsection{hp-Finite Elements}

Let $\Omega\subset\mathbb{R}^{2}$ be a polygonal domain and let
$\mathcal{T}%
:=\left\{ K_{i}:1\leq i\leq N\right\} $ denote a simplicial finite
element mesh which is \textit{conforming} in the sense that there are
no \textit{hanging }nodes. With each element $K\in\mathcal{T}$ we
associate a polynomial degree $p_{K}\in\mathbb{N}_{\geq1}$.

The diameter of an element $K\in\mathcal{T}$ is denoted by $h_{K}%
:=\operatorname{diam}K$ and the maximal mesh width is
$h_{\mathcal{T}}%
:=\max\left\{ h_{K}:K\in\mathcal{T}\right\} $. The \textit{minimal}
polynomial degree is%
\[
p_{\mathcal{T}}:=\min\left\{ p_{K}: K\in\mathcal{T}\right\} .
\]
The \textit{shape regularity} of $\mathcal{T}$ is described by the
constant%
\begin{equation}
  \rho_{\mathcal{T}}:=\max\left\{  \frac{h_{K}}{\operatorname{diam}%
      \mathcal{B}_{K}}:K\in\mathcal{T}\right\}  , \label{defrho}%
\end{equation}
where $\mathcal{B}_{K}$ is the maximal inscribed ball in $K$. Since
$\mathcal{T}$ contains finitely many simplices, the constant $\rho
_{\mathcal{T}}$ is always bounded but becomes large if the simplices
are degenerate, e.g., are flat or needle-shaped. The constants in the
following estimates depend on the mesh via the constant
$\rho_{\mathcal{T}}$; they are bounded for any fixed
$\rho_{\mathcal{T}}$ but, possibly, become large for large
$\rho_{\mathcal{T}}$.

Concerning the polynomial degree distribution we assume throughout the
paper that the polynomial degrees of neighboring elements are
comparable\footnote{We use here the same constant $\rho_{\mathcal{T}}$
  as for the shape regularity to simplify the notation.}:%
\begin{equation}
  \rho_{\mathcal{T}}^{-1}\left(  p_{K}+1\right)  \leq p_{K^{\prime}}+1\leq
  \rho_{\mathcal{T}}\left(  p_{K}+1\right)  \qquad\forall K,K^{\prime}%
  \in\mathcal{T}\text{ with }K\cap K^{\prime}\neq\emptyset. \label{defrho2}%
\end{equation}

By convention the triangles $K\in\mathcal{T}$ are closed sets. The
boundary of a triangle $K\in\mathcal{T}$ consists of three
one-dimensional (relatively closed) edges which are collected in the
set $\mathcal{E}\left( K\right) $.
The subset $\mathcal{E}^{I}\left( K\right) \subseteq\mathcal{E}\left(
  K\right) $ of \textit{inner} edges consists of all edges
$e\in\mathcal{E}%
\left( K\right) $ whose relative interior lie in (the open set)
$\Omega$ while $\mathcal{E}^{B}\left( K\right) :=\mathcal{E}\left(
  K\right) \backslash\mathcal{E}^{I}\left( K\right) $ is the set of
\textit{boundary} edges\textit{.} Further we set%
\[
\partial^{B}K:=\partial
K\cap\partial\Omega\quad\text{and\quad}\partial ^{I}K:=\partial
K\backslash\partial^{B}K.
\]

The conformity of the mesh implies that any $e\in\mathcal{E}^{I}\left(
  K\right) $ is shared by two and only two triangles in
$\mathcal{T}$. The sets of inner/boundary/all edges $\mathcal{E}^{I}$,
$\mathcal{E}^{B}$, $\mathcal{E}$, are defined by%
\[
\mathcal{E}^{I}:=\left\{ e\in\mathcal{E}^{I}\left( K\right) :K\in
  \mathcal{T}\right\} ,\quad\mathcal{E}^{B}:=\left\{ e\in\mathcal{E}%
  ^{B}\left( K\right) :K\in\mathcal{T}\right\} ,\quad\mathcal{E}%
:=\mathcal{E}^{I}\cup\mathcal{E}^{B}.
\]
The interior skeleton $\mathfrak{S}^{I}$ is given by
\[
\mathfrak{S}^{I}:=%
{\displaystyle\bigcup\limits_{K\in\mathcal{T}}}
\partial^{I}K.
\]
Next we introduce patches associated with an edge $e$ or an element
$K$ of the triangulation
\[
\omega_{e}:=\bigcup_{\left\{ K^{\prime}\in\mathcal{T}: e \cap
    K^{\prime}%
    \neq\emptyset\right\} }K^{\prime}\quad\text{and}\quad\omega_{K}%
:=\bigcup_{\left\{ K^{\prime}\in\mathcal{T}: K\cap K^{\prime}\neq
    \emptyset\right\} }K^{\prime}.
\]

Furthermore, we employ the notation%
\begin{equation}
  p_{e}:=\min_{\substack{K\in\mathcal{T} \\e\subset\partial K}}p_{K}%
  \quad\text{and\quad}h_{e}:=\left\vert e\right\vert \quad\text{with the length
  }\left\vert e\right\vert \text{ of }e\text{.} \label{defehe}%
\end{equation}
We define the \textit{mesh functions} $\mathfrak{h}_{\mathcal{T}}%
,\mathfrak{p}_{\mathcal{T}}\in L^{\infty}\left( \Omega\right) $ and
$\mathfrak{h}_{\mathcal{E}},\mathfrak{p}_{\mathcal{E}}$ $\in
L^{\infty}\left( \mathfrak{S}\right) $ by%
\[
\forall K\in\mathcal{T}:~~\left.  \left(
    \mathfrak{h}_{\mathcal{T}}\right) \right\vert _{K}:=h_{K},~~\left.
  \mathfrak{p}_{\mathcal{T}}\right\vert
_{K}:=p_{K}\quad\text{and\quad}\forall e\in\mathcal{E}:~~\left.
  \left( \mathfrak{h}_{\mathcal{E}}\right) \right\vert
_{e}:=h_{e},~~\left.  \mathfrak{p}_{\mathcal{E}}\right\vert
_{e}:=p_{e}.
\]
We skip the indices $\mathcal{T}$ and $\mathcal{E}$ and write short
$\mathfrak{h}$, $\mathfrak{p}$ if no confusion is possible. In the
error estimates, the quantity $k\mathfrak{h}/\mathfrak{p}$ will play
an important role since it is a measure how well the $hp$-finite
element space \textit{resolves} the oscillations in the
solution. Therefore we define
\begin{equation}
  M_{\frac{\operatorname*{kh}}{\operatorname*{p}}}:=\max\left\{  \left\Vert
      \frac{\mathfrak{h}}{\mathfrak{p}}k\right\Vert _{L^{\infty}\left(
        \mathfrak{S}\right)  },\left\Vert \frac{\mathfrak{h}}{\mathfrak{p}%
      }k\right\Vert _{L^{\infty}\left(  \Omega\right)  }\right\}  . \label{defMkhp}%
\end{equation}

The non-conforming $hp$-finite element space for the mesh
$\mathcal{T}$ with local polynomials of degree $p_{K}$ is given by
\begin{equation}
  S_{\mathcal{T}}^{\mathfrak{p}}:=\left\{  u\in L^{2}\left(  \Omega\right)  :
    \left.  u\right\vert _{K}\in\mathbb{P}_{p_{K}}~\forall K\in\mathcal{T}%
  \right\}  . \label{defspt}%
\end{equation}
Here $\mathbb{P}_{p}$ denotes the space of bivariate polynomials of
maximal total degree $p$. For a subset $\omega\subset\Omega$, we write
$\mathbb{P}%
_{p}\left( \omega\right) $ to indicate explicitly that we consider
$u\in\mathbb{P}_{p}\left( \omega\right) $ as a polynomial on $\omega$.

Finally, throughout this paper $C>0$ stands for a generic constant
that does not depend on the parameters $k$, $h_{K}$, and $p_{K}$ and
may change its value in each occurence.

\subsection{$\operatorname{dG}$ Formulation}

For the discretization of the Helmholtz problem we employ a
$\operatorname{dG}%
$ formulation which has been derived from the ultra-weak variational
formulation (cf. \cite{CessenatDespres1998,CessenatDespres2003,D94})
in \cite{GittelsonHiptmairPerugia2007},
\cite{hiptmair-moiola-perugia09b}, and generalized in \cite{MPS13}. It
involves jumps and mean values across edges which we will introduce
next. For an inner edge $e\in\mathcal{E}^{I}$ with two adjacent
triangles $K$, $K^{\prime}\in\mathcal{T}$ we set for simplexwise
sufficiently smooth functions $v$ and vector valued functions
$\mathbf{w}$%
\begin{align*}
  \hspace{-0.1cm}%
  \begin{array} [c]{rlcrl}%
    \left.
      \llbr
      v%
      \rrbr
    \right\vert _{e}\hspace{-0.2cm} & :=\left.  \left(  \left.  v\right\vert
        _{K}\right)  \right\vert _{e}-\left.  \left(  \left.  v\right\vert
        _{K^{\prime}}\right)  \right\vert _{e},\hspace{-0.1cm} &  & \left.  \left\{
        v\right\}  \right\vert _{e}\hspace{-0.2cm} & :=\frac{1}{2}\left(  \left.
        \left(  \left.  v\right\vert _{K}\right)  \right\vert _{e}+\left.  \left(
          \left.  v\right\vert _{K^{\prime}}\right)  \right\vert _{e}\right)  ,\\
    \left.
      \llbr
      v%
      \rrbr
      _{N}\right\vert _{e}\hspace{-0.2cm} & :=\left.  \left(  \left.  v\right\vert
        _{K}\right)  \right\vert _{e}\mathbf{n}_{K}+\left.  \left(  \left.
          v\right\vert _{K^{\prime}}\right)  \right\vert _{e}\mathbf{n}_{K^{\prime}%
    },\hspace{-0.1cm} &  & \left.
      \llbr
      \mathbf{w}%
      \rrbr
      _{N}\right\vert _{e}\hspace{-0.2cm} & :=\left.  \left(  \left.  \mathbf{w}%
        \right\vert _{K}\right)  \right\vert _{e}\cdot\mathbf{n}_{K}+\left.  \left(
        \left.  \mathbf{w}\right\vert _{K^{\prime}}\right)  \right\vert _{e}%
    \cdot\mathbf{n}_{K^{\prime}},
  \end{array}
\end{align*}
where $\mathbf{n}_{K}$, $\mathbf{n}_{K^{\prime}}$ are the respective
outer normal vectors on the boundary of $K$ and $K^{\prime}$ and
\textquotedblleft%
$\cdot$\textquotedblright\ denotes the Euclidean scalar product. The
sign in $\left.
  \llbr
  v%
  \rrbr
\right\vert _{e}$ is arbitrary.

The $\operatorname{dG}$-discretization of (\ref{He}) reads: Find
$u_{\mathcal{T}}\in S_{\mathcal{T}}^{\mathfrak{p}}$ such that%
\begin{subequations}
  \label{dgsesquilin}
\end{subequations}%
\begin{equation}
  a_{\mathcal{T}}\left(  u_{\mathcal{T}},v\right)  =F_{\mathcal{T}}\left(
    v\right)  \qquad\forall v\in S_{\mathcal{T}}^{\mathfrak{p}}\tag{%
    \ref{dgsesquilin}%
    a}\label{dgsesquilin0}%
\end{equation}
with the sesquilinear form%
\begin{align}
  a_{\mathcal{T}}\left( u,v\right) := & (\nabla_{\mathcal{T}}u,\nabla
  _{\mathcal{T}}v)-k^{2}\left( u,v\right) -\left(
    \llbr
    u%
    \rrbr
    _{N},\left\{ \nabla_{\mathcal{T}}v\right\} \right)
  _{\mathfrak{S}^{I}%
  }-\left( \left\{ \nabla_{\mathcal{T}}u\right\} ,%
    \llbr
    v%
    \rrbr
    _{N}\right)  _{\mathfrak{S}^{I}}\nonumber\\
  & -\left( {\mathfrak{d}\frac{k\mathfrak{h}}{\mathfrak{p}}}u,\nabla
    _{\mathcal{T}}v\cdot\mathbf{n}\right) _{\partial\Omega}-\left(
    {\mathfrak{d}\frac{k\mathfrak{h}}{\mathfrak{p}}}\nabla_{\mathcal{T}}%
    u\cdot\mathbf{n,}v\right)  _{\partial\Omega}\nonumber\\
  & -\frac{1}{\operatorname*{i}}\left(
    {\mathfrak{b}\frac{\mathfrak{h}%
      }{\mathfrak{p}}}%
    \llbr
    \nabla_{\mathcal{T}}u%
    \rrbr
    _{N},%
    \llbr
    \nabla_{\mathcal{T}}v%
    \rrbr
    _{N}\right) _{\mathfrak{S}^{I}}-\frac{1}{\operatorname*{i}}\left(
    {\mathfrak{d}\frac{\mathfrak{h}}{\mathfrak{p}}}\nabla_{\mathcal{T}}%
    u\cdot\mathbf{n},\nabla_{\mathcal{T}}v\cdot\mathbf{n}\right)
  _{\partial
    \Omega}\nonumber\\
  & +\operatorname*{i}\left( {\mathfrak{a}\frac{\mathfrak{p}^{2}}%
      {\mathfrak{h}}}%
    \llbr
    u%
    \rrbr
    _{N},%
    \llbr
    v%
    \rrbr
    _{N}\right) _{\mathfrak{S}^{I}}+\operatorname*{i}\left( k\left(
      1-{\mathfrak{d}\frac{k\mathfrak{h}}{\mathfrak{p}}}\right)
    u,v\right) _{\partial\Omega},\tag{%
    \ref{dgsesquilin}%
    b}\label{dgsesquilina}%
\end{align}
where $\nabla_{\mathcal{T}}$ denotes the simplexwise gradient, $\Delta
_{\mathcal{T}}$ the simplexwise Laplacean, and $\left(
  \cdot,\cdot\right) _{\mathfrak{S}^{I}}$, $\left( \cdot,\cdot\right)
_{{\partial\Omega}}$ are the $L^{2}\left( \mathfrak{S}^{I}\right) $
and $L^{2}\left( \partial \Omega\right) $ scalar products. Moreover,
the fixed constants
\[
\mathfrak{a}>0,\qquad\mathfrak{b}>0,\qquad\mathfrak{d}>0
\]
are at our disposal and will be adjusted later. The functional
$F_{\mathcal{T}%
}$ is defined by%
\begin{equation}
  F_{\mathcal{T}}\left(  v\right)  :=\left(  f,v\right)  -\left(  \frac
    {{\mathfrak{d}\mathfrak{h}}}{\operatorname*{i}\mathfrak{p}}g,\nabla
    _{\mathcal{T}}v\cdot\mathbf{n}\right)  _{\partial\Omega}+\left(  \left(
      1-{\mathfrak{d}\frac{k\mathfrak{h}}{\mathfrak{p}}}\right)  g,v\right)
  _{\partial\Omega}.\tag{%
    \ref{dgsesquilin}%
    c}\label{dgsesquilinb}%
\end{equation}

\begin{remark}
  \label{RemUniqueSolvability}In \cite[Section 3, Remark 3.2]{MPS13}
  it is proved that the condition:%
  \begin{equation}
    \left\Vert {\mathfrak{d}\frac{k\mathfrak{h}}{\mathfrak{p}}}\right\Vert
    _{L^{\infty}\left(  \partial\Omega\right)  }{<1/2}\label{solvabilitycond}%
  \end{equation}
  implies the unique solvability of the discrete system
  (\ref{dgsesquilin}). As a consequence, the discrete system is
  \emph{always} solvable for sufficiently small $\mathfrak{d}>0$. In
  addition, for any \emph{fixed} $\mathfrak{d}>0$, condition
  (\ref{solvabilitycond}) can be regarded as an \emph{explicit}
  condition on $\mathfrak{h}$ and $\mathfrak{p}$. This is a
  significant improvement compared to the condition%
  \[
  \left\Vert
    {\mathfrak{d}\frac{k\mathfrak{h}}{\mathfrak{p}}}\right\Vert
  _{L^{\infty}\left( \partial\Omega\right) }{<C\quad}\text{for
    \textquotedblleft sufficiently\textquotedblright\ small }C>0
  \]
  which is typically imposed for the solvability of the standard
  finite element discretization of the Helmholtz problem
  (cf. \cite[Sec. 4.1.3]{Ihlenburgbook} and
  \cite{MelenkSauterMathComp, mm_stas_helm2}).
\end{remark}

\begin{remark}
  For $s>0$, let the \emph{broken }Sobolev space
  $H_{\mathcal{T}}^{s}\left( \Omega\right) $ be defined by%
  \[
  H_{\mathcal{T}}^{s}\left( \Omega\right) :=\left\{ u\in L^{2}\left(
      \Omega\right) \mid\forall K\in\mathcal{T}:\left.  u\right\vert
    _{K}\in H^{s}\left( K\right) \right\} .
  \]
  Then, $a_{\mathcal{T}}\left( \cdot,\cdot\right) $ can be extended to
  a sesquilinear form on $H_{\mathcal{T}}^{3/2+\varepsilon}\left(
    \Omega\right) \times H_{\mathcal{T}}^{3/2+\varepsilon}\left(
    \Omega\right) $ and $F_{\mathcal{T}}\left( \cdot\right) $ to a
  linear functional $F_{\mathcal{T}%
  }:H_{\mathcal{T}}^{3/2+\varepsilon}\left( \Omega\right) \rightarrow
  \mathbb{C}$ for any $\varepsilon>0$.
\end{remark}

\subsection{Discrete Stability and Convergence}

The following mesh-depending norms on
$H_{\mathcal{T}}^{3/2+\varepsilon }\left( \Omega\right) $ for
$\varepsilon>0$ have been introduced in
\cite{GittelsonHiptmairPerugia2007}:
\begin{subequations}
  \label{dgnorms}
\end{subequations}%
\begin{align}
  \left\Vert v\right\Vert _{\operatorname{dG}}:= & \left( \left\Vert
      \nabla_{\mathcal{T}}v\right\Vert ^{2}+\left\Vert
      \sqrt{{\mathfrak{b}%
          \frac{\mathfrak{h}}{\mathfrak{p}}}}%
      \llbr
      \nabla_{\mathcal{T}}v%
      \rrbr
      _{N}\right\Vert _{\mathfrak{S}^{I}}^{2}+\left\Vert
      \sqrt{{\mathfrak{a}%
          \frac{\mathfrak{p}^{2}}{\mathfrak{h}}}}%
      \llbr
      v%
      \rrbr
      _{N}\right\Vert _{\mathfrak{S}^{I}}^{2}\right. \tag{%
    \ref{dgnorms}%
    a$_1$}\label{dgnormsa}\\
  & \left.  +\left\Vert
      \sqrt{{\mathfrak{d}\frac{\mathfrak{h}}{\mathfrak{p}}}%
      }\nabla_{\mathcal{T}}v\cdot\mathbf{n}\right\Vert
    _{\partial\Omega}%
    ^{2}+\left\Vert \sqrt{k\left( 1-{\mathfrak{d}\frac{k\mathfrak{h}%
            }{\mathfrak{p}}}\right) }v\right\Vert
    _{\partial\Omega}^{2}+\left\Vert kv\right\Vert ^{2}\right)
  ^{1/2},\tag{%
    \ref{dgnorms}%
    a$_2$}\label{dgnormsa2}\\
  \;\left\Vert v\right\Vert _{\operatorname{dG}^{+}}:= & \left(
    \left\Vert v\right\Vert _{\operatorname{dG}}^{2}+\left\Vert \left(
        {\mathfrak{a}%
          \frac{\mathfrak{p}^{2}}{\mathfrak{h}}}\right) ^{-1/2}\left\{
        \nabla _{\mathcal{T}}v\right\} \right\Vert
    _{\mathfrak{S}^{I}}^{2}\right) ^{1/2}.  \tag{%
    \ref{dgnorms}%
    b}\label{dgnormsb}%
\end{align}

Before formulating the stability and convergence theorem, we have to
introduce some notation.

The \emph{adjoint Helmholtz problem} reads: For given $w\in
L^{2}(\Omega)$, find $z\in H^{1}\left( \Omega\right) $ such that%
\begin{equation}
  a\left(  v,z\right)  =\left(  v,w\right)  \qquad\forall v\in H^{1}\left(
    \Omega\right)  . \label{adjproblem}%
\end{equation}
The assumptions of Theorem \ref{TheoExUni} ensure well-posedness of
the adjoint problem (cf. \cite[Prop. 8.1.4]{MelenkDiss},
\cite{cummings-feng06}, \cite[Thm. 2.4]{MelenkHelmStab2010},
\cite{MPS13}) and defines a bounded solution operator
$Q_{k}^{\star}:L^{2}(\Omega)\rightarrow H^{1}\left( \Omega\right) $,
$w\mapsto z$.

\begin{lemma}
  \label{LemAdjCons}Let $\Omega\subset\mathbb{R}^{2}$ be a polygonal
  Lipschitz domain and let $w\in L^{2}\left( \Omega\right) $. Then,
  (\ref{adjproblem}) is a well-posed problem.  Denote its solution by
  $z$. Then $z$ satisfies $z\in H^{3/2+\varepsilon}$ for some
  $\varepsilon>0$ depending on $\Omega$ and moreover
  \[
  a_{\mathcal{T}}\left( v,z\right) =\left( v,w\right) \qquad\forall
  v\in H_{\mathcal{T}}^{3/2+\varepsilon}\left( \Omega\right) \text{.}%
  \]
\end{lemma}

This follows from \cite[Rem. 2.6, Lem. 2.7.]{MPS13}.

The key role for the convergence estimates for Helmholtz-type problems
is played by the \emph{adjoint approximation property} which will be
defined next.

\begin{definition}
  Let $S\subset H^{1}\left( \Omega\right) $ be a subspace of
  $H^{1}\left( \Omega\right)$.  Then the adjoint approximation
  property is given by
  \begin{equation}
    \sigma_{k}^{\star}\left(  S\right)  :=\sup_{g\in L^{2}(\Omega)\backslash
      \left\{  0\right\}  }\inf_{v\in S}\dfrac{\left\Vert Q_{k}^{\star}\left(
          k^{2}g\right)  -v\right\Vert _{\operatorname{dG}^{+}}}{\left\Vert
        kg\right\Vert }. \label{defsigmastar}%
  \end{equation}
\end{definition}

There holds the following result on uniqueness and quasi-optimality of
the $\operatorname*{dG}$-finite element solution (see
\cite[Sec. 3]{MPS13}, \cite[Rem. 2.3.1, .2 and
Thm. 2.3.5]{ZechMaster}, and Remark
\ref{rmk:regularity}). \begin{theorem}
  \label{thm:stab} Let $k=\kappa$ be constant satisfying
  (\ref{min_conditions_on_k}). Let $\Omega\subset\mathbb{R}^{2}$ be a
  polygonal Lipschitz domain. Furthermore assume that the constant
  $\mathfrak{a}$ in (\ref{dgsesquilina}) is chosen sufficiently large
  and condition (\ref{solvabilitycond}) is fulfilled. Then, the
  $\operatorname*{dG}$-problem (\ref{dgsesquilin}) has a unique
  solution $u_{\mathcal{T}}\in S_{\mathcal{T}}^{\mathfrak{p}}$.  If,
  in addition, the adjoint approximation condition%
  \begin{align}\label{adjointapp}
    \sigma_{k}^{\ast}\left( S_{\mathcal{T}}^{\mathfrak{p}}\right) \leq
    C_{\ast}%
  \end{align}
  holds for some $C_{\ast}>0$, then, the quasi-optimal error estimate%
  \[
  \left\Vert u-u_{\mathcal{T}}\right\Vert _{\operatorname{dG}}\leq
  C\inf_{v\in S_{\mathcal{T}}^{\mathfrak{p}}}\left\Vert u-v\right\Vert
  _{\operatorname{dG}^{+}}%
  \]
  holds, where $C$ is independent of $k$, $\mathfrak{h}$, and
  $\mathfrak{p}$.
\end{theorem}

\section{A Posteriori Error Estimation\label{SecApostAna}}

In this section we will derive and analyze a residual type a posteriori
estimator for the $\operatorname{dG}$-formulation (\ref{dgsesquilin}) of the
Helmholtz problem (\ref{He}). General techniques of a posteriori error
estimation for elliptic problems are described in \cite{AiOd3},
\cite{MelenkWohlmuth}, \cite{verfuerth_book_neu} while the focus in
\cite{PietroErn} is on $\operatorname{dG}$-methods. A posteriori error
estimation for the conventional conforming discretization of the Helmholtz
problem are described in \cite{doerfler_sauter} and for an IPDG method in
\cite{HoppeHelmholtz2013}.

For the derivation of an a posteriori error estimator for the
$\operatorname{dG}$-formulation of the Helmholtz problem the main challenges
are a) the lower order term $-k^{2}\left(  \cdot,\cdot\right)  $ in the
sesquilinear forms $a\left(  \cdot,\cdot\right)  $ and $a_{\mathcal{T}}\left(
\cdot,\cdot\right)  $, which causes the problem to be highly indefinite and b)
the integrals in \eqref{dgsesquilina} containing the mean of the gradient on
interior edges, which have the effect that $a_{\mathcal{T}}\left(  \cdot
,\cdot\right)  +2k^{2}(\cdot,\cdot)_{L^{2}}$ is not coercive on
$H_{\mathcal{T}}^{3/2+\varepsilon}\cap H^{1}(\Omega)$, $\varepsilon>0$, with
respect to the norm $\left\Vert \cdot\right\Vert _{\mathcal{H}}$.

\subsection{The Residual Error Estimator}

\begin{definition}
  For $v\in S_{\mathcal{T}}^{\mathfrak{p}}$ and $K\in \mathcal{T}$,
  the \emph{local error estimator} is%
  \begin{subequations}
    \label{locerrorest}
  \end{subequations}%
  \begin{equation}
    \eta_{K}\left(  v\right)  :=\left(  \eta_{R_{K}}^{2}\left(  v\right)
      +\eta_{E_{K}}^{2}\left(  v\right)  +\eta_{J_{K}}^{2}\left(  v\right)  \right)
    ^{1/2}\tag{%
      \ref{locerrorest}%
      a}\label{locerroresta}%
  \end{equation}
  with the \emph{internal residual} $\eta_{R_{K}}$, the \emph{edge
    residual} $\eta_{E_{K}}$, and the \emph{trace residual
  }$\eta_{J_{K}}$ given by%
  \begin{align}
    \eta_{R_{K}}\left( v\right) & :=\left( \frac{h_{K}}{p_{K}}\right)
    \left\Vert \Delta_{\mathcal{T}}v+k^{2}v+f\right\Vert _{L^{2}\left(
        K\right) }\tag{%
      \ref{locerrorest}%
      b}\label{locerrorestb}\\
    \eta_{E_{K}}\left( v\right) & :=\left\{ \frac{1}{2}\left\Vert
        \sqrt {{\mathfrak{b}\frac{\mathfrak{h}}{\mathfrak{p}}}}%
        \llbr
        \nabla_{\mathcal{T}}v%
        \rrbr
        _{N}\right\Vert _{\partial^{I}K}^{2}+\left\Vert
        \sqrt{\mathfrak{h}}\left(
          g-\partial_{\mathbf{n}}v-\operatorname*{i}kv\right)
      \right\Vert _{\partial^{B}K}%
      ^{2}\right\} ^{1/2},\tag{%
      \ref{locerrorest}%
      c}\label{locerrorestc}\\
    \eta_{J_{K}}\left( v\right) & :=\frac{1}{\sqrt{2}}\left\Vert
      \sqrt{{\mathfrak{a}\frac{\mathfrak{p}^2}{\mathfrak{h}}} }%
      \llbr
      v%
      \rrbr
    \right\Vert _{\partial^{I}K}.\tag{%
      \ref{locerrorest}%
      d}\label{locerrorestd}%
  \end{align}
  The \emph{global error estimator} is%
  \begin{subequations}
    \label{globerrorest}
  \end{subequations}%
  \begin{equation}
    \eta\left(  v\right)  :=\left(  \eta_{R}^{2}(v)+\eta_{E}^{2}(v)+\eta_{J}^{2}(v)\right)
    ^{1/2}\tag{%
      \ref{globerrorest}%
      a}\label{globerroresta}%
  \end{equation}
  with%
  \begin{equation}
    \eta_{R}\left(  v\right)  :=\left(  \sum_{K\in\mathcal{T}}\eta_{R_{K}}%
      ^{2}\left(  v\right)  \right)  ^{2},~~\eta_{E}\left(  v\right)  :=\left(
      \sum_{K\in\mathcal{T}}\eta_{E_{K}}^{2}\left(  v\right)  \right)  ^{2}%
    ,~~\eta_{J}\left(  v\right)  :=\left(  \sum_{K\in\mathcal{T}}\eta_{J_{K}%
      }^{2}\left(  v\right)  \right)  ^{2}.\tag{%
      \ref{globerrorest}%
      b}\label{globerrorestb}%
  \end{equation}
  For the solution $u_{\mathcal{T}}$ of (\ref{dgsesquilin}), we write
  $\eta$ short for $\eta\left( u_{\mathcal{T}}\right) $ and similarly
  for $\eta_{R_{K}}$, $\eta_{E_{K}}$, etc.
\end{definition}

\subsection{Reliability}

We start the derivation of the reliability estimate by bounding the
$\operatorname{dG}$-norm of the error by parts of the estimator plus
the $k$-weighted $L^{2}$-norm of the error.

\begin{lemma}
  \label{LemReliability}Let $\Omega\subseteq\mathbb{R}^2$ be a
  polygonal Lipschitz domain.  Let $k=\kappa$ be constant satisfying
  (\ref{min_conditions_on_k}) and let $p_{\mathcal{T}}\geq 1$. Let
  $u\in H^{3/2+\varepsilon}\left( \Omega\right) $ be the solution of
  (\ref{He}) for some $\varepsilon>0$ and assume that
  $u_{\mathcal{T}}\in S_{\mathcal{T}}^{\mathfrak{p}}$ solves
  (\ref{dgsesquilin}). Furthermore assume that the constant
  $\mathfrak{a}$ in (\ref{dgsesquilina}) is chosen sufficiently large.
  Then, there exists a constant $C>0$ which only depends on
  $\mathfrak{b}$, $\mathfrak{d}$, $\rho_{\mathcal{T}}$, and $\Omega$
  such that%
  \[
  \left\Vert u-u_{\mathcal{T}}\right\Vert _{\operatorname{dG}}\leq
  C\left( C_{\operatorname{conf}}^{3/2}\eta\left(
      u_{\mathcal{T}}\right) +C_{\operatorname{conf}}^{1/2}\left\Vert
      k\left( u-u_{\mathcal{T}}\right) \right\Vert \right) .
  \]
  where%
  \[
  C_{\operatorname{conf}}:=1+M_{\frac{\operatorname*{kh}}{\operatorname*{p}}}.
  \]
\end{lemma}

Before we prove this lemma we compute an alternative representation of
the term $a_{\mathcal{T}}\left( u-u_{\mathcal{T}},v\right) $ which
will be used frequently in the following.

\begin{lemma}
  \label{LemAltRepA}Let $u\in H^{3/2+\varepsilon}(\Omega)$ be the
  solution of (\ref{He}) for some $\varepsilon>0$ and assume that
  $u_{\mathcal{T}}\in S_{\mathcal{T}}^{\mathfrak{p}}$ solves
  (\ref{dgsesquilin}). Then, we have for $v\in
  H_{\mathcal{T}}^{3/2+\tilde \varepsilon}\left( \Omega\right) $,
  $\tilde \varepsilon>0$,%
  \begin{align}
    &a_{\mathcal{T}}\left( u-u_{\mathcal{T}},v\right) =\left( f+\Delta
      _{\mathcal{T}}u_{\mathcal{T}}+k^{2}u_{\mathcal{T}},v\right)
    -\left(
      \llbr
      \nabla_{\mathcal{T}}u_{\mathcal{T}}%
      \rrbr
      _{N},\left\{ v\right\} \right) _{\mathfrak{S}^{I}}+\left(
      \llbr
      u_{\mathcal{T}}%
      \rrbr
      _{N},\left\{ \nabla_{\mathcal{T}}v\right\} \right)
    _{\mathfrak{S}^{I}%
    }\nonumber\\
    &\qquad~ +\left( \left(
        1-{\mathfrak{d}\frac{k\mathfrak{h}}{\mathfrak{p}}}\right)
      \left( g-\partial_{\mathbf{n}%
        }u_{\mathcal{T}}-\operatorname*{i}ku_{\mathcal{T}}\right)
      ,v\right) _{\partial\Omega}-\left(
      {\frac{\mathfrak{d}\mathfrak{h}}{\operatorname*{i}\mathfrak{p}}}\left(
        g-\partial_{\mathbf{n}}u_{\mathcal{T}}-\operatorname*{i}ku_{\mathcal{T}%
        }\right)  ,\partial_{\mathbf{n}}v\right)  _{\partial\Omega}\nonumber\\
    &\qquad~ -\left(
      \operatorname*{i}{\mathfrak{a}\frac{\mathfrak{p}^2}{\mathfrak{h}}}%
      \llbr
      u_{\mathcal{T}}%
      \rrbr
      _{N},%
      \llbr
      v%
      \rrbr
      _{N}\right) _{\mathfrak{S}^{I}}+\left(
      {\frac{\mathfrak{b}\mathfrak{h}}{\operatorname*{i}\mathfrak{p}}}%
      \llbr
      \nabla_{\mathcal{T}}u_{\mathcal{T}}%
      \rrbr
      _{N},%
      \llbr
      \nabla_{\mathcal{T}}v%
      \rrbr
      _{N}\right) _{\mathfrak{S}^{I}}. \label{4.1.1}%
  \end{align}
\end{lemma}
\begin{proof}%
  Note that $-\Delta u-k^{2}u=f$ in $\Omega$. Integrating by parts we
  obtain with the \textquotedblleft$\operatorname{dG}$-magic
  formula\textquotedblright%
  \begin{align*}
    & \left( \nabla_{\mathcal{T}}\left( u-u_{\mathcal{T}}\right)
      ,\nabla_{\mathcal{T}}v\right) -\left( k^{2}\left(
        u-u_{\mathcal{T}}\right) ,v\right) =\left(
      f+\Delta_{\mathcal{T}}u_{\mathcal{T}}+k^{2}u_{\mathcal{T}%
      },v\right) \\
    & ~\quad+\left( \nabla_{\mathcal{T}}\left(
        u-u_{\mathcal{T}}\right) \cdot\mathbf{n},v\right)
    _{\partial\Omega}+\left(
      \llbr
      \nabla_{\mathcal{T}}\left( u-u_{\mathcal{T}}\right)
      \rrbr
      _{N},\left\{ v\right\} \right) _{\mathfrak{S}^{I}}+\left(
      \left\{ \nabla_{\mathcal{T}}\left( u-u_{\mathcal{T}}\right)
      \right\} ,%
      \llbr
      v%
      \rrbr
      _{N}\right) _{\mathfrak{S}^{I}}.
  \end{align*}

  By inserting this into (\ref{dgsesquilina}) and using
  $\partial_{\mathbf{n}%
  }\left( u-u_{\mathcal{T}}\right) +\operatorname*{i}k\left(
    u-u_{\mathcal{T}%
    }\right)
  =g-\partial_{\mathbf{n}}u_{\mathcal{T}}-\operatorname*{i}%
  ku_{\mathcal{T}}$ on $\partial\Omega$ we get%
  \begin{align*}
    & a_{\mathcal{T}}\left( u-u_{\mathcal{T}},v\right) =\left(
      f+\Delta
      _{\mathcal{T}}u_{\mathcal{T}}+k^{2}u_{\mathcal{T}},v\right) \\
    & ~+\left(
      \llbr
      \nabla_{\mathcal{T}}\left( u-u_{\mathcal{T}}\right)
      \rrbr
      _{N},\left\{ v\right\} \right) _{\mathfrak{S}^{I}}-\left(
      \llbr
      u-u_{\mathcal{T}}%
      \rrbr
      _{N},\left\{  \nabla_{\mathcal{T}}v\right\}  \right)  _{\mathfrak{S}^{I}}\\
    & ~-\left(
      {\frac{\mathfrak{b}\mathfrak{h}}{\operatorname*{i}\mathfrak{p}}}%
      \llbr
      \nabla_{\mathcal{T}}\left( u-u_{\mathcal{T}}\right)
      \rrbr
      _{N},%
      \llbr
      \nabla_{\mathcal{T}}v%
      \rrbr
      _{N}\right) _{\mathfrak{S}^{I}}+\left(
      \operatorname*{i}{\mathfrak{a}%
        \frac{\mathfrak{p}^{2}}{\mathfrak{h}}}%
      \llbr
      u-u_{\mathcal{T}}%
      \rrbr
      _{N},%
      \llbr
      v%
      \rrbr
      _{N}\right)  _{\mathfrak{S}^{I}}\\
    & ~+\left( \left(
        1-{\mathfrak{d}\frac{k\mathfrak{h}}{\mathfrak{p}}%
        }\right) \left(
        g-\partial_{\mathbf{n}}u_{\mathcal{T}}-\operatorname*{i}%
        ku_{\mathcal{T}}\right) \mathbf{,}v\right)
    _{\partial\Omega}-\left(
      {\frac{\mathfrak{d}\mathfrak{h}}{\operatorname*{i}\mathfrak{p}}}\left(
        g-\partial_{\mathbf{n}}u_{\mathcal{T}}-\operatorname*{i}ku_{\mathcal{T}%
        }\right) ,\nabla_{\mathcal{T}}v\cdot\mathbf{n}\right)
    _{\partial\Omega}.
  \end{align*}
  The regularity of the solution $u\in H^{3/2+\varepsilon}\left(
    \Omega\right) $ for some $\varepsilon>0$ implies that all internal
  jumps of $u$ vanish and (\ref{4.1.1}) follows.%
  \qquad\end{proof}%
\medskip%

\begin{proof}%
  (\textit{Lemma \ref{LemReliability}}). We first assume
  $p_{\mathcal{T}}\geq 5$.\newline\textbf{Part 1. }We introduce the
  sesquilinear form $\tilde
  {a}_{\mathcal{T}}:H_{\mathcal{T}}^{1}\left( \Omega\right) \times
  H_{\mathcal{T}}^{1}\left( \Omega\right) \rightarrow\mathbb{C}$ by%
  \[
  \tilde{a}_{\mathcal{T}}\left( v,w\right) :=\left(
    \nabla_{\mathcal{T}%
    }v,\nabla_{\mathcal{T}}w\right) +\left( k^{2}v,w\right)
  +\operatorname*{i}%
  \left( kv,w\right) _{\partial\Omega}%
  \]
  and the associated norm%
  \[
  \left\Vert v\right\Vert _{\tilde{a}}:=\sqrt{\left\vert
      \tilde{a}_{\mathcal{T}%
      }\left( v,v\right) \right\vert }.
  \]
  In Part 2, we will prove%
  \begin{equation}
    \left\Vert u-u_{\mathcal{T}}\right\Vert _{\tilde{a}}\leq
    CC_{\operatorname{conf}}\left(  \eta_{R}^{2}+\eta_{J}^{2}+\eta_{E}^{2}\right)
    ^{1/2}+2\left\Vert k\left(  u-u_{\mathcal{T}}\right)  \right\Vert .
    \label{apost6}%
  \end{equation}
  The combination of%
  \[
  \frac{1}{2}\left( \left\Vert \nabla_{\mathcal{T}}v\right\Vert
    ^{2}+\left\Vert kv\right\Vert ^{2}+\left\Vert k^{1/2}v\right\Vert
    _{\partial\Omega}%
    ^{2}\right) \leq\left\Vert v\right\Vert
  _{\tilde{a}}^{2}\leq\left\Vert \nabla_{\mathcal{T}}v\right\Vert
  ^{2}+\left\Vert kv\right\Vert ^{2}+\left\Vert k^{1/2}v\right\Vert
  _{\partial\Omega}^{2}%
  \]
  with the definition of the $\operatorname{dG}$-norm leads to%
  \begin{align}
    \label{apost4}\left\Vert u-u_{\mathcal{T}}\right\Vert
    _{\operatorname{dG}}^{2} \leq & 2\left\Vert
      u-u_{\mathcal{T}}\right\Vert _{\tilde{a}}^{2} +\left\Vert
      \sqrt{{\mathfrak{d}\frac{\mathfrak{h}}{\mathfrak{p}}}}\nabla_{\mathcal{T}%
      }\left( u-u_{\mathcal{T}}\right) \cdot\mathbf{n}\right\Vert
    _{\partial
      \Omega}^{2}\nonumber\\
    & +\left\Vert
      \sqrt{{\mathfrak{b}\frac{\mathfrak{h}}{\mathfrak{p}}}}%
      \llbr
      \nabla_{\mathcal{T}}\left( u-u_{\mathcal{T}}\right)
      \rrbr
      _{N}\right\Vert _{\mathfrak{S}^{I}}^{2}+\left\Vert
      \sqrt{{\mathfrak{a}%
          \frac{\mathfrak{p}^{2}}{\mathfrak{h}}}}%
      \llbr
      \left( u-u_{\mathcal{T}}\right)
      \rrbr
      _{N}\right\Vert _{\mathfrak{S}^{I}}^{2}.
  \end{align}
  To estimate the boundary term in (\ref{apost4}), we employ $\partial
  _{\mathbf{n}}u=g-\operatorname*{i}ku$ so that for
  $e\in\mathcal{E}^{B}$ it holds%
  \begin{align*}
    \left\Vert \sqrt{{\mathfrak{d}\frac{\mathfrak{h}}{\mathfrak{p}}}}%
      \nabla_{\mathcal{T}}\left( u-u_{\mathcal{T}}\right)
      \cdot\mathbf{n}%
    \right\Vert _{e} &
    =\sqrt{\mathfrak{d}\frac{h_{e}}{p_{e}}}\left\Vert
      g-\partial_{\mathbf{n}}u_{\mathcal{T}}-\operatorname*{i}ku\right\Vert _{e}\\
    & \leq\sqrt{\mathfrak{d}\frac{h_{e}}{p_{e}}}\left\Vert g-\partial
      _{\mathbf{n}}u_{\mathcal{T}}-\operatorname*{i}ku_{\mathcal{T}}\right\Vert
    _{e}+\sqrt{\mathfrak{d}M_{\frac{\operatorname*{kh}}{\operatorname*{p}}}%
    }\left\Vert k^{1/2}\left( u-u_{\mathcal{T}}\right) \right\Vert
    _{e}.
  \end{align*}
  A summation over all $e\in\mathcal{E}^{B}$ leads to%
  \begin{align}
    \label{apostdelta}\left\Vert
      \sqrt{{\mathfrak{d}\frac{\mathfrak{h}%
          }{\mathfrak{p}}}}\nabla_{\mathcal{T}}\left(
        u-u_{\mathcal{T}}\right) \cdot\mathbf{n}\right\Vert
    _{\partial\Omega}^{2} \leq & 2\mathfrak{d}%
    \Bigg( \left\Vert \sqrt{\frac{\mathfrak{h}}{\mathfrak{p}}}\left(
        g-\partial_{\mathbf{n}}u_{\mathcal{T}}-\operatorname*{i}ku_{\mathcal{T}%
        }\right)  \right\Vert _{\partial\Omega}^{2}\nonumber\\
    & +M_{\frac{\operatorname*{kh}}{\operatorname*{p}}}\left\Vert
      k^{1/2}\left(
        u-u_{\mathcal{T}}\right)  \right\Vert _{\partial\Omega}^{2}\Bigg)\nonumber\\
    \leq &
    \frac{2\mathfrak{d}}{p_{\mathcal{T}}}\eta_{E}^{2}+2\mathfrak{d}%
    M_{\frac{\operatorname*{kh}}{\operatorname*{p}}}\left\Vert
      u-u_{\mathcal{T}%
      }\right\Vert _{\tilde{a}}^{2}.
  \end{align}
  For the inner jump terms in (\ref{apost4}) we obtain%
  \begin{subequations}
    \label{apost5}
  \end{subequations}%
  \begin{align}
    \left\Vert \sqrt{{\mathfrak{b}\frac{\mathfrak{h}}{\mathfrak{p}}}}%
      \llbr
      \nabla_{\mathcal{T}}\left( u-u_{\mathcal{T}}\right)
      \rrbr
      _{N}\right\Vert _{\mathfrak{S}^{I}} & =\left\Vert
      \sqrt{{\mathfrak{b}%
          \frac{\mathfrak{h}}{\mathfrak{p}}}}%
      \llbr
      \nabla_{\mathcal{T}}u_{\mathcal{T}}%
      \rrbr
      _{N}\right\Vert _{\mathfrak{S}^{I}}\leq\sqrt{2}\eta_{E}\tag{%
      \ref{apost5}%
      a}\label{apost5a}\\
    \left\Vert
      \sqrt{{\mathfrak{a}\frac{\mathfrak{p}^{2}}{\mathfrak{h}}}}%
      \llbr
      u-u_{\mathcal{T}}%
      \rrbr
      _{N}\right\Vert _{\mathfrak{S}^{I}} & =\left\Vert
      \sqrt{{\mathfrak{a}%
          \frac{\mathfrak{p}^{2}}{\mathfrak{h}}}}%
      \llbr
      u_{\mathcal{T}}%
      \rrbr
    \right\Vert _{\mathfrak{S}^{I}}\leq\sqrt{2}\eta_{J}, \tag{%
      \ref{apost5}%
      b}\label{apost5b}%
  \end{align}
  since the regularity assumptions on $u$ imply that the corresponding
  jump terms vanish.

  The combination of (\ref{apost4}), (\ref{apost6}),
  (\ref{apostdelta}), (\ref{apost5}) yields%
  \begin{align*}
    \left\Vert u-u_{\mathcal{T}}\right\Vert _{\operatorname{dG}}^{2} &
    \leq\left(
      2+2\mathfrak{d}M_{\frac{\operatorname*{kh}}{\operatorname*{p}}%
      }\right) \left\Vert u-u_{\mathcal{T}}\right\Vert
    _{\tilde{a}}^{2}+\left(
      2+\frac{2\mathfrak{d}}{p_{\mathcal{T}}}\right)  \eta_{E}^{2}+2\eta_{J}^{2}\\
    & \leq\left(
      2+2\mathfrak{d}M_{\frac{\operatorname*{kh}}{\operatorname*{p}}%
      }\right) \left( 2C^{2}C_{\operatorname{conf}}^{2}\left(
        \eta_{R}^{2}%
        +\eta_{J}^{2}+\eta_{E}^{2}\right) +4\left\Vert k\left(
          u-u_{\mathcal{T}%
          }\right)  \right\Vert ^{2}\right) \\
    & \quad+\left( 2+\frac{2\mathfrak{d}}{p_{\mathcal{T}}}\right)
    \eta_{E}%
    ^{2}+2\eta_{J}^{2}\\
    & \leq C\left( C_{\operatorname{conf}}^{3}\left(
        \eta_{R}^{2}+\eta_{J}%
        ^{2}+\eta_{E}^{2}\right) +C_{\operatorname{conf}}\left\Vert
        k\left( u-u_{\mathcal{T}}\right) \right\Vert ^{2}\right)
  \end{align*}
  and the assertion follows.\medskip

  \textbf{Part 2.} We will prove (\ref{apost6}). Integration by parts
  leads to%
  \begin{align*}
    \tilde{a}_{\mathcal{T}}\left( v,w\right) = & \left(
      \nabla_{\mathcal{T}%
      }v,\nabla_{\mathcal{T}}w\right) +\left( k^{2}v,w\right)
    +\operatorname*{i}%
    \left(  kv,w\right)  _{\partial\Omega}\\
    = & \left( \left( -\Delta_{\mathcal{T}}+k^{2}\right) v,w\right)
    +\sum_{K\in\mathcal{T}}\left( \partial_{\mathbf{n}_{K}}v,w\right)
    _{\partial
      K}+\operatorname*{i}\left(  kv,w\right)  _{\partial\Omega}\\
    = & \left( \left( -\Delta_{\mathcal{T}}-k^{2}\right) v,w\right)
    +2\left( k^{2}v,w\right) +\left( \left( \partial_{\mathbf{n}%
        }+\operatorname*{i}k\right)  v,w\right)  _{\partial\Omega}\\
    & +\left(
      \llbr
      \nabla_{\mathcal{T}}v%
      \rrbr
      _{N},\left\{ w\right\} \right) _{\mathfrak{S}^{I}}+\left(
      \left\{ \nabla_{\mathcal{T}}v\right\} ,%
      \llbr
      w%
      \rrbr
      _{N}\right) _{\mathfrak{S}^{I}}.
  \end{align*}
  Since $u$ is a solution of (\ref{He}) it holds
  \[
  \left( -\Delta_{\mathcal{T}}-k^{2}\right) \left(
    u-u_{\mathcal{T}}\right) =\left( \Delta_{\mathcal{T}}+k^{2}\right)
  u_{\mathcal{T}}+f ~~\text{and}%
  ~~\left( \partial_{\mathbf{n}}+\operatorname*{i}k\right) \left(
    u-u_{\mathcal{T}}\right)
  =g-\left( \partial_{\mathbf{n}}+\operatorname*{i}%
    k\right) u_{\mathcal{T}}.
  \]
  For test functions $\varphi\in H^{1}\left( \Omega\right) $ we have
  $%
  \llbr
  \varphi%
  \rrbr
  =0$, $\left\{ \varphi\right\} =\varphi$ and $u\in
  H^{3/2+\varepsilon}\left( \Omega\right) $ implies $%
  \llbr
  u%
  \rrbr
  =%
  \llbr
  \nabla_{\mathcal{T}}u%
  \rrbr
  =0$ on interior edges. Therefore%
  \begin{align}
    \tilde{a}_{\mathcal{T}}\left( u-u_{\mathcal{T}},\varphi\right) = &
    \left( \left( \Delta_{\mathcal{T}}+k^{2}\right)
      u_{\mathcal{T}}+f,\varphi\right) -\left(
      \llbr
      \nabla_{\mathcal{T}}u_{\mathcal{T}}%
      \rrbr
      _{N},\varphi\right)  _{\mathfrak{S}^{I}}\nonumber\\
    & +\left( g-\left( \partial_{\mathbf{n}}+\operatorname*{i}k\right)
      u_{\mathcal{T}},\varphi\right) _{\partial\Omega}+2\left(
      k^{2}\left( u-u_{\mathcal{T}}\right) ,\varphi\right)
    . \label{4.1.8}%
  \end{align}
  We choose $u_{\mathcal{T}}^{\ast}\in
  S_{\mathcal{T}}^{\mathfrak{p}}\cap C^{1}\left( \Omega\right) $ as
  the conforming approximant of $u_{\mathcal{T}%
  } $ as in Corollary \ref{CorConfApprox} to obtain%
  \begin{equation}
    \left\Vert u-u_{\mathcal{T}}\right\Vert _{\tilde{a}}\leq\left\Vert
      u-u_{\mathcal{T}}^{\ast}\right\Vert _{\tilde{a}}+\frac{C}{\mathfrak{a}%
    }C_{\operatorname{conf}}\left\Vert \sqrt{{\mathfrak{a}\frac{\mathfrak{p}^{2}%
          }{\mathfrak{h}}}}%
      \llbr
      u_{\mathcal{T}}%
      \rrbr
    \right\Vert _{\mathfrak{S}^{I}}. \label{4.1.9}%
  \end{equation}
  To estimate the first term in (\ref{4.1.9}) we define the set%
  \[
  \Phi:=\left\{ \varphi\in H^{1}\left( \Omega\right) \cap
    H^{3/2+\varepsilon }_{\mathcal{T}}\left( \Omega\right) :\left\Vert
      \varphi\right\Vert _{\tilde{a}}\leq1\right\} .
  \]
  Let $I_{1}^{\operatorname{hp}}:H^{1}\left( \Omega\right) \rightarrow
  S_{\mathcal{T}}^{\mathfrak{p}}\cap C^{1}\left( \Omega\right) $ be
  the interpolation operator as in Theorem \ref{TheoClQI}. Then,
  $\left( u-u_{\mathcal{T}}^{\ast}\right) /\left\Vert
    u-u_{\mathcal{T}}^{\ast }\right\Vert _{\tilde{a}}\in\Phi$ and we
  obtain again with Corollary \ref{CorConfApprox}%
  \begin{align}
    \left\Vert u-u_{\mathcal{T}}^{\ast}\right\Vert _{\tilde{a}} & \leq
    \sup_{\varphi\in\Phi}\left\vert \tilde{a}_{\mathcal{T}}\left(
        u-u_{\mathcal{T}}^{\ast},\varphi\right)  \right\vert \nonumber\\
    & \leq\sup_{\varphi\in\Phi}\left\vert
      \tilde{a}_{\mathcal{T}}\left( u-u_{\mathcal{T}},\varphi\right)
    \right\vert +\sup_{\varphi\in\Phi}\left\vert
      \tilde{a}_{\mathcal{T}}\left(
        u_{\mathcal{T}}-u_{\mathcal{T}}^{\ast}%
        ,\varphi\right)  \right\vert \nonumber\\
    & \leq\sup_{\varphi\in\Phi}\left\vert
      \tilde{a}_{\mathcal{T}}\left( u-u_{\mathcal{T}},\varphi\right)
    \right\vert +\sup_{\varphi\in\Phi}\left\Vert
      u_{\mathcal{T}}-u_{\mathcal{T}}^{\ast}\right\Vert
    _{\tilde{a}}\left\Vert
      \varphi\right\Vert _{\tilde{a}}\nonumber\\
    & \leq\sup_{\varphi\in\Phi}\Bigg\vert
    \tilde{a}_{\mathcal{T}}\left( u-u_{\mathcal{T}},\varphi\right)
    -\underset{=0}{\underbrace{a_{\mathcal{T}%
        }\left(
          u-u_{\mathcal{T}},I_{1}^{\operatorname{hp}}\varphi\right) }%
    }\Bigg\vert
    +\frac{C}{\mathfrak{a}}C_{\operatorname{conf}}\left\Vert
      \sqrt{{\mathfrak{a}\frac{\mathfrak{p}^{2}}{\mathfrak{h}}}}%
      \llbr
      u_{\mathcal{T}}%
      \rrbr
    \right\Vert _{\mathfrak{S}^{I}}. \label{4.1.10}%
  \end{align}
  Next, we use the representations (\ref{4.1.8}) of
  $\tilde{a}_{\mathcal{T}%
  }\left( u-u_{\mathcal{T}},\varphi\right) $ and (\ref{4.1.1}) of
  $a_{\mathcal{T}}\left( u-u_{\mathcal{T}},I_{1}^{\operatorname{hp}}%
    \varphi\right) $ to derive the following expression for the
  supremum in (\ref{4.1.10})%
  \begin{align}
    & \tilde{a}_{\mathcal{T}}\left( u-u_{\mathcal{T}},\varphi\right)
    -a_{\mathcal{T}}\left(
      u-u_{\mathcal{T}},I_{1}^{\operatorname{hp}}%
      \varphi\right) =\left( \left( \Delta_{\mathcal{T}}+k^{2}\right)
      u_{\mathcal{T}}+f,\varphi\right) -\left(
      \llbr
      \nabla_{\mathcal{T}}u_{\mathcal{T}}%
      \rrbr
      _{N},\varphi\right)  _{\mathfrak{S}^{I}}\nonumber\\
    & \quad+\left(
      g-\left( \partial_{\mathbf{n}}+\operatorname*{i}k\right)
      u_{\mathcal{T}},\varphi\right) _{\partial\Omega}+2\left(
      k^{2}\left(
        u-u_{\mathcal{T}}\right)  ,\varphi\right) \nonumber\\
    & \quad-\left(
      f+\Delta_{\mathcal{T}}u_{\mathcal{T}}+k^{2}u_{\mathcal{T}%
      },I_{1}^{\operatorname{hp}}\varphi\right) +\left(
      \llbr
      \nabla_{\mathcal{T}}u_{\mathcal{T}}%
      \rrbr
      _{N},\left\{ I_{1}^{\operatorname{hp}}\varphi\right\} \right)
    _{\mathfrak{S}^{I}}-\left(
      \llbr
      u_{\mathcal{T}}%
      \rrbr
      _{N},\left\{
        \nabla_{\mathcal{T}}I_{1}^{\operatorname{hp}}\varphi\right\}
    \right)  _{\mathfrak{S}^{I}}\nonumber\\
    & \quad-\left( \left(
        1-{\mathfrak{d}\frac{k\mathfrak{h}}{\mathfrak{p}}%
        }\right) \left(
        g-\partial_{\mathbf{n}}u_{\mathcal{T}}-\operatorname*{i}%
        ku_{\mathcal{T}}\right)
      ,I_{1}^{\operatorname{hp}}\varphi\right)
    _{\partial\Omega}+\left(
      {\frac{\mathfrak{d}\mathfrak{h}}{\operatorname*{i}%
          \mathfrak{p}}}\left(
        g-\partial_{\mathbf{n}}u_{\mathcal{T}}-\operatorname*{i}%
        ku_{\mathcal{T}}\right)
      ,\partial_{\mathbf{n}}I_{1}^{\operatorname{hp}%
      }\varphi\right)  _{\partial\Omega}\nonumber\\
    & \quad+\left(
      \operatorname*{i}{\mathfrak{a}\frac{\mathfrak{p}^{2}%
        }{\mathfrak{h}}}%
      \llbr
      u_{\mathcal{T}}%
      \rrbr
      _{N},%
      \llbr
      I_{1}^{\operatorname{hp}}\varphi%
      \rrbr
      _{N}\right) _{\mathfrak{S}^{I}}-\left(
      {\frac{\mathfrak{b}\mathfrak{h}%
        }{\operatorname*{i}\mathfrak{p}}}%
      \llbr
      \nabla_{\mathcal{T}}u_{\mathcal{T}}%
      \rrbr
      _{N},%
      \llbr
      \nabla_{\mathcal{T}}I_{1}^{\operatorname{hp}}\varphi%
      \rrbr
      _{N}\right)  _{\mathfrak{S}^{I}}\nonumber\\
    & =\left( \left( \Delta_{\mathcal{T}}+k^{2}\right) u_{\mathcal{T}%
      }+f,\varphi-I_{1}^{\operatorname{hp}}\varphi\right) +2\left(
      k^{2}\left( u-u_{\mathcal{T}}\right) ,\varphi\right) -\left(
      \llbr
      \nabla_{\mathcal{T}}u_{\mathcal{T}}%
      \rrbr
      _{N},\varphi-I_{1}^{\operatorname{hp}}\varphi\right)
    _{\mathfrak{S}^{I}%
    }\nonumber\\
    & \quad+\left(
      g-\left( \partial_{\mathbf{n}}+\operatorname*{i}k\right)
      u_{\mathcal{T}},\varphi-I_{1}^{\operatorname{hp}}\varphi\right)
    _{\partial\Omega}+\left(
      {\mathfrak{d}\frac{k\mathfrak{h}}{\mathfrak{p}}%
      }\left(
        g-\partial_{\mathbf{n}}u_{\mathcal{T}}-\operatorname*{i}%
        ku_{\mathcal{T}}\right)
      ,I_{1}^{\operatorname{hp}}\varphi\right)
    _{\partial\Omega}\nonumber\\
    & \quad+\left( {\frac{\mathfrak{d}\mathfrak{h}}{\operatorname*{i}%
          \mathfrak{p}}}\left(
        g-\partial_{\mathbf{n}}u_{\mathcal{T}}-\operatorname*{i}%
        ku_{\mathcal{T}}\right)
      ,\partial_{\mathbf{n}}I_{1}^{\operatorname{hp}%
      }\varphi\right) _{\partial\Omega}-\left(
      \llbr
      u_{\mathcal{T}}%
      \rrbr
      _{N},\nabla_{\mathcal{T}}I_{1}^{\operatorname{hp}}\varphi\right)
    _{\mathfrak{S}^{I}}. \label{equalsign}%
  \end{align}
  We denote the terms after the equal sign in (\ref{equalsign}) by
  $T_{1}%
  ,\ldots,T_{7}$ and separately estimate them in the sequel. The
  constants $C$ only depend on $\mathfrak{b}$, $\mathfrak{d}$ in
  (\ref{dgsesquilin}), the shape regularity of the mesh, and the
  constant $C$ in (\ref{C1Clement}).

  @$T_{1}:$%
  \[
  \left\vert \left( \left( \Delta_{\mathcal{T}}+k^{2}\right)
      u_{\mathcal{T}%
      }+f,\varphi-I_{1}^{\operatorname{hp}}\varphi\right) \right\vert
  \overset{\text{(\ref{C1Clementa})}}{\leq}C\left\Vert
    \frac{\mathfrak{h}%
    }{\mathfrak{p}}\left( \Delta_{\mathcal{T}}u_{\mathcal{T}}+k^{2}%
      u_{\mathcal{T}}+f\right) \right\Vert \left\Vert
    \nabla\varphi\right\Vert .
  \]

  @$T_{3}:$
  \begin{align}
    \left\vert \left(
        \llbr
        \nabla_{\mathcal{T}}u_{\mathcal{T}}%
        \rrbr
        _{N},\varphi-I_{1}^{\operatorname{hp}}\varphi\right)
      _{\mathfrak{S}^{I}%
      }\right\vert & \leq\left\Vert
      \sqrt{{\mathfrak{b}\frac{\mathfrak{h}%
          }{\mathfrak{p}}}}%
      \llbr
      \nabla_{\mathcal{T}}u_{\mathcal{T}}%
      \rrbr
      _{N}\right\Vert _{\mathfrak{S}^{I}}\left(
      \sum_{e\in\mathcal{E}^{I}}%
      \frac{p_{e}}{\mathfrak{b}h_{e}}\left\Vert \left( \varphi-I_{1}%
          ^{\operatorname{hp}}\varphi\right) \right\Vert
      _{e}^{2}\right)
    ^{1/2}\label{tracephiIhpphi}\\
    & \overset{\text{(\ref{C1Clementb})}}{\leq}C\left\Vert
      \sqrt{{\mathfrak{b}%
          \frac{\mathfrak{h}}{\mathfrak{p}}}}%
      \llbr
      \nabla_{\mathcal{T}}u_{\mathcal{T}}%
      \rrbr
      _{N}\right\Vert _{\mathfrak{S}^{I}}\left\Vert
      \nabla\varphi\right\Vert .\nonumber
  \end{align}

  @$T_{7}:$ Using $H^{1}$-stability of $I_{1}^{\operatorname*{hp}}$,
  we obtain
  \begin{align}
    \left\vert \left(
        \llbr
        u_{\mathcal{T}}%
        \rrbr
        _{N},\nabla_{\mathcal{T}}I_{1}^{\operatorname{hp}}\varphi\right)
      _{\mathfrak{S}^{I}}\right\vert &
    \leq\frac{1}{\sqrt{\mathfrak{a}}}\left\Vert
      \sqrt{{\mathfrak{a}\frac{\mathfrak{p}^{2}}{\mathfrak{h}}}}%
      \llbr
      u_{\mathcal{T}}%
      \rrbr
      _{N}\right\Vert _{\mathfrak{S}^{I}}\left\Vert
      \frac{\sqrt{\mathfrak{h}}%
      }{\mathfrak{p}}\nabla_{\mathcal{T}}I_{1}^{\operatorname{hp}}\varphi\right\Vert
    _{\mathfrak{S}^{I}}\nonumber\\
    & \leq C_{1}\left\Vert \sqrt{{\mathfrak{a}\frac{\mathfrak{p}^{2}%
          }{\mathfrak{h}}}}%
      \llbr
      u_{\mathcal{T}}%
      \rrbr
      _{N}\right\Vert _{\mathfrak{S}^{I}}\left\Vert
      \nabla_{\mathcal{T}}%
      I_{1}^{\operatorname{hp}}\varphi\right\Vert \label{edgeinvest}\\
    & \leq C_{1}C_{2}\left\Vert
      \sqrt{{\mathfrak{a}\frac{\mathfrak{p}^{2}%
          }{\mathfrak{h}}}}%
      \llbr
      u_{\mathcal{T}}%
      \rrbr
      _{N}\right\Vert _{\mathfrak{S}^{I}}\left\Vert
      \nabla\varphi\right\Vert ,\nonumber
  \end{align}
  where $C_{1}$ depends on the constant in an $hp$-explicit inverse
  estimate for polynomials (see \cite[Thm. 4.76]{SchwabhpBook}).

  @$T_{4}$:%
  \[
  \left\vert \left(
      g-\left( \partial_{\mathbf{n}}+\operatorname*{i}k\right)
      u_{\mathcal{T}},\varphi-I_{1}^{\operatorname{hp}}\varphi\right)
    _{\partial\Omega}\right\vert
  \overset{\text{(\ref{C1Clementb})}}{\leq }C\left\Vert
    \sqrt{\frac{\mathfrak{h}}{\mathfrak{p}}}\left( g-\left(
        \partial_{\mathbf{n}}+\operatorname*{i}k\right)
      u_{\mathcal{T}}\right) \right\Vert _{\partial\Omega}\left\Vert
    \nabla\varphi\right\Vert .
  \]

  @$T_{5}$: We use $\mathfrak{d}k^{1/2}h_{e}/p_{e}\leq\mathfrak{d}%
  M_{\frac{\operatorname*{kh}}{\operatorname*{p}}}^{1/2}\left( h_{e}%
    /p_{e}\right) ^{1/2}$ and obtain%
  \begin{align*}
    \left\Vert {\mathfrak{d}\frac{k\mathfrak{h}}{\mathfrak{p}}} I_{1}%
      ^{\operatorname{hp}}\varphi\right\Vert _{e} & \leq\left\Vert
      {\mathfrak{d}%
        \frac{k\mathfrak{h}}{\mathfrak{p}}}\varphi\right\Vert
    _{e}+\left\Vert
      {\mathfrak{d}\frac{k\mathfrak{h}}{\mathfrak{p}}}\left(
        \varphi-I_{1}%
        ^{\operatorname{hp}}\varphi\right)  \right\Vert _{e}\\
    & \overset{\text{(\ref{C1Clementb})}}{\leq}C\mathfrak{d}\left(
      \frac{h_{e}%
      }{p_{e}}\right) ^{1/2}\left(
      M_{\frac{\operatorname*{kh}}{\operatorname*{p}%
        }}^{1/2}\left\Vert k^{1/2}\varphi\right\Vert _{L^{2}\left(
          e\right)
      }+M_{\frac{\operatorname*{kh}}{\operatorname*{p}}}\left\Vert
        \nabla \varphi\right\Vert _{L^{2}\left( \omega_{e}\right)
      }\right) .
  \end{align*}
  This leads to
  \begin{align*}
    \left\vert T_{5}\right\vert & \leq2C\mathfrak{d}\left( 1+M_{\frac
        {\operatorname*{kh}}{\operatorname*{p}}}\right) \left\Vert
      \sqrt {\frac{\mathfrak{h}}{\mathfrak{p}}}\left(
        g-\partial_{\mathbf{n}%
        }u_{\mathcal{T}}-\operatorname*{i}ku_{\mathcal{T}}\right)
    \right\Vert _{\partial\Omega}\left( \left\Vert
        k^{1/2}\varphi\right\Vert _{\partial
        \Omega}+\left\Vert \nabla\varphi\right\Vert \right) \\
    & {\leq}\tilde{C}\mathfrak{d}\left( 1+M_{\frac{\operatorname*{kh}%
        }{\operatorname*{p}}}\right) \left\Vert
      \sqrt{\frac{\mathfrak{h}%
        }{\mathfrak{p}}}\left( g-\partial_{\mathbf{n}}u_{\mathcal{T}}%
        -\operatorname*{i}ku_{\mathcal{T}}\right) \right\Vert
    _{\partial\Omega }\left\Vert \varphi\right\Vert _{\tilde a}.
  \end{align*}

  @$T_{6}$: We obtain similarly as in (\ref{edgeinvest})%
  \begin{align*}
    \left\vert \left(
        {\frac{\mathfrak{d}\mathfrak{h}}{\operatorname*{i}%
            \mathfrak{p}}}\left(
          g-\partial_{\mathbf{n}}u_{\mathcal{T}}-\operatorname*{i}%
          ku_{\mathcal{T}}\right)
        ,\partial_{\mathbf{n}}I_{1}^{\operatorname{hp}%
        }\varphi\right) _{\partial\Omega}\right\vert & \leq
    C\mathfrak{d}\left\Vert
      g-\partial_{\mathbf{n}}u_{\mathcal{T}}-\operatorname*{i}ku_{\mathcal{T}%
      }\right\Vert _{\partial\Omega}\left\Vert
      \frac{\mathfrak{h}}{\mathfrak{p}%
      }\partial_{\mathbf{n}}I_{1}^{\operatorname{hp}}\varphi\right\Vert
    _{\partial\Omega}\\
    & \leq C\mathfrak{d}\left\Vert \sqrt{\mathfrak{h}}\left(
        g-\partial
        _{\mathbf{n}}u_{\mathcal{T}}-\operatorname*{i}ku_{\mathcal{T}}\right)
    \right\Vert _{\partial\Omega}\left\Vert \nabla\varphi\right\Vert .
  \end{align*}
  These estimates allow to bound the expression in the supremum of
  (\ref{4.1.10}) by%
  \begin{align*}
    & \left\vert \tilde{a}_{\mathcal{T}}\left(
        u-u_{\mathcal{T}},\varphi\right) -a_{\mathcal{T}}\left(
        u-u_{\mathcal{T}},I_{1}^{\operatorname{hp}}%
        \varphi\right) \right\vert \leq2\left\Vert k\left(
        u-u_{\mathcal{T}}\right)
    \right\Vert \left\Vert k\varphi\right\Vert \\
    & \qquad+C\Bigg( \left\Vert
      \frac{\mathfrak{h}}{\mathfrak{p}}\left(
        \Delta_{\mathcal{T}}u_{\mathcal{T}}+k^{2}u_{\mathcal{T}}+f\right)
    \right\Vert \left\Vert \nabla\varphi\right\Vert +\left\Vert
      \sqrt{{\mathfrak{b}%
          \frac{\mathfrak{h}}{\mathfrak{p}}}}%
      \llbr
      \nabla_{\mathcal{T}}u_{\mathcal{T}}%
      \rrbr
      _{N}\right\Vert _{\mathfrak{S}^{I}}\left\Vert \nabla\varphi\right\Vert \\
    & \qquad+\left\Vert \sqrt{\frac{\mathfrak{h}}{\mathfrak{p}}}\left(
        g-\left(
          \partial_{\mathbf{n}}+\operatorname*{i}k\right)
        u_{\mathcal{T}}\right)
    \right\Vert _{\partial\Omega}\left\Vert \nabla\varphi\right\Vert \\
    & \qquad+\left(
      1+M_{\frac{\operatorname*{kh}}{\operatorname*{p}}}\right)
    \left\Vert \sqrt{\frac{\mathfrak{h}}{\mathfrak{p}}}\left(
        g-\partial
        _{\mathbf{n}}u_{\mathcal{T}}-\operatorname*{i}ku_{\mathcal{T}}\right)
    \right\Vert _{\partial\Omega}\left\Vert \varphi\right\Vert _{\tilde a}\\
    & \qquad+\left\Vert \sqrt{\mathfrak{h}}\left(
        g-\partial_{\mathbf{n}%
        }u_{\mathcal{T}}-\operatorname*{i}ku_{\mathcal{T}}\right)
    \right\Vert _{\partial\Omega}\left\Vert \nabla\varphi\right\Vert
    +\left\Vert
      \sqrt{{\mathfrak{a}\frac{\mathfrak{p}^{2}}{\mathfrak{h}}}}%
      \llbr
      u_{\mathcal{T}}%
      \rrbr
      _{N}\right\Vert _{\mathfrak{S}^{I}}\left\Vert
      \nabla\varphi\right\Vert \Bigg) .
  \end{align*}
  The combination of (\ref{4.1.9}), (\ref{4.1.10}) with the
  definitions of $\eta_{R}$, $\eta_{E}$, $\eta_{J}$ leads to%
  \begin{align*}
    \left\Vert u-u_{\mathcal{T}}\right\Vert _{\tilde{a}} &
    \leq2\left\Vert k\left( u-u_{\mathcal{T}}\right) \right\Vert
    +C\Bigg( \left\Vert \frac{\mathfrak{h}}{\mathfrak{p}}\left(
        \Delta_{\mathcal{T}}u_{\mathcal{T}%
        }+k^{2}u_{\mathcal{T}}+f\right)  \right\Vert \\
    & \quad+C_{\operatorname{conf}}\left\Vert \sqrt{{\mathfrak{a}\frac
          {\mathfrak{p}^{2}}{\mathfrak{h}}}}%
      \llbr
      u_{\mathcal{T}}%
      \rrbr
    \right\Vert _{ \mathfrak{S}^{I} }+\left\Vert \sqrt{{\mathfrak{b}%
          \frac{\mathfrak{h}}{\mathfrak{p}}}}%
      \llbr
      \nabla_{\mathcal{T}}u_{\mathcal{T}}%
      \rrbr
      _{N}\right\Vert _{\mathfrak{S}^{I}}\\
    & \quad+\left(
      1+M_{\frac{\operatorname*{kh}}{\operatorname*{p}}}\right)
    \left\Vert \sqrt{\mathfrak{h}}\left(
        g-\partial_{\mathbf{n}}u_{\mathcal{T}%
        }-\operatorname*{i}ku_{\mathcal{T}}\right) \right\Vert
    _{\partial\Omega
    }\Bigg)\\
    & \leq CC_{\operatorname{conf}}\left(
      \eta_{R}^{2}+\eta_{E}^{2}+\eta_{J}%
      ^{2}\right) ^{1/2}+2k\left\Vert u-u_{\mathcal{T}}\right\Vert .
  \end{align*}
  This concludes the proof for $p_{\mathcal{T}}\geq5$.

  For $1\leq p_{\mathcal{T}}<5$ we have to employ
  $I_{1}^{\operatorname{hp},0}$ instead of $I_{1}^{\operatorname{hp}}$
  (cf. Theorem \ref{TheoClQI}). For the details of this case we refer
  to \cite[Rem. 4.1.4]{ZechMaster}.%
  \qquad\end{proof}%

To prove the reliability estimate it remains to bound the term
$\left\Vert k\left( u-u_{\mathcal{T}}\right) \right\Vert $ by the
estimator. We will show that $\left\Vert k\left(
    u-u_{\mathcal{T}}\right) \right\Vert $ is bounded (modulo
constants) by the product of $\eta\left( u_{\mathcal{T}%
  }\right) $ with the adjoint approximation property
$\sigma_{k}^{\star}\left( S\right) $ (see (\ref{defsigmastar})).

\begin{lemma}
  \label{Lemabsorb}Let the assumptions of Lemma \ref{LemReliability}
  be satisfied. There exists a constant $C$ solely depending on
  $\rho_{\mathcal{T}%
  }$, $\mathfrak{b}$, $\mathfrak{d}$, and $\Omega$ such that%
  \[
  \left\Vert k\left( u-u_{\mathcal{T}}\right) \right\Vert \leq
  C\eta\left( u_{\mathcal{T}}\right) \sigma_{k}^{\ast}\left(
    S_{\mathcal{T}}%
    ^{\mathfrak{p}}\right)
  \]
  with $\sigma_{k}^{\ast}\left( S_{\mathcal{T}}^{\mathfrak{p}}\right)
  $ as in (\ref{defsigmastar}).
\end{lemma}
\begin{proof}%
  \textbf{Part 1. }We will prove%
  \begin{equation}
    \left\vert a_{\mathcal{T}}\left(  u-u_{\mathcal{T}},\varphi\right)
    \right\vert \leq C\eta\left(  u_{\mathcal{T}}\right)  \left\Vert
      \varphi\right\Vert _{\operatorname{dG}^{+}}\qquad\forall\varphi\in
    H^{1}\left(  \Omega\right)  \cap H_{\mathcal{T}}^{3/2+\varepsilon}\left(
      \Omega\right)  . \label{LemabsorbPart1}%
  \end{equation}
  Note that $%
  \llbr
  I_{1}^{\operatorname{hp}}\varphi%
  \rrbr
  _{N}=%
  \llbr
  \nabla I_{1}^{\operatorname{hp}}\varphi%
  \rrbr
  _{N}=0$. We employ Lemma \ref{LemAltRepA} and the estimates for
  $T_{1}%
  ,\ldots,T_{7}$ in the proof of Lemma \ref{LemReliability} to obtain%
  \begin{align}
    &\left\vert a_{\mathcal{T}}\left( u-u_{\mathcal{T}},\varphi\right)
    \right\vert = \left\vert a_{\mathcal{T}}\left( u-u_{\mathcal{T}}%
        ,\varphi-I_{1}^{\operatorname{hp}}\varphi\right)
    \right\vert \label{atabsorb}%
    \\
    &\leq C\Bigg( \left\Vert \frac{\mathfrak{h}}{\mathfrak{p}}\left(
        \Delta_{\mathcal{T}}u_{\mathcal{T}}+k^{2}u_{\mathcal{T}}+f\right)
    \right\Vert \left\Vert \nabla\varphi\right\Vert +\left\Vert
      \sqrt{{\mathfrak{b}\frac{\mathfrak{h}}{\mathfrak{p}}}}%
      \llbr
      \nabla_{\mathcal{T}}u_{\mathcal{T}}%
      \rrbr
      _{N}\right\Vert _{\mathfrak{S}^{I}}\left\Vert
      \sqrt{\frac{\mathfrak{p}%
        }{\mathfrak{b}\mathfrak{h}}} \left(
        \varphi-I_{1}^{\operatorname{hp}}%
        \varphi\right)  \right\Vert _{\mathfrak{S}^{I}}\nonumber\\
    &\quad +\left\Vert
      \sqrt{{\mathfrak{a}\frac{\mathfrak{p}^{2}}{\mathfrak{h}}}}%
      \llbr
      u_{\mathcal{T}}%
      \rrbr
      _{N}\right\Vert _{\mathfrak{S}^{I}}\left\Vert \left(
        {\mathfrak{a}%
          \frac{\mathfrak{p}^{2}}{\mathfrak{h}}}\right) ^{-1/2}\left\{
        \nabla\left( \varphi-I_{1}^{\operatorname{hp}}\varphi\right)
      \right\} \right\Vert
    _{\mathfrak{S}^{I}}\nonumber\\
    &\quad +\left\Vert \sqrt{\mathfrak{h}}\left(
        g-\partial_{\mathbf{n}%
        }u_{\mathcal{T}}-\operatorname*{i}ku_{\mathcal{T}}\right)
    \right\Vert _{\partial\Omega}\left\Vert
      \frac{\sqrt{\mathfrak{h}}}{\mathfrak{p}}%
      \partial_{\mathbf{n}}\left(
        \varphi-I_{1}^{\operatorname{hp}}\varphi\right)
    \right\Vert _{\partial\Omega}\nonumber\\
    & \quad+\left\Vert \sqrt{\frac{\mathfrak{h}}{\mathfrak{p}}}\left(
        g-
        \partial_{\mathbf{n}}u_{\mathcal{T}}-\operatorname*{i}k
        u_{\mathcal{T}}\right) \right\Vert _{\partial\Omega}\left\Vert
      \nabla\varphi\right\Vert +\left\Vert
      \sqrt{{\mathfrak{b}\frac{\mathfrak{h}}{\mathfrak{p}}}} [\![
      \nabla_{\mathcal{T}} u_{\mathcal{T}} ]\!] _{N} \right\Vert
    _{\mathfrak{S}^{I}} \left\Vert
      \sqrt{{\mathfrak{b}\frac{\mathfrak{h}}{\mathfrak{p}}}} [\![
      \nabla_{\mathcal{T}} \varphi]\!] _{N} \right\Vert
    _{\mathfrak{S}^{I}%
    }\Bigg) .\nonumber
  \end{align}
  Note that%
  \begin{align}
    \left\Vert \left(
        {\mathfrak{a}\frac{\mathfrak{p}^{2}}{\mathfrak{h}}}\right)
      ^{-1/2}\left\{ \nabla\left(
          \varphi-I_{1}^{\operatorname{hp}}\varphi\right) \right\}
    \right\Vert _{\mathfrak{S}^{I}} & \leq C \left( \left\Vert \left(
          {\mathfrak{a}\frac{\mathfrak{p}^{2}}{\mathfrak{h}}}\right)
        ^{-1/2}\left\{ \nabla\varphi\right\} \right\Vert
      _{\mathfrak{S}^{I}%
      }+\left\Vert
        \frac{\sqrt{\mathfrak{h}}}{\mathfrak{p}\sqrt{\mathfrak{a}}}\nabla
        I_{1}^{\operatorname{hp}}\varphi\right\Vert
      _{\mathfrak{S}^{I}}\right)
    \label{atabsorbaux1}\\
    & \hspace{-1.5cm} \overset{\text{(\ref{edgeinvest})}}{\leq}C\left(
      \left\Vert \left(
          {\mathfrak{a}\frac{\mathfrak{p}^{2}}{\mathfrak{h}}}\right)
        ^{-1/2}\left\{ \nabla\varphi\right\} \right\Vert
      _{\mathfrak{S}^{I}%
      }+\left\Vert \nabla\varphi\right\Vert \right) \leq C\left\Vert
      \varphi \right\Vert _{\operatorname{dG}^{+}}.\nonumber
  \end{align}
  We also use%
  \begin{align}
    \left\Vert
      \frac{\sqrt{\mathfrak{h}}}{\mathfrak{p}}\partial_{\mathbf{n}%
      }\left( \varphi-I_{1}^{\operatorname{hp}}\varphi\right)
    \right\Vert _{\partial\Omega} & \leq C\left\Vert
      \varphi\right\Vert _{\operatorname{dG}%
    }+\left\Vert
      \frac{\sqrt{\mathfrak{h}}}{\mathfrak{p}}\partial_{\mathbf{n}%
      }I_{1}^{\operatorname{hp}}\varphi\right\Vert _{\partial\Omega}%
    \label{atabsorbaux3}\\
    & \overset{\text{(\ref{edgeinvest})}}{\leq}C\left( \left\Vert
        \varphi \right\Vert _{\operatorname{dG}}+\left\Vert
        \nabla\varphi\right\Vert \right) \leq C\left\Vert
      \varphi\right\Vert _{\operatorname{dG}}.\nonumber
  \end{align}
  From the combination of (\ref{tracephiIhpphi}), (\ref{atabsorb}),
  (\ref{atabsorbaux1}), (\ref{atabsorbaux3}) with the definition of
  the error estimator we conclude that (\ref{LemabsorbPart1})
  holds.\medskip

  \textbf{Part 2.} We will derive the assertion by using
  (\ref{LemabsorbPart1}) and an Aubin-Nitsche argument. For
  $Q_{k}^{\star}$ as defined after (\ref{adjproblem}), let
  $z:=Q_{k}^{\ast}\left( k^{2}\left( u-u_{\mathcal{T}%
      }\right) \right) $. Furthermore let $z_{S}\in
  S_{\mathcal{T}}^{\mathfrak{p}%
  }$ be the best approximation of $z$ in the finite element space with
  respect to the norm $\Vert\cdot\Vert_{\operatorname{dG}^{+}}$, i.e.
  \[
  \left\Vert z-z_{S}\right\Vert _{\operatorname{dG}^{+}}=\inf_{w\in
    S_{\mathcal{T}}^{\mathfrak{p}}}\left\Vert z-w\right\Vert
  _{\operatorname{dG}%
    ^{+}}.
  \]
  With Lemma \ref{LemAdjCons} it follows
  \[
  \left\Vert k\left( u-u_{\mathcal{T}}\right) \right\Vert ^{2}=\left(
    u-u_{\mathcal{T}},k^{2}\left( u-u_{\mathcal{T}}\right) \right)
  =a_{\mathcal{T}}\left( u-u_{\mathcal{T}},z\right)
  =a_{\mathcal{T}}\left( u-u_{\mathcal{T}},z-z_{S}\right) .
  \]
  By using the adjoint approximation property (\ref{defsigmastar}) we
  get%
  \[
  \left\Vert z-z_{S}\right\Vert _{\operatorname{dG}^{+}}=\inf_{w\in
    S_{\mathcal{T}}^{\mathfrak{p}}}\left\Vert Q_{k}^{\ast}\left(
      k^{2}\left( u-u_{\mathcal{T}}\right) \right) -w\right\Vert
  _{\operatorname{dG}^{+}}%
  \leq\sigma_{k}^{\star}\left( S_{\mathcal{T}}^{\mathfrak{p}}\right)
  \left\Vert k\left( u-u_{\mathcal{T}}\right) \right\Vert .
  \]
  Employing (\ref{LemabsorbPart1}) we end up with%
  \begin{align*}
    \left\Vert k\left( u-u_{\mathcal{T}}\right) \right\Vert ^{2} &
    =a_{\mathcal{T}}\left( u-u_{\mathcal{T}},z-z_{S}\right) \leq
    C\eta\left( u_{\mathcal{T}}\right) \left\Vert z-z_{S}\right\Vert
    _{\operatorname{dG}^{+}%
    }\\
    & \leq C\eta\left( u_{\mathcal{T}}\right) \sigma_{k}^{\star}\left(
      S_{\mathcal{T}}^{\mathfrak{p}}\right) \left\Vert k\left(
        u-u_{\mathcal{T}%
        }\right) \right\Vert ,
  \end{align*}
  which implies the assertion.
  \qquad\end{proof}%

The next theorem states the reliability estimate for our a posteriori
error estimator which is explicit in the discretization parameters
$h$, $p$, and the wavenumber $k$. Its proof is a simple combination of
Lemma \ref{LemReliability} and Lemma \ref{Lemabsorb}. For later use we
define a modified error estimator where $f$ and $g$ are replaced by
projections to polynomial spaces and \textit{data oscillations}. In
order to obtain reliability \textit{and} efficiency for the
\textit{same} error estimator (up to data oscillations) we will also
state reliability for the modified error estimator in the following
theorem; the latter follows from the reliability of the original error
estimator $\eta$ (cf. \cite[Thm. 4.1.10]{ZechMaster}) via a triangle
inequality.

\begin{definition}
  For $f\in L^{2}\left( \Omega\right) $, let $f_{\mathcal{T}}$ be the
  simplex-wise polynomial function with $\left.
    f_{\mathcal{T}}\right\vert _{K}$ denoting the $L^{2}\left(
    K\right) $ orthogonal projection of $\left.  f\right\vert _{K}$
  onto $\mathbb{P}_{p_{K}}\left( K\right) $.  For $g\in
  L^{2}\left( \partial^{B}K\right) $, let $g_{\partial^{B}K}\in
  L^{2}\left( \partial^{B}K\right) $ be the edge-wise polynomial
  function with $\left.  g_{\partial^B K}\right\vert _{e}$ denoting
  the $L^{2}\left(%
    e\right) $ orthogonal projection of $\left.  g\right\vert _{e}$
  onto $\mathbb{P}_{p_{K}}\left( e\right) $.  The data oscillations
  are given for $K\in\mathcal{T}$ by%
  \[
  \operatorname*{osc}\nolimits_{K}:=\left( \left\Vert
      \frac{h_{K}}{p_{K}%
      }\left( f-f_{\mathcal{T}}\right) \right\Vert _{L^{2}\left(
        K\right) }%
    ^{2}+\left\Vert \sqrt{\mathfrak{h}}\left(
        g-g_{\partial^{B}K}\right) \right\Vert
    _{\partial^{B}K}^{2}\right) ^{1/2}
  \]
  and
  \[
  \operatorname*{osc}\nolimits_{\mathcal{T}}:=\left(
    \sum_{K\in\mathcal{T}%
    }\operatorname*{osc}\nolimits_{K}^{2}\right) ^{1/2}.
  \]
  The local error estimators $\tilde{\eta}_{K}$,
  $\tilde{\eta}_{R_{K}}$, $\tilde{\eta}_{E_{K}}$ are given by
  replacing $f$ by $f_{\mathcal{T}}$ in (\ref{locerrorestb}), $g$ by
  $g_{\partial^{B}K}$ in (\ref{locerrorestb}), and $\eta_{R_{K}}$ and
  $\eta_{E_{k}}$ by $\tilde{\eta}_{R_{K}}$ and $\tilde{\eta }_{E_{K}}$
  in (\ref{locerroresta}). The global estimators $\tilde{\eta}_{R}$,
  $\tilde{\eta}_{E}$, and $\tilde \eta$ are given by replacing
  $\eta_{R_{K}}$ and $\eta_{E_{K}}$ by $\tilde{\eta}_{R_{K}}$ and
  $\tilde{\eta}_{E_{K}}$ in (\ref{globerrorestb}) and $\eta_{R}$ and
  $\eta_{E}$ by $\tilde{\eta}_{R}$ and $\tilde{\eta}_{E}$ in
  (\ref{globerroresta}).
\end{definition}

\begin{theorem}
  \label{TheoRel}Let $\mathcal{T}$ be a shape regular, conforming
  simplicial finite element mesh of the polygonal Lipschitz domain
  $\Omega\subseteq \mathbb{R}^2$ and let the polynomial degree
  function $\mathfrak{p}$ satisfies (\ref{defrho2}) and
  $p_{\mathcal{T}}\geq 1$. Assume that $k>1$ is constant.  Let $u\in
  H^{3/2+\varepsilon}\left(\Omega\right) $ be the solution of
  (\ref{He}) for some $\varepsilon>0$ and assume that $u_{\mathcal{T}%
  }\in S_{\mathcal{T}}^{\mathfrak{p}}$ solves (\ref{dgsesquilin}) with
  $\mathfrak{a}\geq1$. %
  Then, there exists a constant $C>0$ solely depending on
  $\rho_{\mathcal{T}}$, $\mathfrak{b}$, $\mathfrak{d}$, and $\Omega$
  such that%
  \[
  \left\Vert u-u_{\mathcal{T}}\right\Vert _{\operatorname{dG}}\leq
  C\sqrt{1+M_{\frac{\operatorname*{kh}}{\operatorname*{p}}}}\left(
    1+M_{\frac{\operatorname*{kh}}{\operatorname*{p}}}+\sigma_{k}^{\ast}\left(
      S_{\mathcal{T}}^{\mathfrak{p}}\right) \right) \eta\left(
    u_{\mathcal{T}%
    }\right) .
  \]
  For the modified error estimator it holds%
  \[
  \left\Vert u-u_{\mathcal{T}}\right\Vert _{\operatorname{dG}}\leq
  C\left( 1+M_{\frac{\operatorname*{kh}}{\operatorname*{p}}}\right)
  ^{3/2}\left( 1+\sigma_{k}^{\ast}\left(
      S_{\mathcal{T}}^{\mathfrak{p}}\right) \right) \left(
    \tilde{\eta}\left( u_{\mathcal{T}}\right) +\operatorname*{osc}%
    \nolimits_{\mathcal{T}}\left( u_{\mathcal{T}}\right) \right) .
  \]
\end{theorem}

\subsection{Efficiency}

The reliability estimate in the form of Theorem \ref{TheoRel} shows that the
error estimator (modulo a constant $C$ which only depends on $\rho
_{\mathcal{T}}$, $\mathfrak{b}$, $\mathfrak{d}$, and $\Omega$) controls the
error of the $\operatorname{dG}$-approximation $u_{\mathcal{T}}$ in a reliable
way. This estimate can be used as a stopping criterion within an adaptive
discretization process.

In this section we are concerned with the \textit{efficiency} of the error
estimator which ensures that the error estimator converges with the same rate
as the true error. Efficiency can be proved locally, i.e., the localized error
estimator is estimated by the localized error. For the proof, we employ ideas
which have been developed for conforming finite element methods in
\cite{MelenkWohlmuth} and for $\operatorname{dG}$-methods, e.g., in \cite[Thm.
3.2]{HoustonSchoetzWihl2007}. As is common for efficiency estimates one has to
deal with \textit{data oscillations}.

\begin{theorem}
Let the assumptions of Theorem \ref{TheoRel} be satisfied. There exists a constant
independent of $k$, $h_{K}$, $p_{K}$ such that the modified local internal residual
can be estimated by%
\begin{subequations}
\label{effest}
\end{subequations}%
\begin{equation}
\tilde{\eta}_{R_{K}}\leq Cp_{K}\left(  \left\Vert \nabla\left(
u-u_{\mathcal{T}}\right)  \right\Vert _{L^{2}\left(  K\right)  }%
+M_{\frac{\operatorname*{kh}}{\operatorname*{p}}}\left\Vert k\left(
u-u_{\mathcal{T}}\right)  \right\Vert _{L^{2}\left(  K\right)  }+\left\Vert
\frac{h_{K}}{p_{K}}\left(  f-f_{\mathcal{T}}\right)  \right\Vert
_{L^{2}\left(  K\right)  }\right)  .\tag{%
\ref{effest}%
a}\label{effesta}%
\end{equation}
For the gradient jumps in the error estimator it holds%
\begin{align}
\left\Vert \sqrt{{\mathfrak{b}\frac{\mathfrak{h}}{\mathfrak{p}}}}%
\llbr
\nabla_{\mathcal{T}}u_{\mathcal{T}}%
\rrbr
_{N}\right\Vert _{e}\leq &Cp_{e}^{3/2}\Bigg(  \left\Vert \nabla\left(
u-u_{\mathcal{T}}\right)  \right\Vert _{L^{2}\left(  \omega_{e}\right)
}+M_{\frac{\operatorname*{kh}}{\operatorname*{p}}}\left\Vert k\left(
u-u_{\mathcal{T}}\right)  \right\Vert _{L^{2}\left(  \omega_{e}\right)
}\tag{%
\ref{effest}%
b}\label{effestb}\\
&+\left\Vert \frac{h_{e}}{p_{e}}\left(  f-f_{p_{e}}\right)  \right\Vert
_{L^{2}\left(  \omega_{e}\right)  }\Bigg) \nonumber .
\end{align}
For the modified local edge residuals it holds%
\begin{align}
\tilde{\eta}_{E_{K}} \leq &Cp_{K}^{2}\Bigg(  \left\Vert \nabla\left(
u-u_{\mathcal{T}}\right)  \right\Vert _{L^{2}\left(  \omega_{K}\right)
}+M_{\frac{\operatorname*{kh}}{\operatorname*{p}}}\left\Vert k\left(
u-u_{\mathcal{T}}\right)  \right\Vert _{L^{2}\left(  \omega_{K}\right)
}+\left\Vert \frac{h_{K}}{p_{K}}\left(  f-f_{\mathcal{T}}\right)  \right\Vert
_{L^{2}\left(  \omega_{K}\right)  }  \tag{%
\ref{effest}%
c}\label{effestc}\\
&    +\sqrt{M_{\frac{\operatorname*{kh}}{\operatorname*{p}}}%
}\left\Vert \sqrt{\frac{k}{p_{K}}}\left(  u-u_{\mathcal{T}}\right)
\right\Vert _{\partial^{B}K}+\left\Vert \frac{h_{K}^{1/2}}{p_{K}}\left(
g-g_{\partial^{B}K}\right)  \right\Vert _{\partial^{B}K}\Bigg)  .\nonumber
\end{align}
Let $\mathfrak{a}\geq C_{\mathfrak{a}}>0$ for some sufficiently large constant
$C_{\mathfrak{a}}$ depending only on the shape regularity of the mesh. Then,
there exists a constant $C>0$ such that%
\begin{align}
\eta_{J}^{2}  \leq &C\sum_{K\in\mathcal{T}}p_{K}^{4}\Bigg(  M_{\frac
{\operatorname*{kh}}{\operatorname*{p}}}^{2}\left\Vert k\left(
u-u_{\mathcal{T}}\right)  \right\Vert _{L^{2}\left(  K\right)  }%
^{2}+\left\Vert \frac{\mathfrak{h}}{\mathfrak{p}}\left(  f-f_{\mathcal{T}%
}\right)  \right\Vert _{L^{2}\left(  K\right)  }^{2}+\left\Vert \nabla\left(
u-u_{\mathcal{T}}\right)  \right\Vert _{L^{2}\left(  K\right)  }^{2}
\tag{%
\ref{effest}%
d}\label{effestd}\\
&   +M_{\frac{\operatorname*{kh}}{\operatorname*{p}}%
}\left\Vert \sqrt{\frac{k}{p_{K}}}\left(  u-u_{\mathcal{T}}\right)
\right\Vert _{\partial^{B}K}^{2}+\left\Vert \frac{\sqrt{h_{K}}}{p_{K}}\left(
g-g_{\partial^{B}K}\right)  \right\Vert _{\partial^{B}K}^{2}\Bigg)
.\nonumber
\end{align}
\end{theorem}
\begin{proof}%
The proof of these estimates follow the ideas of \cite{MelenkWohlmuth} (see
also \cite[Proof of Thm. 4.12]{doerfler_sauter}) and are worked out in detail
in \cite[Sec. 4.2]{ZechMaster}. Here we prove exemplarily (\ref{effestc}) and
(\ref{effestd}).\medskip

\textbf{Proof of (\ref{effestc}).}

We consider the estimate for the edge residuals and start by introducing an
\textit{edge bubble function}. We define $\hat{e}:=\left[  0,1\right]  $ and
$\Phi_{\hat{e}}:\left[  0,1\right]  \rightarrow\mathbb{R}$ by $\Phi_{\hat{e}%
}\left(  x\right)  :=x\left(  1-x\right)  $. For $K\in\mathcal{T}$, let
$F_{K}:\widehat{K}\rightarrow K$ be a usual affine pullback to the reference
element $\widehat{K}:=\operatorname*{conv}\left(  \tbinom{0}{0},\tbinom{1}%
{0},\tbinom{0}{1}\right)  $. For $e\in\mathcal{E}\left(  K\right)  $, we may
choose $F_{K}$ in such a way that $F_{e}:=\left.  F_{K}\right\vert _{\hat{e}%
}:\hat{e}\rightarrow e$. Then we define $\Phi_{e}:e\rightarrow\mathbb{R}$ and
the global version $\Phi_{\mathcal{E}}:\mathfrak{S}\rightarrow\mathbb{R}$ by%
\[
\Phi_{e}:=c_{e}\Phi_{\hat{e}}\circ F_{e}^{-1}\quad\text{with\ }c_{e}%
\in\mathbb{R}\text{ such that }\int_{e}\Phi_{e}=h_{e}\quad\text{and\quad
}\forall e\in\mathcal{E}:\quad\left.  \Phi_{\mathcal{E}}\right\vert _{e}%
:=\Phi_{e}\text{.}%
\]
For $\zeta\in\left[  0,1\right]  $, we introduce%
\begin{align}
\label{defetatildezetra}\tilde{\eta}_{\zeta;E_{K}}\left(  u_{\mathcal{T}%
}\right)  :=  &  \Bigg( \frac{1}{2}\left\Vert \sqrt{{\mathfrak{b}%
\frac{\mathfrak{h}}{\mathfrak{p}}}}%
\llbr
\nabla_{\mathcal{T}}u_{\mathcal{T}}%
\rrbr
_{N}\Phi_{\mathcal{E}}^{\zeta/2}\right\Vert _{\partial^{I}K}^{2}\\
&  +\left\Vert \sqrt{\mathfrak{h}}\left(  g_{\partial^{B}K}-\partial
_{\mathbf{n}}u_{\mathcal{T}}-\operatorname*{i}ku_{\mathcal{T}}\right)
\Phi_{\mathcal{E}}^{\zeta/2}\right\Vert _{\partial^{B}K}^{2}\Bigg)^{1/2}%
\nonumber
\end{align}
and note that $\tilde{\eta}_{0,E_{K}}\left(  u_{\mathcal{T}}\right)
=\tilde{\eta}_{E_{K}}\left(  u_{\mathcal{T}}\right)  $.

For the remaining part of the proof we follow the arguments in \cite[Lem.
3.5]{MelenkWohlmuth} and consider first the second term in the right-hand side
of (\ref{defetatildezetra}). Let first $\zeta\in\left]  \frac{1}{2},1\right]
$. To estimate the second term we employ a certain extension of $\Phi
_{e}^{\zeta}$ to $K$ whose existence is proved in \cite[Lem. 2.6]%
{MelenkWohlmuth} and is stated as follows: Let $\widehat{K}$ be the reference
element and let $\hat{e}=\left[  0,1\right]  \times\left\{  0\right\}  $. Let
$\zeta\in\left]  \frac{1}{2},1\right]  $. Then there exists $C=C\left(
\zeta\right)  >0$ such that, for any $\varepsilon\in\left]  0,1\right]  $,
$p\in\mathbb{N}$, and $\hat{q}\in\mathbb{P}_{p}\left(  \hat{e}\right)  $,
there exists an extension $v_{\hat{e}}\in H^{1}( \widehat{K}) $ of $\hat
{q}\Phi_{\hat{e}}^{\zeta}$ with%
\begin{subequations}
\label{bubext}
\end{subequations}%
\begin{align}
\left.  v_{\hat{e}}\right\vert _{\hat{e}}  &  =\hat{q}\Phi_{\hat{e}}^{\zeta
}\quad\text{and\quad}\left.  v_{\hat{e}}\right\vert _{\partial\widehat
{K}\backslash\hat{e}}=0,\tag{%
\ref{bubext}%
a}\label{bubexta}\\
\left\Vert v_{\hat{e}}\right\Vert _{L^{2}\left(  \widehat{K}\right)  }^{2}  &
\leq C\varepsilon\left\Vert \hat{q}\Phi_{\hat{e}}^{\zeta/2}\right\Vert
_{\hat{e}}^{2},\tag{%
\ref{bubext}%
b}\label{bubextb}\\
\left\Vert \nabla v_{\hat{e}}\right\Vert _{L^{2}\left(  \widehat{K}\right)
}^{2}  &  \leq C\left(  \varepsilon p^{2\left(  2-\zeta\right)  }%
+\varepsilon^{-1}\right)  \left\Vert \hat{q}\Phi_{\hat{e}}^{\zeta
/2}\right\Vert _{\hat{e}}^{2}. \tag{%
\ref{bubext}%
c}\label{bubextc}%
\end{align}

For $e\subset\partial^{B}K$, choose the affine pullback $F_{K}$ such that, for
$F_{e}:=\left.  F_{K}\right\vert _{e}$, it holds $F_{e}\left(  \hat{e}\right)
=e$. We set $q:=g_{\partial^{B}K}-\partial_{\mathbf{n}}u_{\mathcal{T}%
}-\operatorname*{i}ku_{\mathcal{T}}$, denote the pullback by $\hat{q}:=q\circ
F_{e}\in\mathbb{P}_{p_{K}}$, and let $v_{\hat{e}}$ denote the above extension
for this choice of $\hat{q}$. Then $w_{e}:=v_{\hat{e}}\circ F_{K}\in
H^{1}\left(  K\right)  $ and satisfies $\left.  w_{e}\right\vert _{\partial
K\backslash e}=0$. Thus, we obtain with $\partial_{\mathbf{n}}%
u+\operatorname*{i}ku=g$ on $\partial\Omega$%
\begin{align}
\left\Vert q\Phi_{e}^{\zeta/2}\right\Vert _{e}^{2} &  =\left(  g_{\partial
^{B}K}-\partial_{\mathbf{n}}u_{\mathcal{T}}-\operatorname*{i}ku_{\mathcal{T}%
},w_{e}\right)  _{e}\nonumber\\
&  =\left(  \partial_{\mathbf{n}}\left(  u-u_{\mathcal{T}}\right)
,w_{e}\right)  _{e}+\left(  \operatorname*{i}k\left(  u-u_{\mathcal{T}%
}\right)  ,w_{e}\right)  _{e}+\left(  g_{\partial^{B}K}-g,w_{e}\right)
_{e}.\label{qphie}%
\end{align}
We estimate these terms separately and start with the last one and obtain by
using that $\Phi_{\mathcal{E}}^{\zeta/2}$ is bounded pointwise by a constant
$C>0$ uniformly in $\zeta\in\left[  0,1\right]  $ and $x\in\mathfrak{S}$%
\begin{align*}
\left(  g_{\partial^{B}K}-g,w_{e}\right)  _{e}  & \leq\left\Vert
g_{\partial^{B}K}-g\right\Vert _{e}\left\Vert w\right\Vert _{e}=\left\Vert
g_{\partial^{B}K}-g\right\Vert _{e}\left\Vert q\Phi_{e}^{\zeta}\right\Vert
_{e}\\
& \leq C\left\Vert g_{\partial^{B}K}-g\right\Vert _{e}\left\Vert q\Phi
_{e}^{\zeta/2}\right\Vert _{e}.
\end{align*}
For the second term of the right-hand side in (\ref{qphie}) we derive in a
similar fashion%
\[
\left(  \operatorname*{i}k\left(  u-u_{\mathcal{T}}\right)  ,w_{e}\right)
_{e}\leq C\left\Vert k\left(  u-u_{\mathcal{T}}\right)  \right\Vert
_{e}\left\Vert q\Phi_{e}^{\zeta/2}\right\Vert _{e}.
\]
For the first term in (\ref{qphie}) we get%
\begin{align*}
&  \left(  \partial_{\mathbf{n}}\left(  u-u_{\mathcal{T}}\right)
,w_{e}\right)  _{e}=\left(  \partial_{\mathbf{n}}\left(  u-u_{\mathcal{T}%
}\right)  ,w_{e}\right)  _{\partial K}\\
&  \qquad=\left(  \nabla\left(  u-u_{\mathcal{T}}\right)  ,\nabla
w_{e}\right)  _{L^{2}\left(  K\right)  }+\left(  \Delta\left(
u-u_{\mathcal{T}}\right)  ,w_{e}\right)  _{L^{2}\left(  K\right)  }\\
&  \qquad=\left(  \nabla\left(  u-u_{\mathcal{T}}\right)  ,\nabla
w_{e}\right)  _{L^{2}\left(  K\right)  }+\left(  k^{2}\left(  u_{\mathcal{T}%
}-u\right)  ,w_{e}\right)  _{L^{2}\left(  K\right)  }-\left(  \Delta
u_{\mathcal{T}}+k^{2}u_{\mathcal{T}}+f,w_{e}\right)  _{L^{2}\left(  K\right)
}\\
&  \qquad\leq\left\Vert \nabla\left(  u-u_{\mathcal{T}}\right)  \right\Vert
_{L^{2}\left(  K\right)  }\left\Vert \nabla w_{e}\right\Vert _{L^{2}\left(
K\right)  }+\Big(\left\Vert k^{2}\left(  u-u_{\mathcal{T}}\right)  \right\Vert
_{L^{2}\left(  K\right)  }\\
&  \qquad\quad+\left\Vert \Delta u_{\mathcal{T}}+k^{2}u_{\mathcal{T}%
}+f_{\mathcal{T}}\right\Vert _{L^{2}\left(  K\right)  }+\left\Vert
f-f_{\mathcal{T}}\right\Vert _{L^{2}\left(  K\right)  }\Big)\left\Vert
w_{e}\right\Vert _{L^{2}\left(  K\right)  }.
\end{align*}
By scaling (\ref{bubextb}), (\ref{bubextc}) to the triangle $K$ and estimating
$\left\Vert \Delta u_{\mathcal{T}}+k^{2}u_{\mathcal{T}}+f_{\mathcal{T}%
}\right\Vert _{L^{2}\left(  K\right)  }=\frac{p_{K}}{h_{K}}\tilde{\eta}%
_{R_{K}}\left(  u_{\mathcal{T}}\right)  $ via (\ref{effesta}), we get%
\begin{align*}
&  \left(  \partial_{\mathbf{n}}\left(  u-u_{\mathcal{T}}\right)
,w_{e}\right)  _{e}\leq C\left\Vert q\Phi_{e}^{\zeta/2}\right\Vert
_{e}\Bigg\{\left(  \frac{\varepsilon p_{K}^{2\left(  2-\zeta\right)
}+\varepsilon^{-1}}{h_{K}}\right)  ^{1/2}\left\Vert \nabla\left(
u-u_{\mathcal{T}}\right)  \right\Vert _{L^{2}\left(  K\right)  }\\
&  \qquad\quad+\sqrt{\varepsilon h_{K}}\left(  \left\Vert k^{2}\left(
u-u_{\mathcal{T}}\right)  \right\Vert _{L^{2}\left(  K\right)  }+\frac{p_{K}%
}{h_{K}}\tilde{\eta}_{R_{K}}\left(  u_{\mathcal{T}}\right)  +\left\Vert
f-f_{\mathcal{T}}\right\Vert _{L^{2}\left(  K\right)  }\right)  \Bigg\}\\
&  \qquad\leq C\left\Vert q\Phi_{e}^{\zeta/2}\right\Vert _{e}\Bigg\{\left(
\frac{\varepsilon p_{K}^{2\left(  2-\zeta\right)  }+\varepsilon^{-1}%
+\varepsilon p_{K}^{4}}{h_{K}}\right)  ^{1/2}\left\Vert \nabla\left(
u-u_{\mathcal{T}}\right)  \right\Vert _{L^{2}\left(  K\right)  }\\
&  \qquad\quad+\sqrt{\varepsilon h_{K}}\Big(\left\Vert k^{2}\left(
u-u_{\mathcal{T}}\right)  \right\Vert _{L^{2}\left(  K\right)  }+\frac
{p_{K}^{2}}{h_{K}}M_{\frac{\operatorname*{kh}}{\operatorname*{p}}}\left\Vert
k\left(  u-u_{\mathcal{T}}\right)  \right\Vert _{L^{2}\left(  K\right)  }\\
&  \qquad\quad+p_{K}\left\Vert \left(  f-f_{\mathcal{T}}\right)  \right\Vert
_{L^{2}\left(  K\right)  }\Big)\Bigg\}.
\end{align*}
Altogether we have proved (for the choice $\varepsilon=p_{K}^{-2}$)%
\begin{align}
&  \left\Vert \sqrt{\mathfrak{h}}q\Phi_{e}^{\zeta/2}\right\Vert _{\partial
^{B}K}^{2}\leq Cp_{K}^{2}\Bigg(\left\Vert \nabla\left(  u-u_{\mathcal{T}%
}\right)  \right\Vert _{L^{2}\left(  K\right)  }^{2}+M_{\frac
{\operatorname*{kh}}{\operatorname*{p}}}^{2}\left\Vert k\left(
u-u_{\mathcal{T}}\right)  \right\Vert _{L^{2}\left(  K\right)  }%
^{2}\label{2ndfinaledgeres}\\
&  \qquad+\left\Vert \frac{h_{K}}{p_{K}}f-f_{\mathcal{T}}\right\Vert
_{L^{2}\left(  K\right)  }^{2}+M_{\frac{\operatorname*{kh}}{\operatorname*{p}%
}}\left\Vert \sqrt{\frac{k}{p_{K}}}\left(  u-u_{\mathcal{T}}\right)
\right\Vert _{\partial^{B}K}^{2}+\left\Vert \frac{\sqrt{\mathfrak{h}}}{p_{K}%
}\left(  g_{\partial^{B}K}-g\right)  \right\Vert _{\partial^{B}K}%
^{2}\Bigg).\nonumber
\end{align}
For $\zeta\in\left[  0,1/2\right]  $ we obtain from \cite[Lem. 2.4 with
$\beta=1$ and $\alpha=\zeta$]{MelenkWohlmuth}%
\begin{equation}
\left\Vert q\Phi_{e}^{\zeta/2}\right\Vert _{e}\leq Cp_{K}^{1-\zeta}\left\Vert
q\Phi_{e}^{1/2}\right\Vert _{e}.\label{scaledgeinvin}%
\end{equation}
By choosing $\zeta=0$ in (\ref{scaledgeinvin}) and $\zeta=1$ in
(\ref{2ndfinaledgeres}) we get%
\begin{align}
\left\Vert \sqrt{\mathfrak{h}}q\right\Vert _{\partial^{B}K}\leq &
Cp_{K}\left\Vert q\Phi_{\mathcal{E}}^{1/2}\right\Vert _{\partial^{B}K}\leq
Cp_{K}^{2}\Bigg(\left\Vert \nabla\left(  u-u_{\mathcal{T}}\right)  \right\Vert
_{L^{2}\left(  K\right)  }\label{sqrthq}\\
&  +M_{\frac{\operatorname*{kh}}{\operatorname*{p}}}\left\Vert k\left(
u-u_{\mathcal{T}}\right)  \right\Vert _{L^{2}\left(  K\right)  }+\left\Vert
\frac{h_{K}}{p_{K}}f-f_{\mathcal{T}}\right\Vert _{L^{2}\left(  K\right)
}\nonumber\\
&  +\sqrt{M_{\frac{\operatorname*{kh}}{\operatorname*{p}}}}\left\Vert
\sqrt{\frac{k}{p_{K}}}\left(  u-u_{\mathcal{T}}\right)  \right\Vert
_{\partial^{B}K}+\left\Vert \frac{\sqrt{\mathfrak{h}}}{p_{K}}\left(
g_{\partial^{B}K}-g\right)  \right\Vert _{\partial^{B}K}\Bigg).\nonumber
\end{align}

This finishes the estimate of the second term in the right-hand side of
(\ref{defetatildezetra}). The first term can be estimated via (\ref{effestb})
and leads to (\ref{effestc}).\medskip

\textbf{Proof of (\ref{effestd}).}

\textbf{Part 1. }We prove%
\begin{align}
\label{etajumpproof}\eta_{J}^{2}\leq &  C\Bigg(\operatorname*{osc}%
\nolimits_{\mathcal{T}}^{2}+\tilde{\eta}_{R}^{2}+\sum_{K\in\mathcal{T}%
}\Bigg( \frac{p_{K}}{2}\left\Vert \sqrt{{\mathfrak{b}\frac{\mathfrak{h}%
}{\mathfrak{p}}}}%
\llbr
\nabla_{\mathcal{T}}u_{\mathcal{T}}%
\rrbr
_{N}\right\Vert _{\partial^{I}K}^{2}\\
&  +\left\Vert \sqrt{\mathfrak{h}}\left(  g_{\partial^{B}K}-\left(
\partial_{\mathbf{n}}+\operatorname*{i}k\right)  u_{\mathcal{T}}\right)
\right\Vert _{\partial^{B}K}^{2}\Bigg) \Bigg) .\nonumber
\end{align}
Let $u_{\mathcal{T}}^{\ast}\in S_{\mathcal{T}}^{\mathfrak{p}}$ denote the
conforming approximant of $u_{\mathcal{T}}$ (cf. Corollary \ref{CorConfApprox}%
). Due to Galerkin orthogonality it holds%
\begin{equation}
a_{\mathcal{T}}\left(  u-u_{\mathcal{T}},u_{\mathcal{T}}-u_{\mathcal{T}}%
^{\ast}\right)  =0. \label{GalOrthJump}%
\end{equation}
The continuity of $u_{\mathcal{T}}^{\ast}$ implies%
\begin{align*}
\sum_{K\in\mathcal{T}}\eta_{J_{K}}^{2}  &  =\sum_{K\in\mathcal{T}}\frac{1}%
{2}\left\Vert \sqrt{{\mathfrak{a}\frac{\mathfrak{p}^{2}}{\mathfrak{h}}}}%
\llbr
u_{\mathcal{T}}%
\rrbr
\right\Vert _{\partial^{I}K}^{2}=\left\vert \left(  \operatorname*{i}%
{\mathfrak{a}\frac{\mathfrak{p}^{2}}{\mathfrak{h}}}%
\llbr
u_{\mathcal{T}}%
\rrbr
_{N},%
\llbr
u_{\mathcal{T}}%
\rrbr
_{N}\right)  _{\mathfrak{S}^{I}}\right\vert \\
&  =\left\vert \left(  \operatorname*{i}{\mathfrak{a}\frac{\mathfrak{p}^{2}%
}{\mathfrak{h}}}%
\llbr
u_{\mathcal{T}}%
\rrbr
_{N},%
\llbr
u_{\mathcal{T}}-u_{\mathcal{T}}^{\ast}%
\rrbr
_{N}\right)  _{\mathfrak{S}^{I}}\right\vert
\end{align*}
and we combine (\ref{GalOrthJump}) with the representation as in Lemma
\ref{LemAltRepA} to obtain%
\begin{align*}
\sum_{K\in\mathcal{T}}\eta_{J_{K}}^{2}\leq &  \eta_{R}\left\Vert
\dfrac{\mathfrak{p}}{\mathfrak{h}}\left(  u_{\mathcal{T}}-u_{\mathcal{T}%
}^{\ast}\right)  \right\Vert +\left\Vert \sqrt{\mathfrak{hd}}%
\llbr
\nabla_{\mathcal{T}}u_{\mathcal{T}}%
\rrbr
_{N}\right\Vert _{\mathfrak{S}^{I}}\left\Vert \left(  \mathfrak{hd}\right)
^{-1/2}\left\{  u_{\mathcal{T}}-u_{\mathcal{T}}^{\ast}\right\}  \right\Vert
_{\mathfrak{S}^{I}}\\
&  +\left\Vert \sqrt{{\mathfrak{a}\frac{\mathfrak{p}^{2}}{\mathfrak{h}}}}%
\llbr
u_{\mathcal{T}}%
\rrbr
_{N}\right\Vert _{\mathfrak{S}^{I}}\left\Vert \left(  {\mathfrak{a}%
\frac{\mathfrak{p}^{2}}{\mathfrak{h}}}\right)  ^{-1/2}\left\{  \nabla
_{\mathcal{T}}\left(  u_{\mathcal{T}}-u_{\mathcal{T}}^{\ast}\right)  \right\}
\right\Vert _{\mathfrak{S}^{I}}\\
&  +\left\Vert \sqrt{\mathfrak{h}}\left(  g-\partial_{\mathbf{n}%
}u_{\mathcal{T}}-\operatorname*{i}ku_{\mathcal{T}}\right)  \right\Vert
_{\partial\Omega}\left\Vert \mathfrak{h}^{-1/2}\left(  u_{\mathcal{T}%
}-u_{\mathcal{T}}^{\ast}\right)  \right\Vert _{\partial\Omega}\\
&  +\left\Vert \sqrt{\mathfrak{h}}\left(  g-\partial_{\mathbf{n}%
}u_{\mathcal{T}}-\operatorname*{i}ku_{\mathcal{T}}\right)  \right\Vert
_{\partial\Omega}\left\Vert \mathfrak{d}\frac{\sqrt{\mathfrak{h}}%
}{\mathfrak{p}}\partial_{\mathbf{n}}\left(  u_{\mathcal{T}}-u_{\mathcal{T}%
}^{\ast}\right)  \right\Vert _{\partial\Omega}\\
&  +\left\Vert \sqrt{\mathfrak{bh}}%
\llbr
\nabla_{\mathcal{T}}u_{\mathcal{T}}%
\rrbr
_{N}\right\Vert _{\mathfrak{S}^{I}}\left\Vert \frac{\sqrt{\mathfrak{bh}}%
}{\mathfrak{p}}%
\llbr
\nabla_{\mathcal{T}}\left(  u_{\mathcal{T}}-u_{\mathcal{T}}^{\ast}\right)
\rrbr
_{N}\right\Vert _{\mathfrak{S}^{I}}.
\end{align*}
The factors which contain $u_{\mathcal{T}}-u_{\mathcal{T}}^{\ast}$ can be
estimated by using Theorem \ref{TheoConfApprox} and polynomial inverse
estimates
\begin{align*}
\left\Vert \frac{\mathfrak{p}}{\mathfrak{h}}\left(  u_{\mathcal{T}%
}-u_{\mathcal{T}}^{\ast}\right)  \right\Vert  &  \leq\frac{C}{\sqrt
{\mathfrak{a}}}\left\Vert \sqrt{{\mathfrak{a}\frac{\mathfrak{p}^{2}%
}{\mathfrak{h}}}}%
\llbr
u_{\mathcal{T}}%
\rrbr
\right\Vert _{\mathfrak{S}^{I}},\\
\left\Vert \frac{\left\{  \nabla_{\mathcal{T}}\left(  u_{\mathcal{T}%
}-u_{\mathcal{T}}^{\ast}\right)  \right\}  }{\sqrt{{\mathfrak{a}%
\frac{\mathfrak{p}^{2}}{\mathfrak{h}}}}}\right\Vert _{\mathfrak{S}^{I}}  &
\leq\frac{C}{\sqrt{\mathfrak{a}}}\left\Vert \nabla_{\mathcal{T}}\left(
u_{\mathcal{T}}-u_{\mathcal{T}}^{\ast}\right)  \right\Vert _{\mathfrak{S}^{I}%
}\leq\frac{C}{\mathfrak{a}}\left\Vert \sqrt{{\mathfrak{a}\frac{\mathfrak{p}%
^{2}}{\mathfrak{h}}}}%
\llbr
u_{\mathcal{T}}%
\rrbr
\right\Vert _{\mathfrak{S}^{I}},\\
\left\Vert \frac{\left\{  u_{\mathcal{T}}-u_{\mathcal{T}}^{\ast}\right\}
}{\sqrt{\mathfrak{hd}}}\right\Vert _{\mathfrak{S}^{I}}+\left\Vert
\frac{u_{\mathcal{T}}-u_{\mathcal{T}}^{\ast}}{\sqrt{\mathfrak{h}}}\right\Vert
_{\partial\Omega}  &  \leq C\left\Vert \frac{\mathfrak{p}}{\mathfrak{h}%
}\left(  u_{\mathcal{T}}-u_{\mathcal{T}}^{\ast}\right)  \right\Vert \leq
\frac{C}{\sqrt{\mathfrak{a}}}\left\Vert \sqrt{{\mathfrak{a}\frac
{\mathfrak{p}^{2}}{\mathfrak{h}}}}%
\llbr
u_{\mathcal{T}}%
\rrbr
\right\Vert _{\mathfrak{S}^{I}},\\
\left\Vert \mathfrak{d}\frac{\sqrt{\mathfrak{h}}}{\mathfrak{p}}\partial
_{\mathbf{n}}\left(  u_{\mathcal{T}}-u_{\mathcal{T}}^{\ast}\right)
\right\Vert _{\partial\Omega}  &  +\left\Vert \frac{\sqrt{\mathfrak{bh}}%
}{\mathfrak{p}}%
\llbr
\nabla_{\mathcal{T}}\left(  u_{\mathcal{T}}-u_{\mathcal{T}}^{\ast}\right)
\rrbr
_{N}\right\Vert _{\mathfrak{S}^{I}}\\
&  \leq C\left\Vert \nabla\left(  u_{\mathcal{T}}-u_{\mathcal{T}}^{\ast
}\right)  \right\Vert \leq\frac{C}{\sqrt{\mathfrak{a}}}\left\Vert
\sqrt{{\mathfrak{a}\frac{\mathfrak{p}^{2}}{\mathfrak{h}}}}%
\llbr
u_{\mathcal{T}}%
\rrbr
\right\Vert _{\mathfrak{S}^{I}}.
\end{align*}
This finally leads to%
\begin{align*}
\sum_{K\in\mathcal{T}}\eta_{J_{K}}^{2}\leq &  \frac{C}{\sqrt{\mathfrak{a}}}
\Bigg( \eta_{R}+\left\Vert \sqrt{\mathfrak{h}}%
\llbr
\nabla_{\mathcal{T}}u_{\mathcal{T}}%
\rrbr
_{N}\right\Vert _{\mathfrak{S}^{I}}+\left\Vert \sqrt{\mathfrak{h}}\left(
g-\partial_{\mathbf{n}}u_{\mathcal{T}}-\operatorname*{i}ku_{\mathcal{T}%
}\right)  \right\Vert _{\partial\Omega}\\
&  +\frac{1}{\sqrt{\mathfrak{a}}}\left\Vert \sqrt{{\mathfrak{a}\frac
{\mathfrak{p}^{2}}{\mathfrak{h}}}}%
\llbr
u_{\mathcal{T}}%
\rrbr
_{N}\right\Vert _{\mathfrak{S}^{I}}\Bigg) \left\Vert \sqrt{{\mathfrak{a}%
\frac{\mathfrak{p}^{2}}{\mathfrak{h}}}}%
\llbr
u_{\mathcal{T}}%
\rrbr
\right\Vert _{\mathfrak{S}^{I}}.
\end{align*}
We divide this inequality by the last factor, absorb the last summand in the
left-hand side for sufficiently large $\mathfrak{a}$, and estimate $\eta
_{R}\leq\tilde{\eta}_{R}+\operatorname*{osc}_{\mathcal{T}}$. Thus, we have
proved (\ref{etajumpproof}).\medskip

\textbf{Part 2.} From (\ref{etajumpproof}) we will derive (\ref{effestd}%
)\textbf{.}

The second term in the right-hand side of (\ref{etajumpproof}) can be
estimated by using (\ref{effesta}) while the estimate for the last sum in
(\ref{etajumpproof}) follows from (\ref{effestb}) and (\ref{sqrthq}).%
\qquad\end{proof}%

\begin{remark}
\label{Rempoll}\quad

\begin{enumerate}
\item[a.] As is well-known for residual a posteriori error estimation in the
context of $hp$-finite elements, the \emph{reliability estimate} is robust with
respect to the polynomial degree while the \emph{efficiency estimate} is
polluted by powers of $p_{K}$ due to inverse inequalities. The theory of
\cite{MelenkWohlmuth} allows to shift powers of $p_{K}$ in the efficiency
estimate to powers of $p_{K}$ in the reliability estimate by employing certain
powers of bubble functions in the definition of the error estimator. This can
also be done for the $\operatorname*{dG}$-formulation of the Helmholtz problem
and is worked out in \cite{ZechMaster}.
\item[b.] A difference to standard elliptic problems is the appearance of the
adjoint approximation property $\sigma_{k}^{\ast}\left(  S_{\mathcal{T}}^{\mathfrak{p}}\right)$
(cf. (\ref{adjointapp})) in the reliability estimate, and powers of the quantity
$M_{\frac{\operatorname*{kh}}{\operatorname*{p}}}$ in (\ref{defMkhp}) in both, the
efficiency and the reliability estimates. For convex polygonal domains, it can be
shown that $p\ge C_0 \log (k)$ and the \emph{resolution condition }%
\begin{equation}
\frac{kh_{K}}{p_{K}}\leq C_{1}\qquad\forall K\in\mathcal{T}, \label{rescond}%
\end{equation}
for some $C_0,C_1>0$, together with appropriate geometrical
mesh refinement in neighbourhoods of the polygon vertices are sufficient to bound
the adjoint approximation property $\sigma_{k}^{\ast}\left(  S_{\mathcal{T}}^{\mathfrak{p}}
\right)$ (see \cite[Thm. 2.4.2]{ZechMaster} and \cite{MelenkSauterMathComp, MPS13}).
The constant $M_{\frac{\operatorname*{kh}}{\operatorname*{p}}}$ is then controlled
by $C_{1}$. The above conditions are easily satisfied and imply that only $O\left(1\right)
$ degrees of freedom per wave length and per coordinate direction are necessary to obtain
a $k$-independent reliabilty estimate.
\item [c.] Note that in the reliability estimate the factor
$\sigma_{k}^{\ast}\left(  S_{\mathcal{T}}^{\mathfrak{p}}\right) M_{\frac{\operatorname*{kh}}%
{\operatorname*{p}}}^{3/2}$ appears and in the efficiency estimate the factor
$M_{\frac{\operatorname*{kh}}{\operatorname*{p}}}$ appears. This indicates that for large
$M_{\frac{\operatorname*{kh}}{\operatorname*{p}}}$ the estimator might overestimate or
underestimate the error, whereas a large value of $\sigma_{k}^{\ast}\left(  S_{\mathcal{T}}^%
{\mathfrak{p}}\right)$ suggests that the error might be underestimated (cf. \cite{Strouboulis1,%
Ihlenburgbook, Oden2005} and also Fig. \ref{fig:fig2}).
\end{enumerate}
\end{remark}

\begin{remark}
\label{RemWithoutJump}The proof of (\ref{etajumpproof}) implies that the jump
term $\eta_{J}$ in the error estimator can be omitted under two mild
restrictions: a) The constant $\mathfrak{a}$ in (\ref{dgsesquilina}) must
satisfy $\mathfrak{a}\geq C_{\mathfrak{a}}>0$ for a sufficiently large
constant $C_{\mathfrak{a}}$ which only depends on the shape regularity via
$\rho_{\mathcal{T}}$. However, explicit estimates for $C_{\mathfrak{a}}$ are
not available yet. b) The edge terms in the right-hand side of
(\ref{etajumpproof}) are by a factor $\sqrt{p_{e}}$ larger compared to edge
residuals $\eta_{E_{K}}$ and this leads to a reliability error estimate for
the error estimator without jump term $\eta_{J}$ which is polluted by a factor
$\sqrt{p_{e}}$. However, the a priori analysis in \cite{MelenkSauterMathComp}
and \cite{mm_stas_helm2} indicates that $p\sim\log k$ is a typical choice so
that this pollution is expected to be quite harmless.
\end{remark}

\section{Numerical Experiments\label{SecNumExp}}

In this section we will report on numerical experiments to get
insights in the following questions: a) How sharp does the error
estimator reflect the behavior of the true error for uniform as well
as for adaptive mesh refinement. b) How does the error estimator
behave for scenarios which are not covered by our theory: for
non-constant wavenumbers as well as for non-convex domains.

We have realized the $\operatorname*{dG}$-discretization with
\textsc{Matlab} and based the implementation on the finite element
toolbox
LehrFEM\footnote{\url{http://www.sam.math.ethz.ch/~hiptmair/tmp/LehrFEMManual.pdf}}%
.

The error in this section will be measured in the norm
\begin{align*}
  \Vert u \Vert_{\mathcal{H};\mathcal{T}}:=\Vert k u \Vert+\Vert\nabla
  _{\mathcal{T}} u\Vert.
\end{align*}

\subsection{Adaptive Algorithm}

First, we will briefly describe our adaptive algorithm and refer for
details, e.g., to \cite{NochSiebVees2009}. It consists of the
following flow of modules:
\texttt{SOLVE}$\longrightarrow$\texttt{ESTIMATE}$\longrightarrow
$\texttt{MARK}$\longrightarrow$\texttt{REFINE} and we will comment on
their realization next.

\subsubsection{Solve}

The module \texttt{SOLVE} finds the solution $u_{\mathcal{T}}$ of
(\ref{dgsesquilin}) for a given mesh $\mathcal{T}$ with polynomial
degree function $\mathfrak{p}$ and data $f$, $g$, $k$, $\Omega$. In
our implementation all integrals involved in (\ref{dgsesquilin}) are
computed by quadrature on edges and elements.

\subsubsection{Estimate}

As explained in Remark \ref{RemWithoutJump} we have omitted the jump
term $\eta_{J}$ and realized the right-hand side in
(\ref{etajumpproof}) as the error estimator. For simplicity we have
also omitted the oscillation terms and worked with the functions $f$,
$g$ instead. Again, all integrals are computed via numerical
quadrature. The resulting local and global error estimator are denoted
by
\[
\check{\eta}_{K}^{2}:\dot{=}\eta_{R_{K}}^{2}+\frac{p_{K}}{2}\left\Vert
  \sqrt{{\mathfrak{b}\frac{\mathfrak{h}}{\mathfrak{p}}}}%
  \llbr
  \nabla_{\mathcal{T}}u_{\mathcal{T}}%
  \rrbr
  _{N}\right\Vert _{\partial^{I}K}^{2}+\left\Vert
  \sqrt{\mathfrak{h}}\left(
    g-\left( \partial_{\mathbf{n}}+\operatorname*{i}k\right)
    u_{\mathcal{T}%
    }\right) \right\Vert _{\partial^{B}K}^{2}%
\]
and
\[
\check{\eta}:=\sum_{K\in\mathcal{T}}\check{\eta}_{K}^{2},
\]
where the notation \textquotedblleft$\dot{=}$\textquotedblright\
indicates that the left-hand side equals the right-hand side up to
numerical quadrature.

\subsubsection{Mark}

After having computed the local estimators $\check{\eta}_{K}$ a
refinement strategy has to be applied and we employ D\"{o}rfler's
marking strategy: Fix the triangulation $\mathcal{T}$ and let
$u_{\mathcal{T}}\in S_{\mathcal{T}%
}^{\mathfrak{p}}$ be the $\operatorname*{dG}$-solution. Denote by
$\mathcal{S}$ some subset of $\mathcal{T}$. We write%
\[
\check{\eta}\left( u_{\mathcal{T}},\mathcal{S}\right) :=\sum_{K\in
  \mathcal{S}}\check{\eta}_{K}^{2}\left( u_{\mathcal{T}}\right) .
\]
For fixed threshold $\theta\in\left] 0,1\right] $, the set of marked
elements $\mathcal{M}\subseteq\mathcal{T}$ is defined by%
\[
\mathcal{M}:=\operatorname{argmin}\left\{ \operatorname{card}\left(
    \mathcal{S}\right)
  \mid\mathcal{S}\subseteq\mathcal{T}\wedge\check{\eta }\left(
    u_{\mathcal{T}},\mathcal{S}\right) \geq\theta\check{\eta}\left(
    u_{\mathcal{T}},\mathcal{T}\right) \right\} .
\]

\subsubsection{Refine}

In this step, all elements $K\in\mathcal{M}$ are refined. Some
additional elements are refined to eliminate hanging nodes and we have
realized the largest edge bisection for this purpose. We emphasize
that our implementation is currently restricted to $h$ refinement
while an extension to adaptive $hp$-refinement will be the topic of
future research.

\subsection{Plane Wave Solutions}

The parameters $\mathfrak{a}=30$, $\mathfrak{b}=1$, and
$\mathfrak{d}=1/4$ in (\ref{dgsesquilin}) are fixed for all
experiments in this section. The adaptive refinement process is always
started on a coarse mesh where the number of mesh cells is $O\left(
  1\right) $ independent of $k$ and $p$.

\subsubsection{Example 1}

Let $\Omega=\left( 0,1\right) ^{2}$ and the data $f$, $g$ be given
such that $u\left( x,y\right) :=\exp\left( \operatorname*{i}k\left(
    x+y\right) \right) $ is the exact solution. As $u$ is an entire
function it is reasonable to refine the mesh uniformly. In
Fig. \ref{fig:fig1}, we compare the relative error in the $\left\Vert
  \cdot\right\Vert _{\mathcal{H}}$ norm for different wavenumbers. As
expected a) the pollution effect is visible, i.e., the convergence
starts later for higher wavenumbers and b) the pollution becomes
smaller for higher polynomial degree.

\begin{figure}[t]
  \subfloat[$p=1$]{\includegraphics[width=0.49\textwidth]{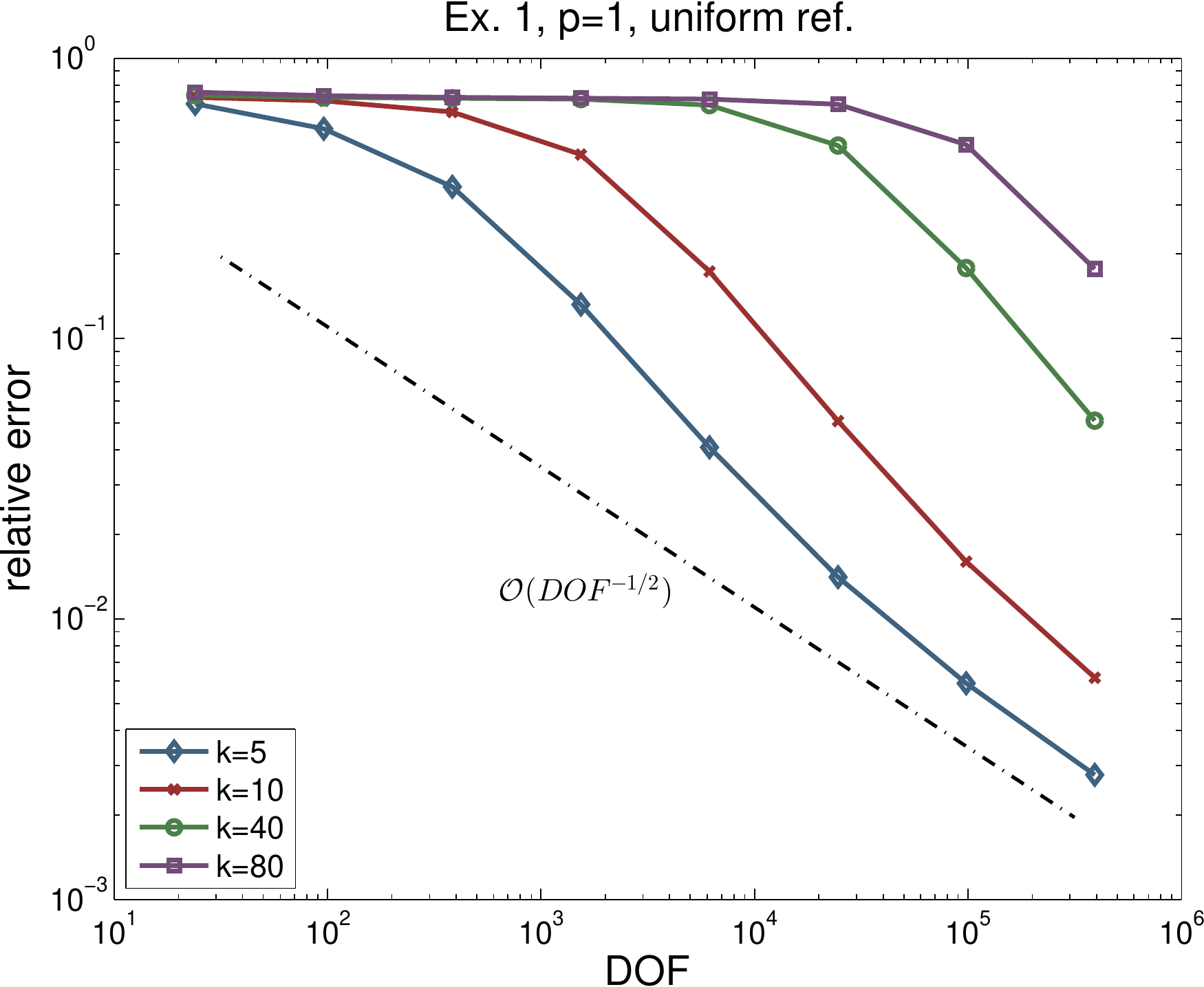}}
  \hfill
  \subfloat[$p=3$]{\includegraphics[width=0.49\textwidth]{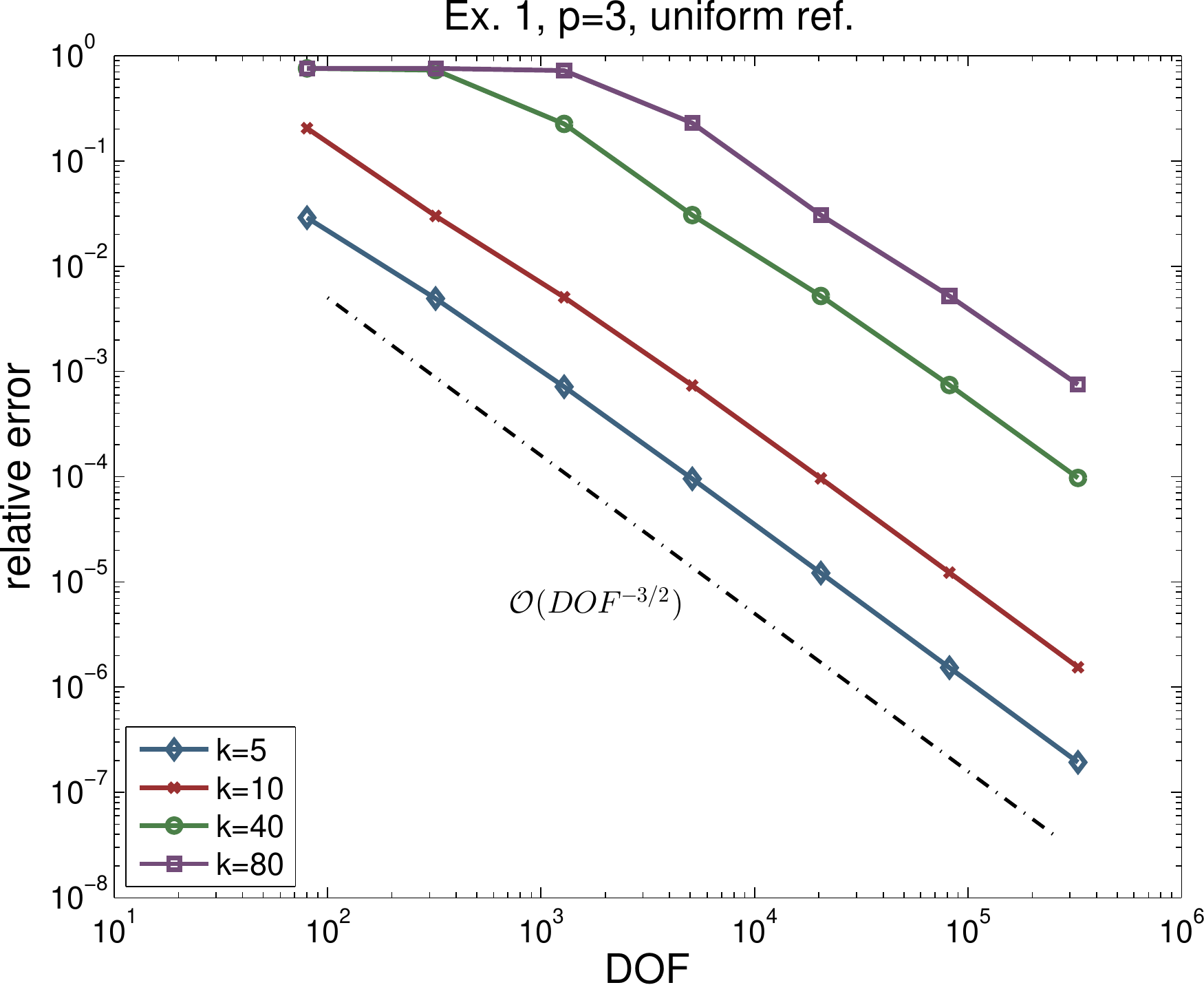}}
  \caption[Convergence for different values of $k$, Ex. 1]{Comparison
    of the relative error in the norm $\Vert \cdot
    \Vert_{\mathcal{H};\mathcal{T}}$, for the polynomial degrees $p=1$
    and $p=3$ for different values of $k$ in Example
    1.}\label{fig:fig1}
\end{figure}

Next we test the sharpness of the reliability estimate for the error
estimator. In Fig. \ref{fig:fig2} the ratio $\left\Vert
  u-u_{\mathcal{T}%
  }\right\Vert _{\mathcal{H};\mathcal{T}}/\check{\eta}\left(
  u_{\mathcal{T}%
  }\right) $ for different polynomial degrees and wavenumbers are
depicted.  Since we start with a very coarse initial mesh the constant
$M_{\frac {\operatorname*{kh}}{\operatorname*{p}}}$ increases with
increasing $k$ in the pre-asymptotic regime and, due to Remark
\ref{Rempoll}.c, an underestimating can be expected (as compared to
when the asymptotic regime is reached). This effect can be seen in
Fig. \ref{fig:fig2} while the asymptotic regime is reached faster for
higher order polynomial degree.

\begin{figure}[t]
  \subfloat[$p=1$]{\includegraphics[width=0.49\textwidth]{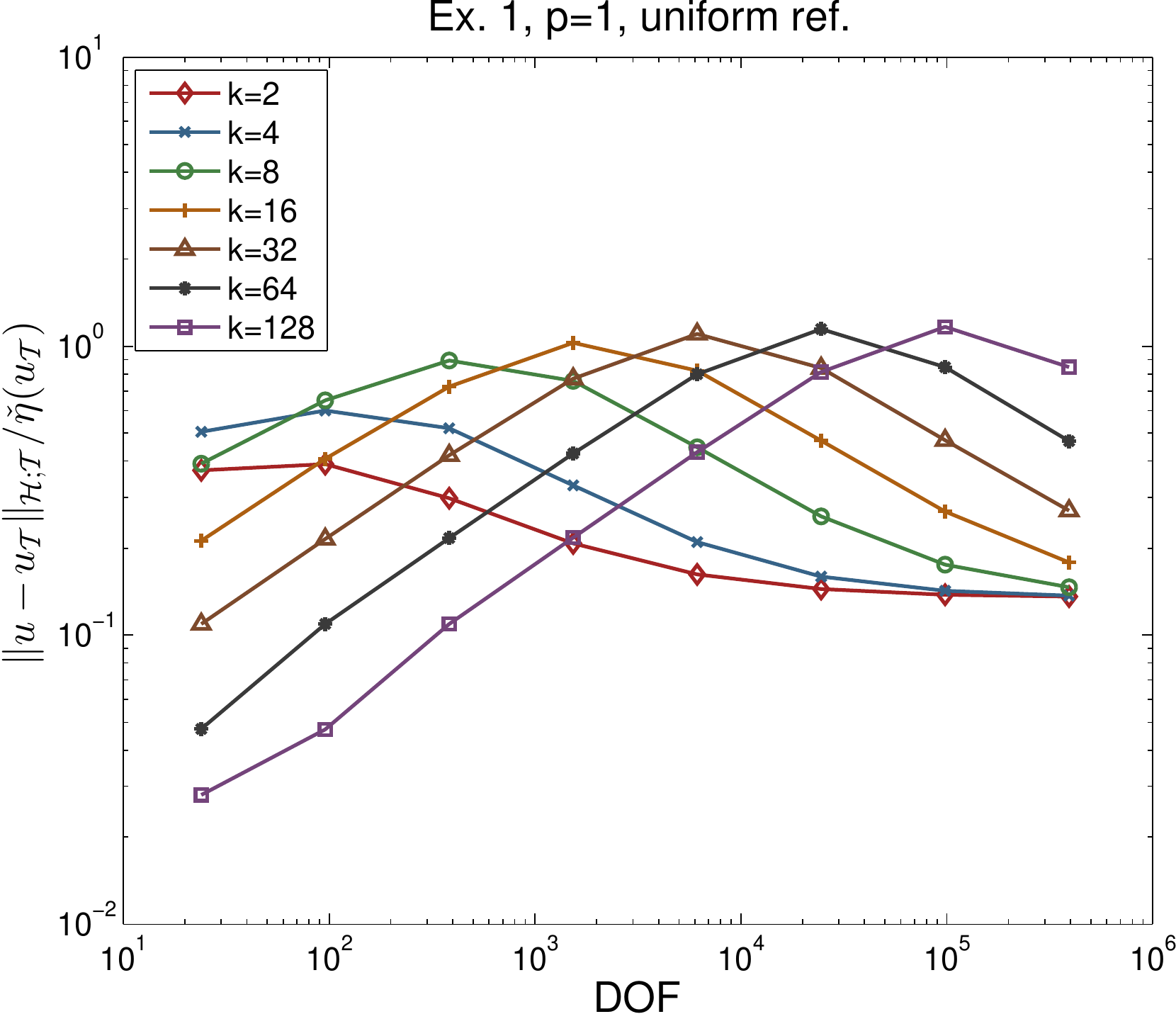}}
  \hfill
  \subfloat[$p=3$]{\includegraphics[width=0.49\textwidth]{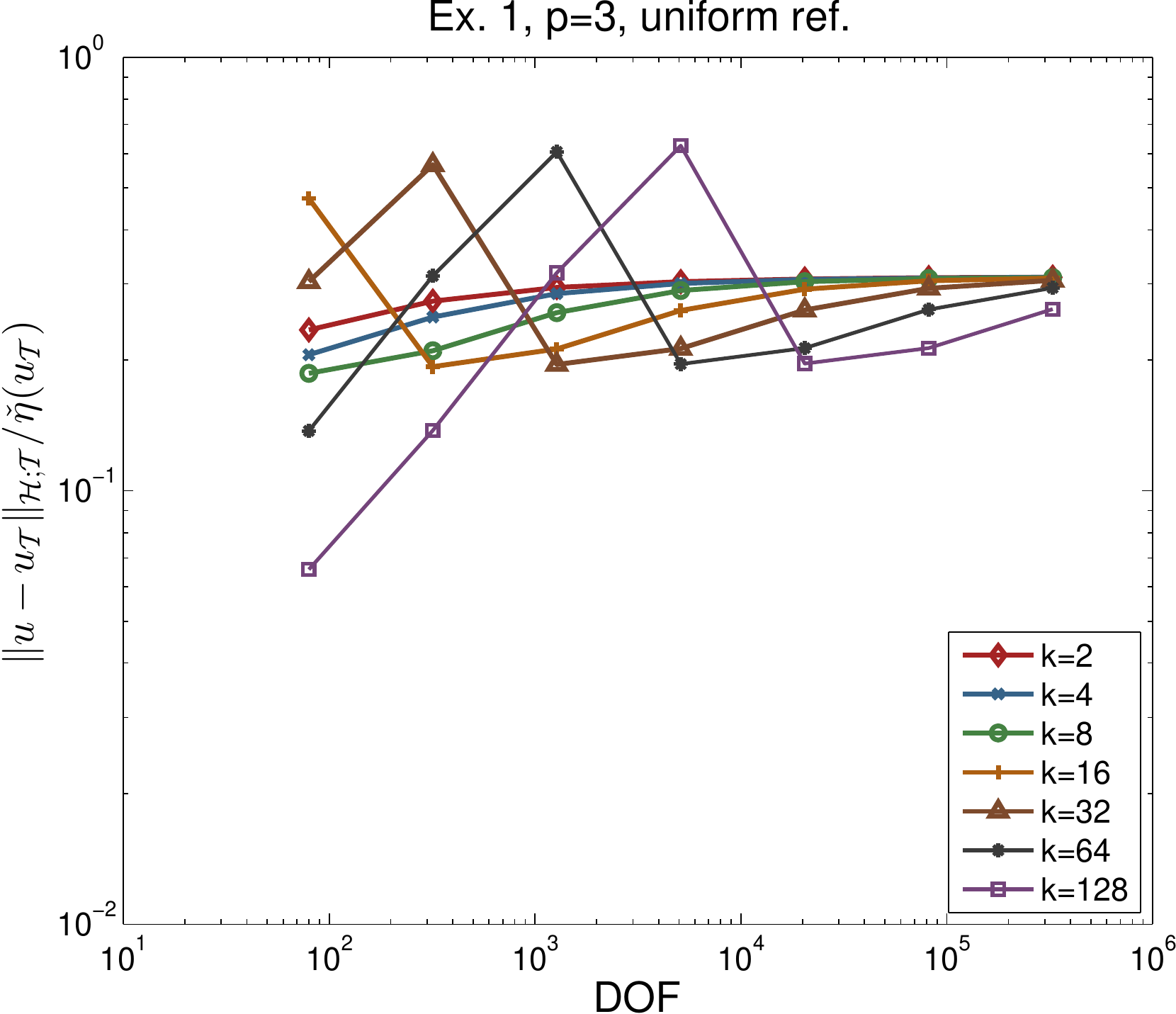}}
  \caption[Ratio of the exact and estimated error, Ex. 1]{Ratio of the
    exact error $\Vert
    u-u_{\mathcal{T}}\Vert_{\mathcal{H};\mathcal{T}}$ and the
    estimated error $\check{\eta}(u_{\mathcal{T}})$ for different
    values of $k$ in Example 1.}\label{fig:fig2}
\end{figure}

\subsubsection{Example 2}

We consider the Helmholtz problem on $\Omega=\left( 0,2\pi\right)
^{2}$ with the exact solution $u\left( x,y\right) =\exp\left(
  \operatorname*{i}%
  kx\right) $. The corresponding functions $f$ and $g$ are chosen
accordingly:%
\begin{equation}
  f:=0\quad\text{and\quad}g\left(  x,y\right)  :=\left\{
    \begin{array}
      [c]{ll}%
      0 & \text{if }x=0,\\
      2\operatorname*{i}k & \text{if }x=2\pi,\\
      \operatorname*{i}k\operatorname*{e}^{\operatorname*{i}kx} & \text{otherwise,}%
    \end{array}
  \right.  \qquad\forall\left(  x,y\right)  \in\partial\Omega. \label{defg2}%
\end{equation}
The $\operatorname*{dG}$-solution for very coarse meshes is strongly
polluted and does not reflect the uniformly oscillating behavior,
e.g., in the imaginary part $\operatorname{Im}u=\sin kx$ of the
solution. One possible interpretation is that $f=0$ in $\Omega$ and
$g=0$ at the left boundary have the effect that $u_{\mathcal{T}}$ is
small close to the left boundary while at the right boundary the
oscillations got resolved earlier. This is \textquotedblleft
seen\textquotedblright\ also by the error estimator and stronger
refinement takes place in the early stage of adaptivity close to the
right boundary. Only after some refinement steps the strong mesh
refinement penetrates from right to left into the whole domain
(see. Fig. \ref{fig:fig4}%
). In Fig. \ref{fig:fig5}(a), we see that the mesh starts to become
uniform as soon as the resolution condition (\ref{rescond}) is
fulfilled and the error starts to decrease.

\begin{figure}[t]
  \subfloat[Mesh after 26 adaptive refinements, 9044 elements, $\min_e
  h_e = 0.049$] {\includegraphics[width=0.45\textwidth]{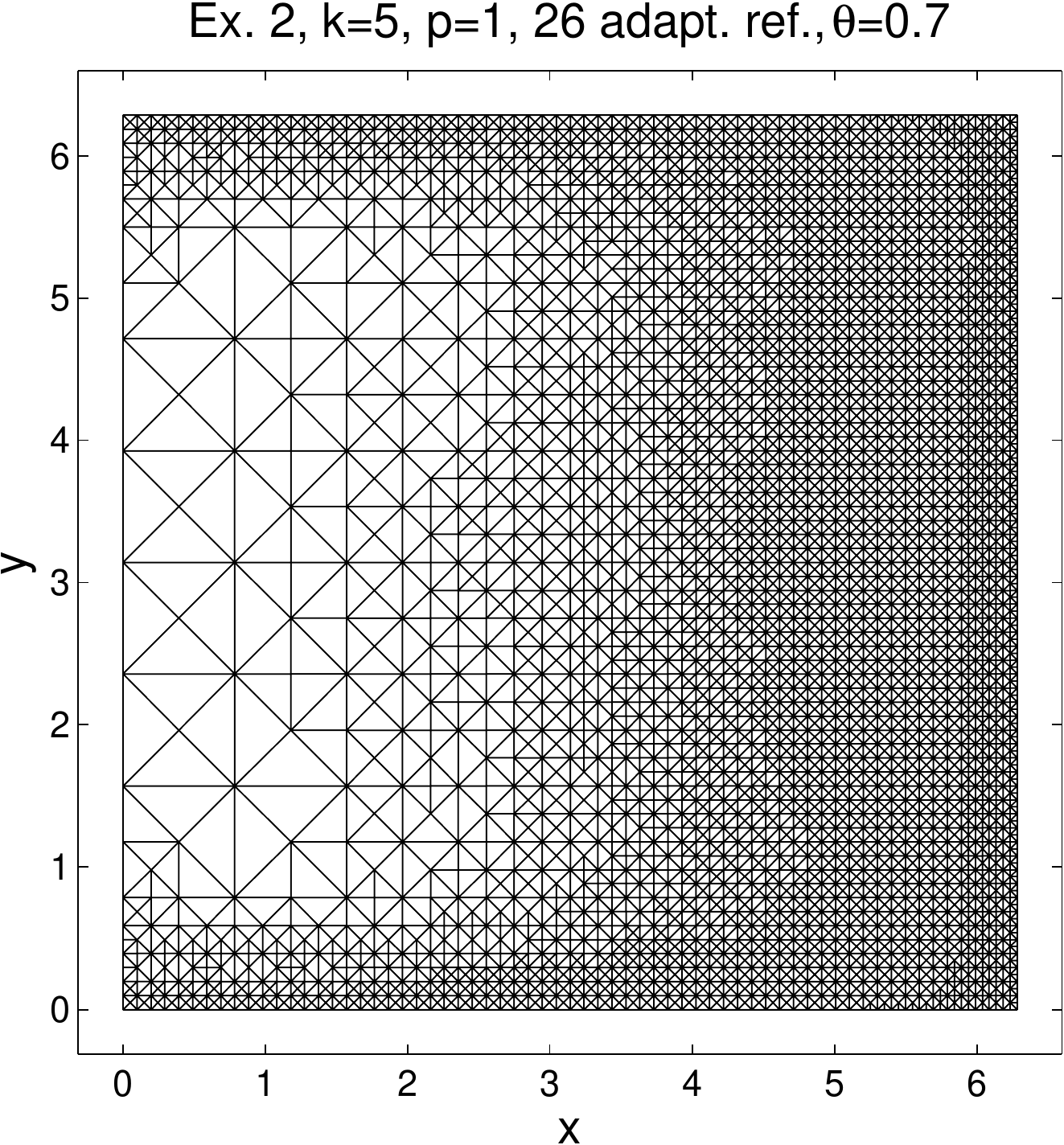}}
  \hfill \subfloat[$\Im (u_{\mathcal{T}})$ for a uniformly refined
  mesh]{\includegraphics[width=0.521\textwidth]{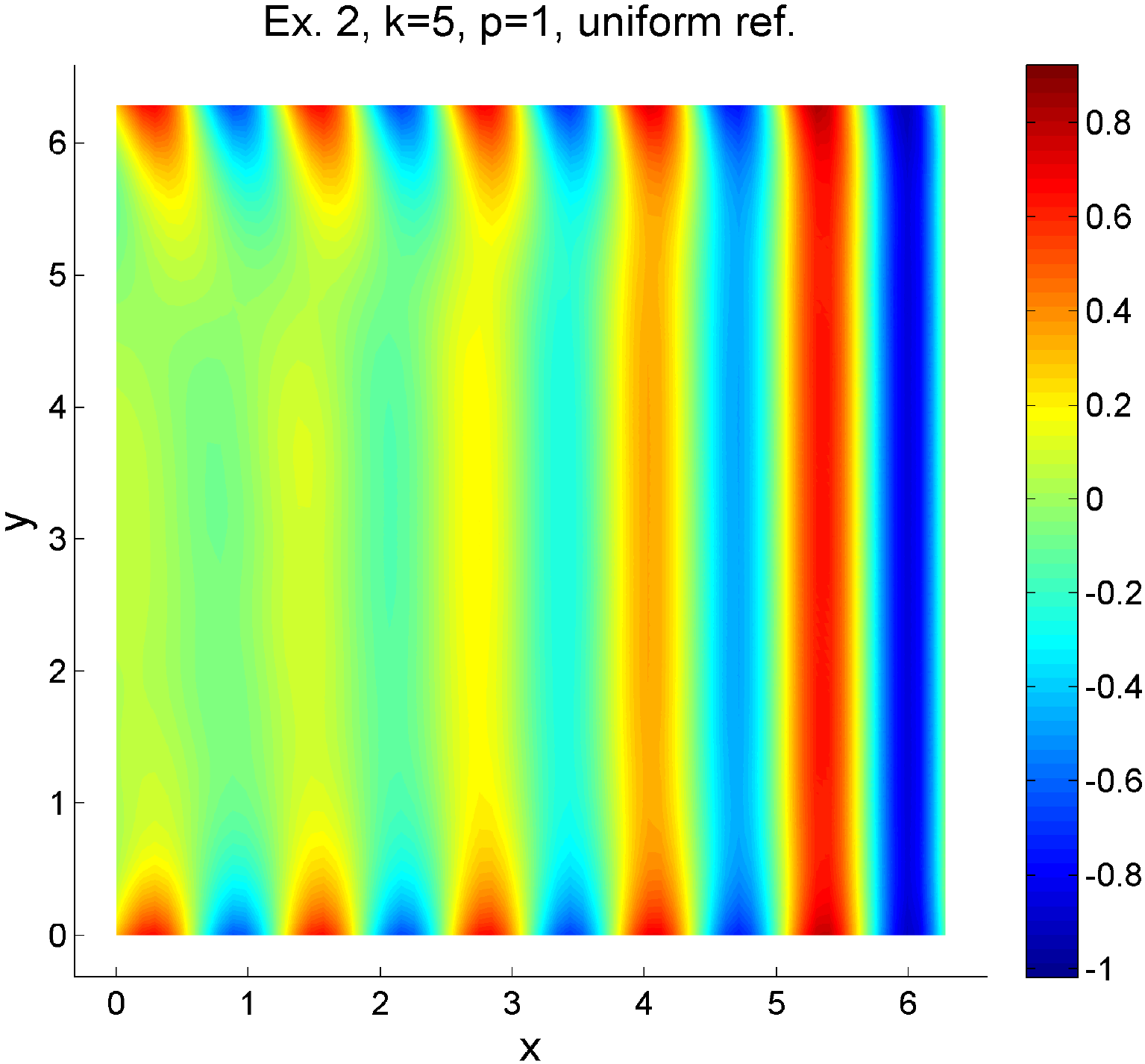}}

  \caption[Adaptive mesh and DGFEM solution, Ex. 2]{Adaptive mesh and
    imaginary part of the DGFEM solution for a uniform mesh with large
    mesh width, $k=5$, and $p=1$ in Example 2. The exact solution is
    $u(x,y)=\exp(\operatorname*{i} kx)$, and therefore $\Im(u(x,y)) =
    \sin(kx)$.}
  \label{fig:fig4}
\end{figure}

\begin{figure}[t]
  \subfloat[In this plot $|e|$ denotes the length of the edge $e$.
  The plot shows the maximum length of an edge, the minimum length of
  an edge, and the ratio for the $j$-th adaptively refined mesh.]
  {\includegraphics[width=0.475\textwidth]{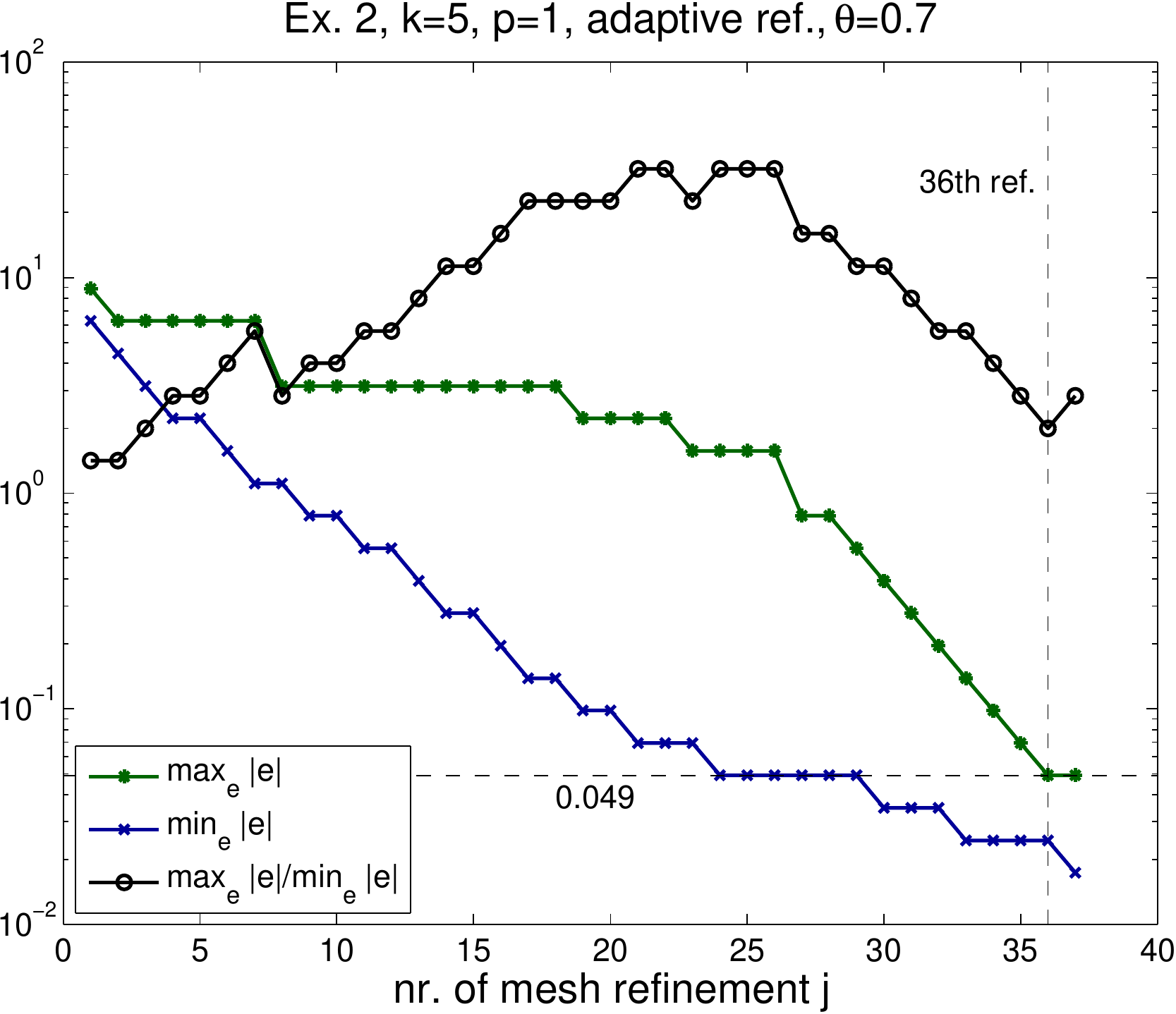}} \hfill
  \subfloat[Error $\Vert
  u-u_{\mathcal{T}}\Vert_{\mathcal{H};\mathcal{T}}$ for uniform and
  adaptive refinement (with $\theta=0.7$) and the estimated error
  $\check \eta(u_{\mathcal{T}})$ for adaptive refinement]
  {\includegraphics[width=0.49\textwidth]{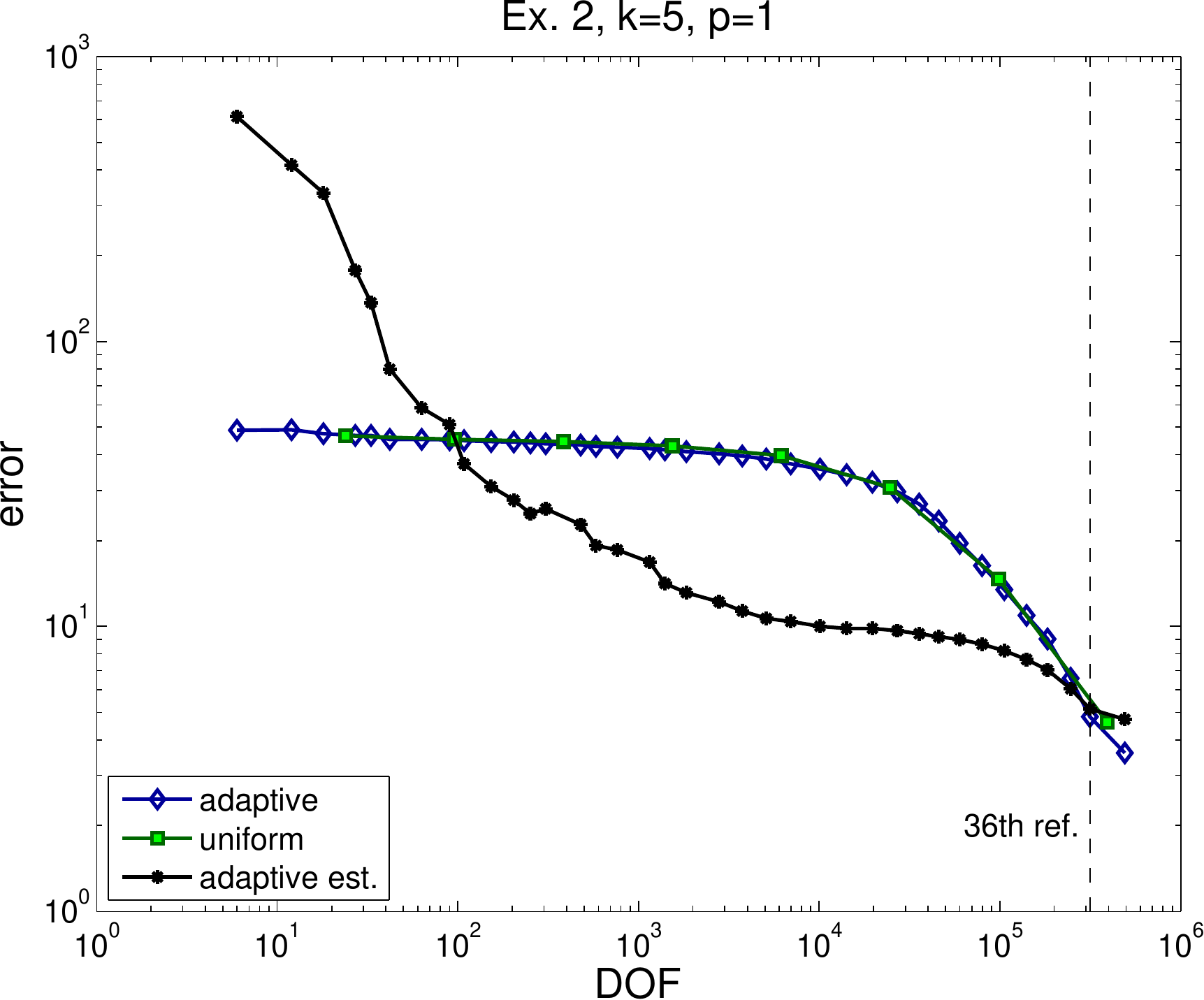}}
  \caption[Critical mesh width, Ex. 2]{In Fig. (a) it can be seen that
    the adaptive algorithm, applied to Example 2 with $k=5$ and $p=1$,
    at first generates a mesh with very diverse element sizes, which
    then turns into an almost uniform mesh at about the 36th
    refinement. This refinement corresponds to a maximum edge length
    of $0.049$. In both plots, the dashed line marks this mesh width,
    respectively the point at which this adaptive refinement takes
    place. We observe that convergence for uniform refinement starts
    shortly before this mesh size is reached. Moreover, at this
    refinement, the error estimator surpasses the actual error in this
    example, and the error is underestimated in the preasymptotic
    range.}%
  \label{fig:fig5}
\end{figure}

Furthermore we emphasize the following two points.

\begin{enumerate}
\item[a.] As is well-known reliability is not a local property and we
  have here an example where the local error indicator
  $\check{\eta}_{K}$ differs significantly from the local error in the
  left part of the domain in the pre-asymptotic regime. In addition,
  $M_{\frac{\operatorname*{kh}%
    }{\operatorname*{p}}}$ is large and due to Remark \ref{Rempoll}.c
  the underestimation of the error in this early stage of refinement
  can be explained. This behavior is illustrated in
  Fig. \ref{fig:fig5}(b).%
\item[b.] It is also worth mentioning that we start the adaptive
  discretization with a very coarse initial mesh where the resolution
  condition (\ref{rescond}) is not fulfilled for a moderate constant
  $C_{0}$. The numerical experiments indicate that the adaptive
  process behaves robustly for the $\operatorname*{dG}$-formulation
  already in the pre-asymptotic regime.
\end{enumerate}

\subsection{Example 3: L-shaped Domain}

In this example we consider the L-shaped domain $\Omega:=\left(
  -1,1\right) ^{2}\backslash\left( \left[ 0,1\right] \times\left[
    -1,0\right] \right) $ with right-hand sides $f$ and $g$ chosen
such that the first kind Bessel function $u\left( x,y\right)
:=J_{1/2}\left( kr\right) $ with $r:=\sqrt{x^{2}+y^{2}}$ is the exact
solution (see also \cite{HoppeHelmholtz2013}). The Bessel function and
solution $u$ are plotted in Fig. \ref{fig:fig6}. The problem is chosen
such that the solution has a singularity at the reentrant corner
located at $\mathbf{0}$.

In Fig. \ref{fig:fig7}, two meshes generated by the adaptive procedure
are depicted for uniform polynomial degree $p=1$ and wavenumber
$k=10$. The oscillating nature of the solution as well as the singular
behavior is nicely reflected by the distribution of the mesh cells.

In Fig. \ref{fig:fig8}, we compare uniform with adaptive mesh
refinement for different values of $k$ and $p$. As expected the
uniform mesh refinement results in suboptimal convergence rates while
the optimal convergence rates are preserved by adaptive refinement for
the considered polynomial degrees $p=2,4$. In both cases some initial
refinement steps are required before the error starts to decrease due
to the pollution effect. Again the pollution is significantly reduced
for higher polynomial degree.

\begin{figure}
  \subfloat[$J_{1/2}(10 \Vert\vec{x}\Vert)$]%
  {\includegraphics[width=0.49\textwidth]{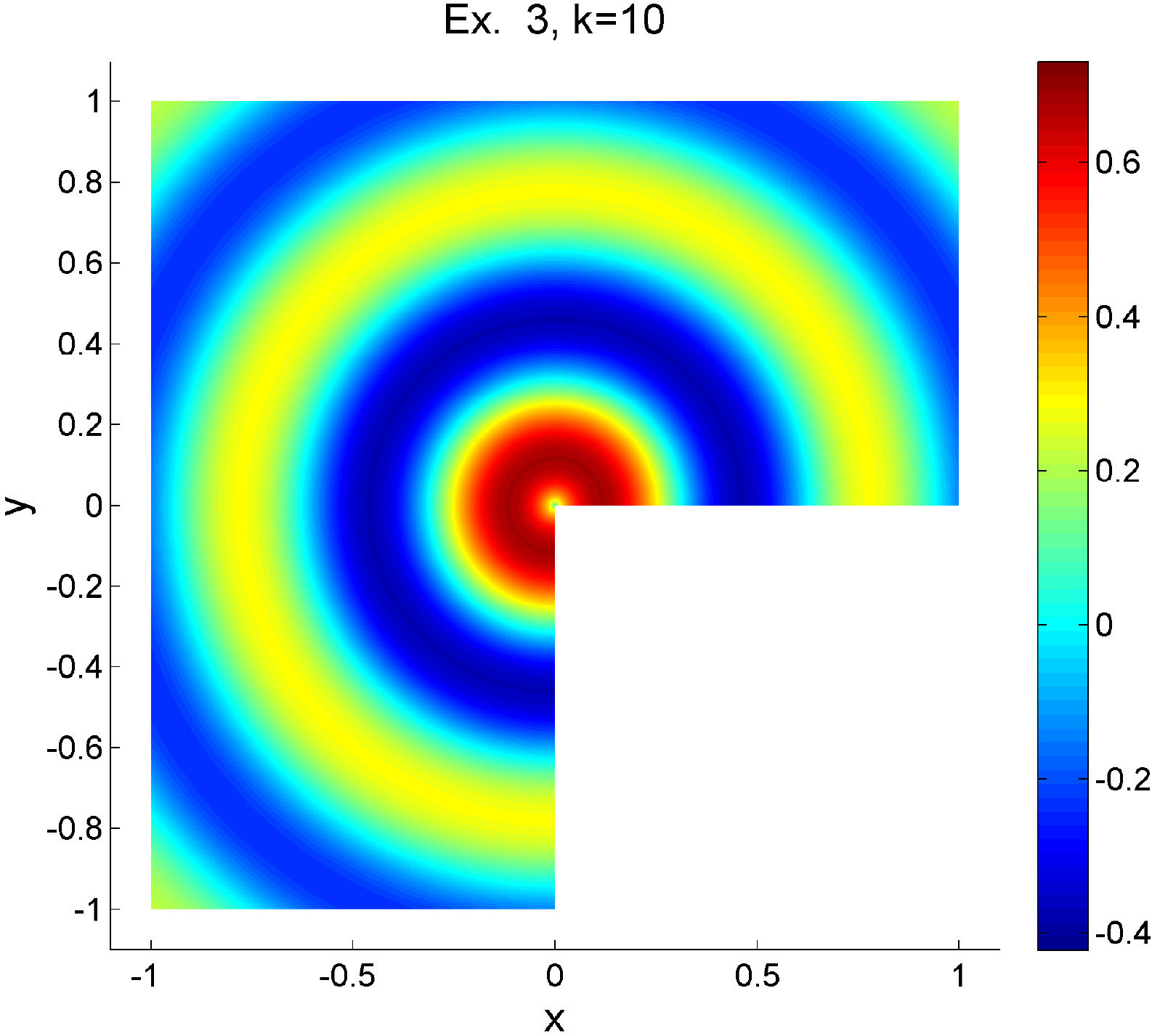}} \hfill
  \subfloat[Bessel function $J_{1/2}$ of the first kind]
  {\includegraphics[width=0.49\textwidth]{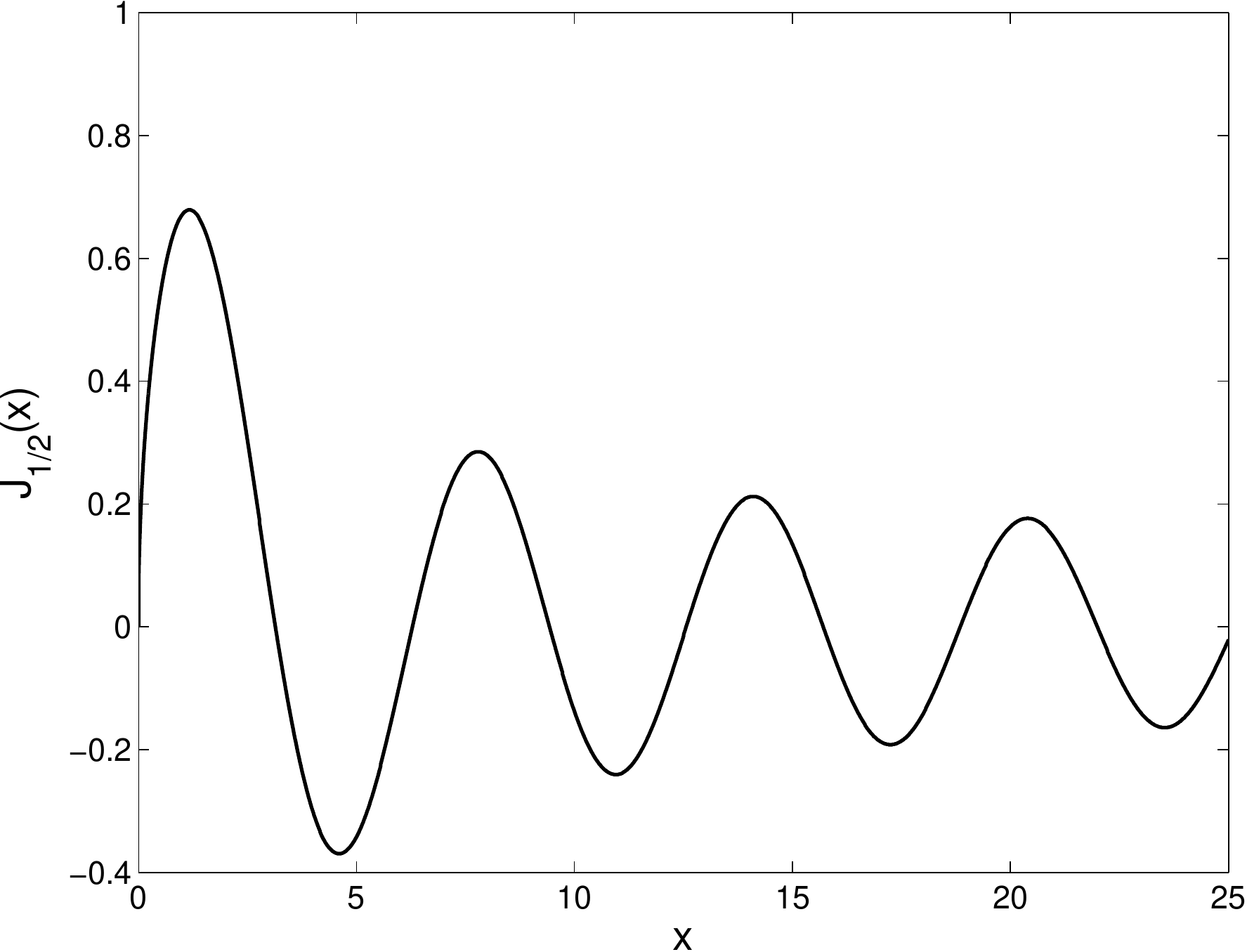}}
  \caption[Solution, Ex. 3]{The solution $u=J_{1/2}(kr)$ in Example 3
    for $k=10$, and the Bessel function $J_{1/2}(x)$, whose derivative
    goes to infinity for $x \to 0$.}\label{fig:fig6}
\end{figure}
\begin{figure}
  \subfloat[Mesh after 52 refinements, 677 elements, $\min_e h_e =
  0.00024$] {\includegraphics[width=0.48\textwidth]{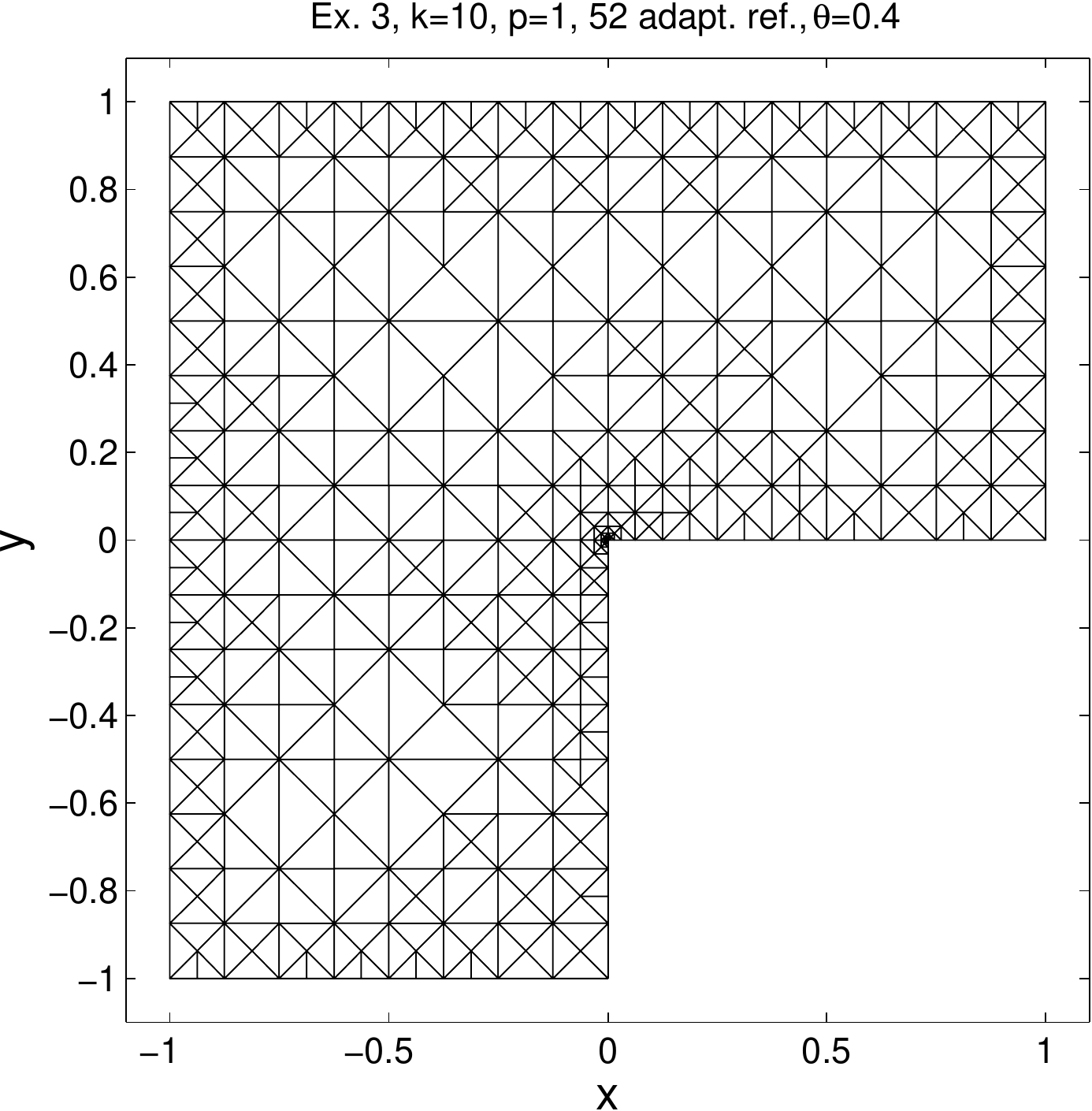}} \hfill
  \subfloat[Mesh after 66 refinements, 3677 elements, $\min_e h_e =
  0.000022$] {\includegraphics[width=0.48\textwidth]{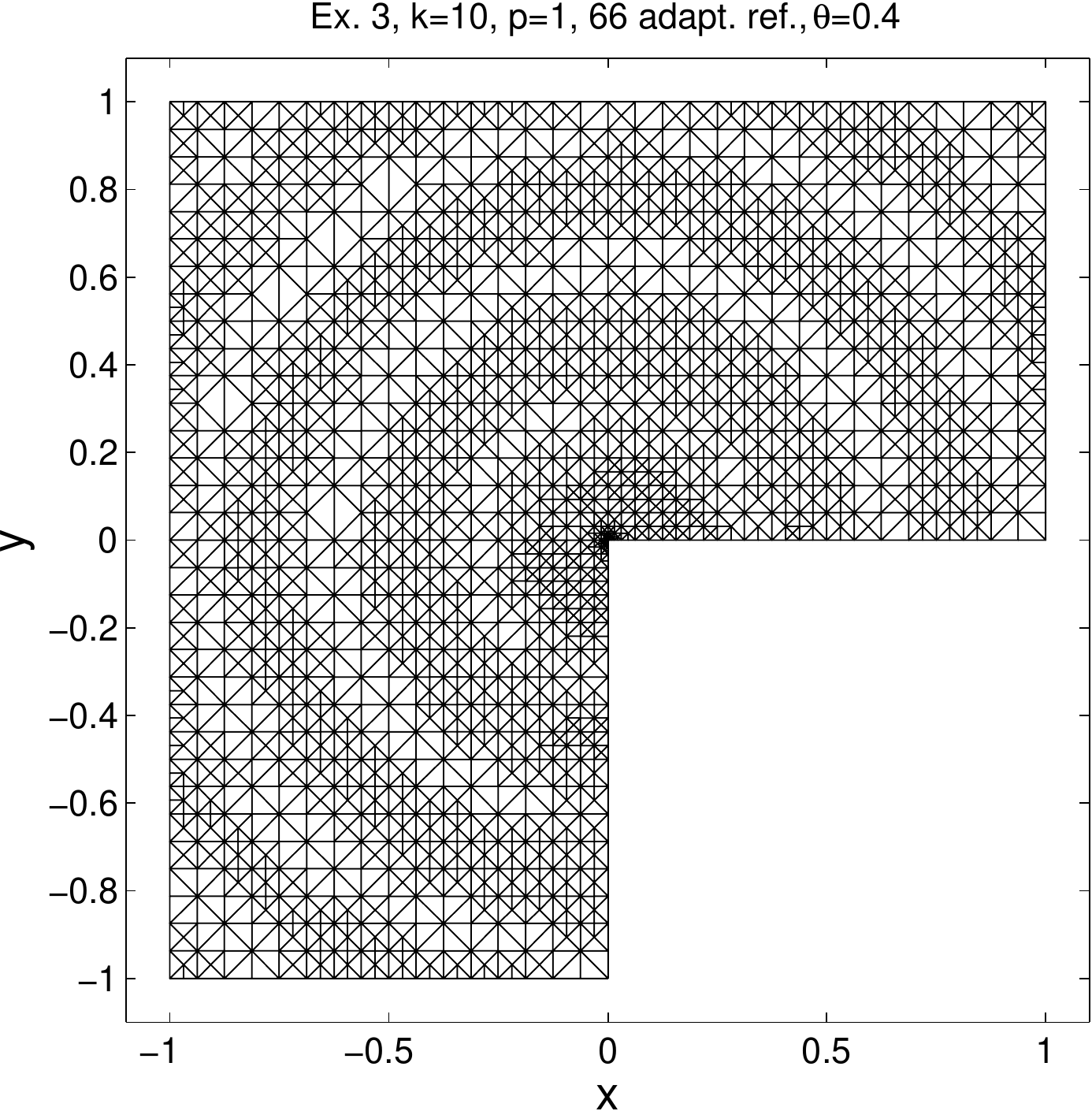}}
  \caption[Adaptive meshes, Ex. 3]{Meshes obtained by the adaptive
    algorithm for Example 3.}\label{fig:fig7}
\end{figure}

\begin{figure}
  \subfloat[$k=5$,
  $p=2$]{\includegraphics[width=0.49\textwidth]{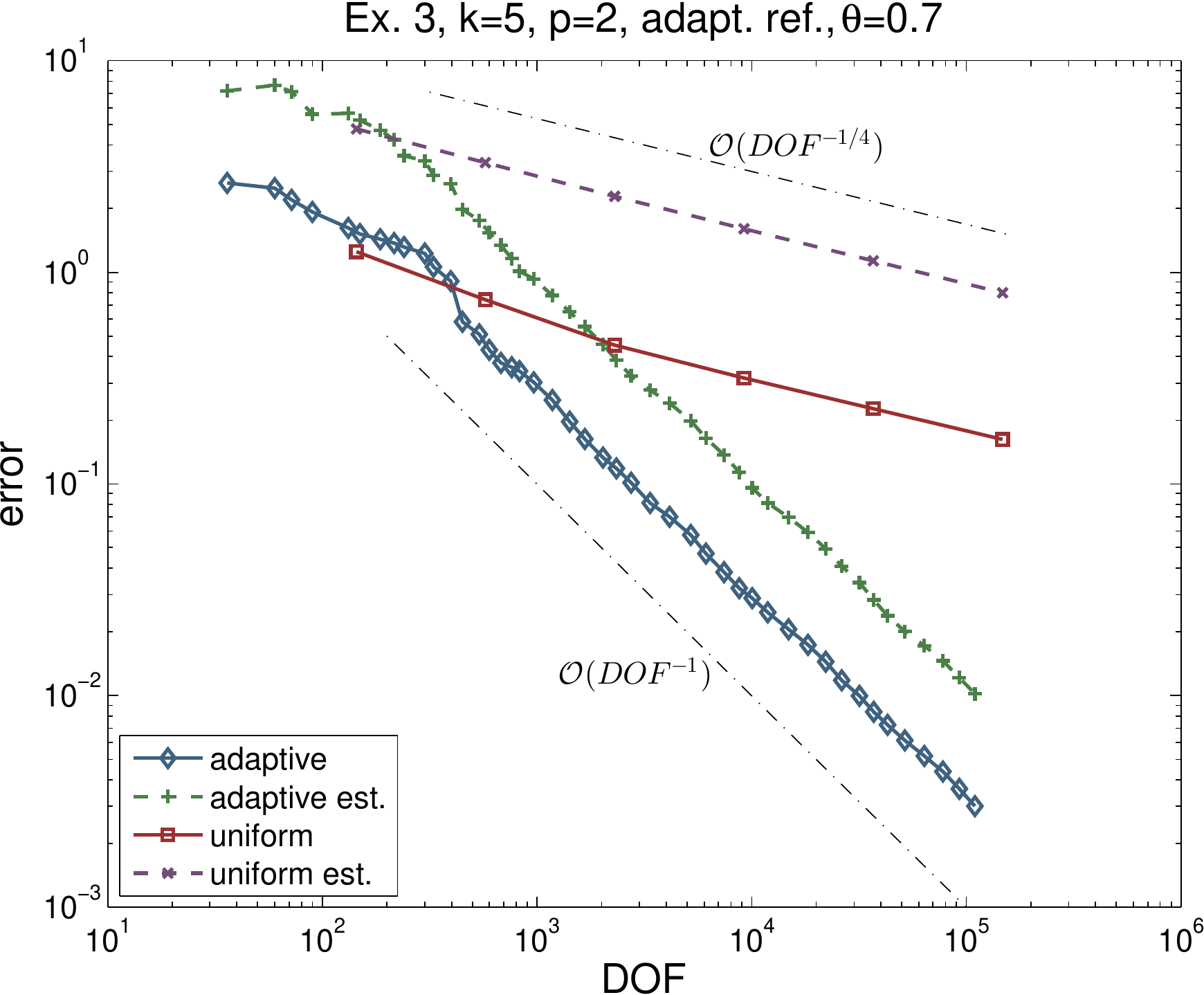}} \hfill
  \subfloat[$k=10$,
  $p=2$]{\includegraphics[width=0.49\textwidth]{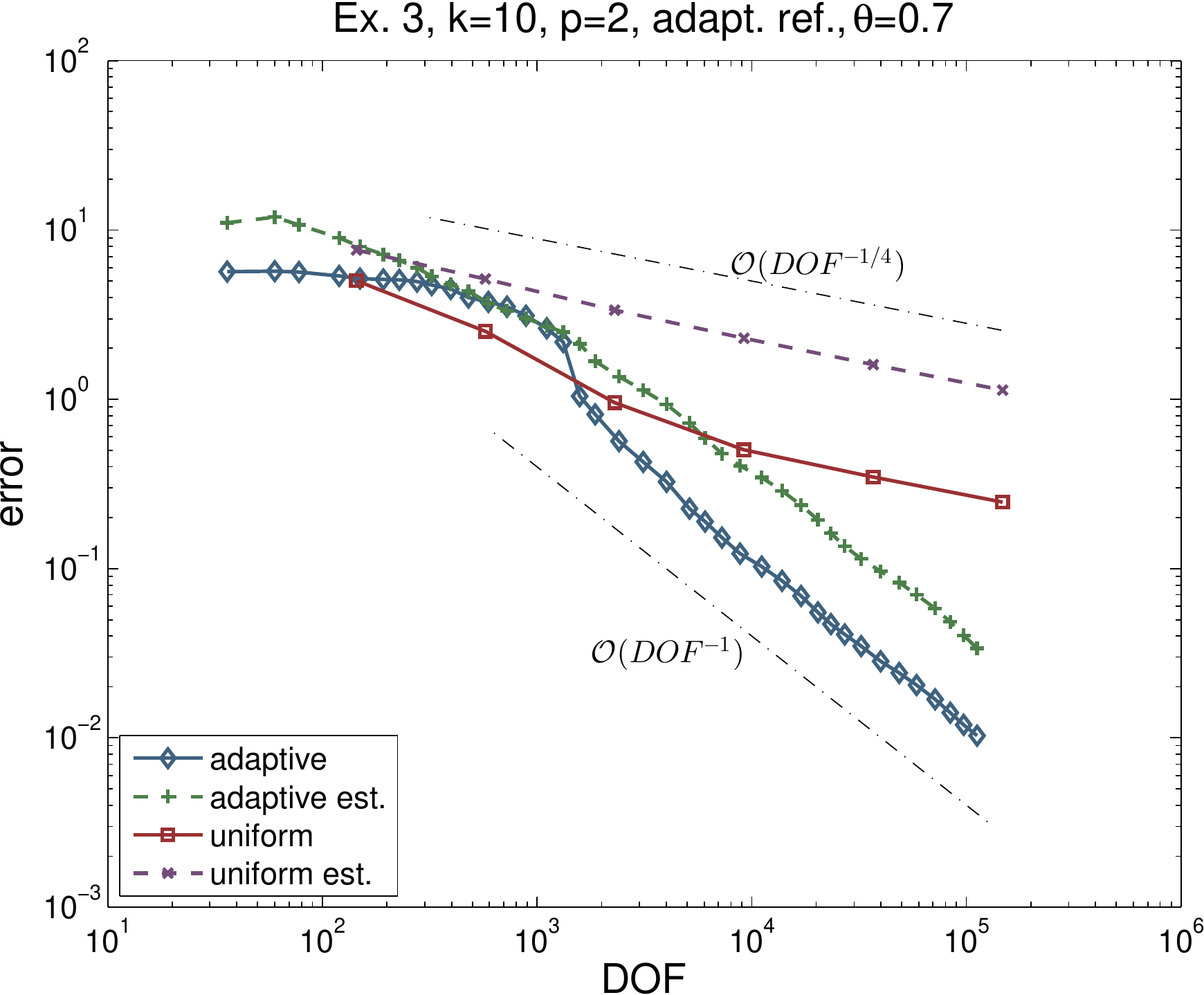}}

  \subfloat[$k=10$,
  $p=4$]{\includegraphics[width=0.49\textwidth]{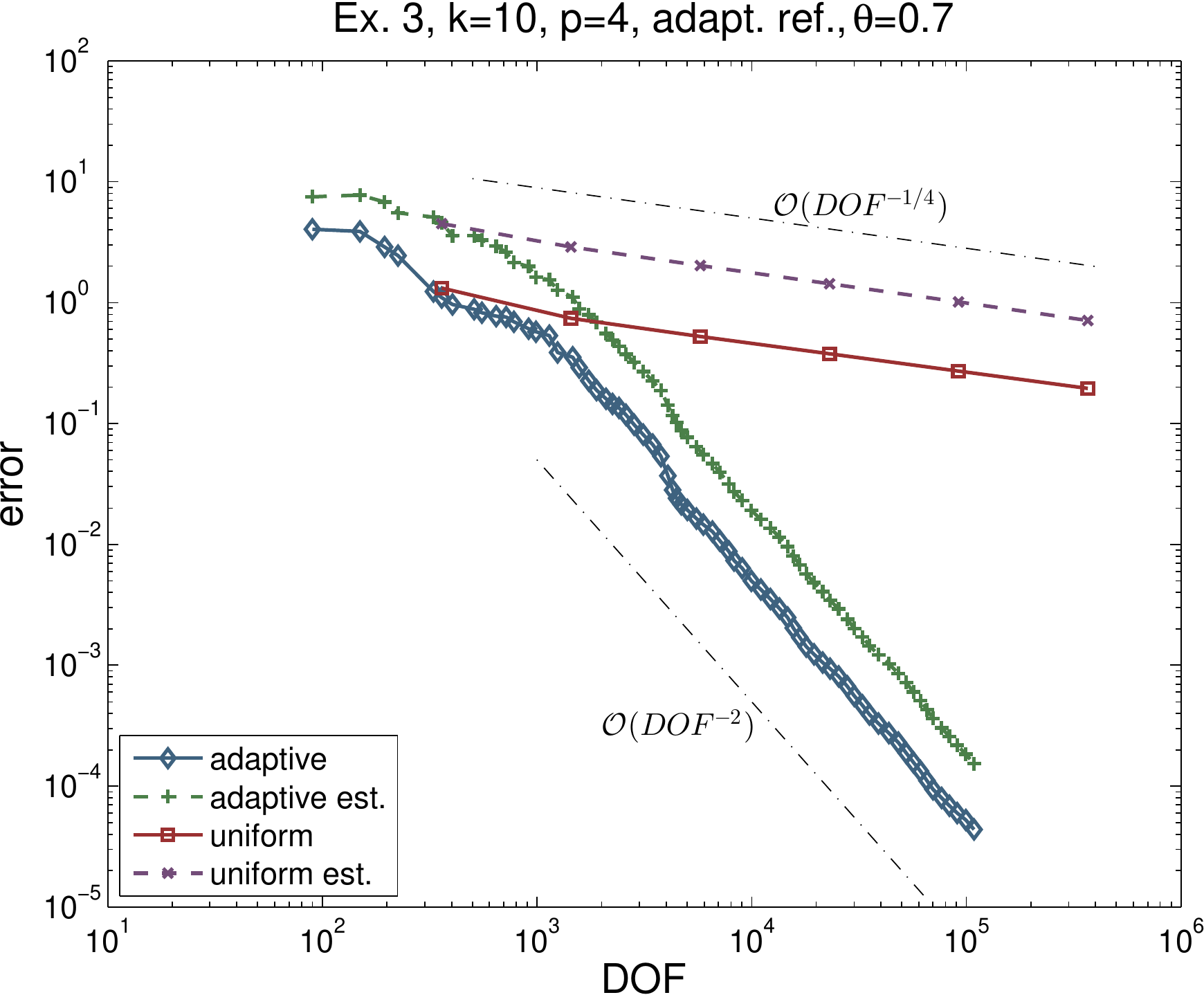}} \hfill
  \subfloat[$k=40$,
  $p=4$]{\includegraphics[width=0.49\textwidth]{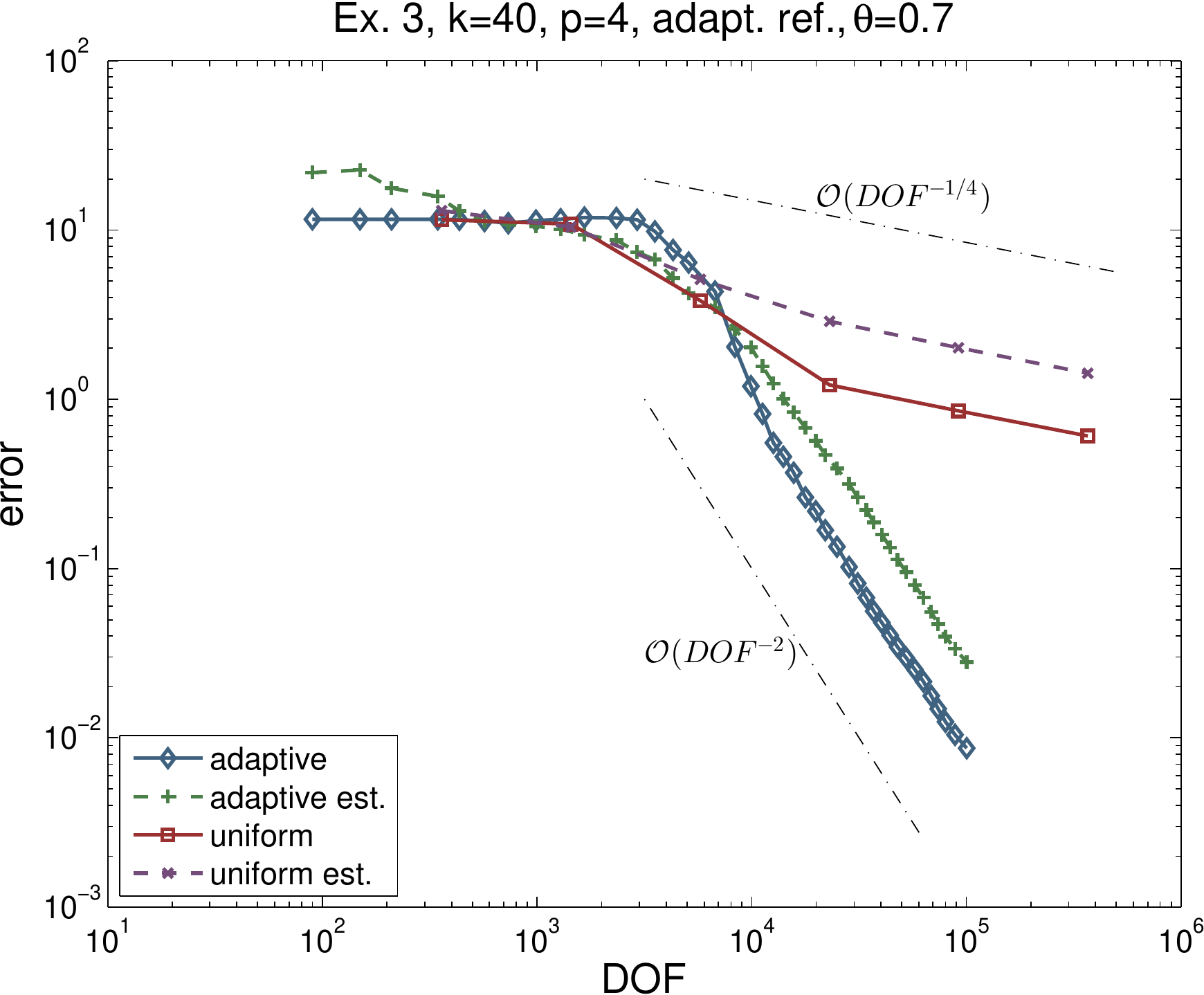}}
  \caption[Uniform vs. adaptive ref., Ex. 3]{Comparison of the actual
    error $\Vert u-u_{\mathcal{T}%
    }\Vert_{\mathcal{H};\mathcal{T}}$ and the estimated error
    $\check{\eta }(u_{\mathcal{T}} )$, using uniform and adaptive
    refinement with $\theta= 0.7$ in Example 3 for different values of
    $k$ and $p$.}
  \label{fig:fig8}
\end{figure}

\subsection{Example 4: Non-constant Wavenumber}

In this section, we consider the case of non-constant wavenumber $k$
which has important practical applications. Although we have
formulated the $\operatorname*{dG}$-method for non-constant wavenumber
our theory only covers the constant case. Nonetheless the numerical
experiments indicate that the a posteriori error estimation leads to
an efficient adaptive solution method.

Consider the domain $\Omega=\left( 0,2\pi\right) ^{2}$. We partition
$\Omega$ into the disc $\Omega_{1}$ about $\left( \pi,\pi\right)
^{\top} $ with radius $3/2$ and its complement
$\Omega_{2}:=\Omega\backslash\Omega_{1}$.  Let $k_{1},k_{2}>0$. The
function $k$ is defined piecewise by $\left.  k\right\vert
_{\Omega_{i}}:=k_{i}$, $i=1,2$. We have chosen $f=0$ and
\begin{equation}
  g_{1}\left(  x,y\right)  :=\left\{
    \begin{array}
      [c]{cl}%
      -1 & x=0,\\
      \operatorname*{i} & x=2\pi,\\
      0 & \text{otherwise},
    \end{array}
  \right.  \qquad\forall\left(  x,y\right)  \in\partial\Omega. \label{5.4.1}%
\end{equation}
Alternatively we will consider boundary data as defined in
(\ref{defg2}) with $k:=k_{2}$ and denote them here by $g_{2}$.

In Fig. \ref{fig:fig11}, the adaptively refined mesh and the real part
of the $\operatorname*{dG}$-solution are plotted for $k_{1}=10$,
$k_{2}=1$, and the boundary data $g_{1}$. Strong refinement takes
place in the vicinity of the circular interface between $\Omega_{1}$
and $\Omega_{2}$. Moreover, the mesh width is much smaller inside the
circle, where the wavenumber is high in accordance with the smoothness
properties of the solution.

In the next example we have considered the reversed situation:
$k_{1}=1$, $k_{2}=10$, and boundary data $g_{2}$. Fig. \ref{fig:fig12}
implies, that strong refinement close to the jump of the wavenumber is
not always necessary.  In this case, the solution appears to be
smooth, respectively almost zero near the left part of the inner
circle where $k=k_{2}$ holds and this is taken into account by the
adaptive algorithm. Fig. \ref{fig:fig15} reflects the convergence of
the estimated error for Example 4.

\begin{figure}
  \subfloat[Mesh after 28 refinements, 6076 elements, $\min_e h_e =
  0.012$] {\includegraphics[width=0.45\textwidth]{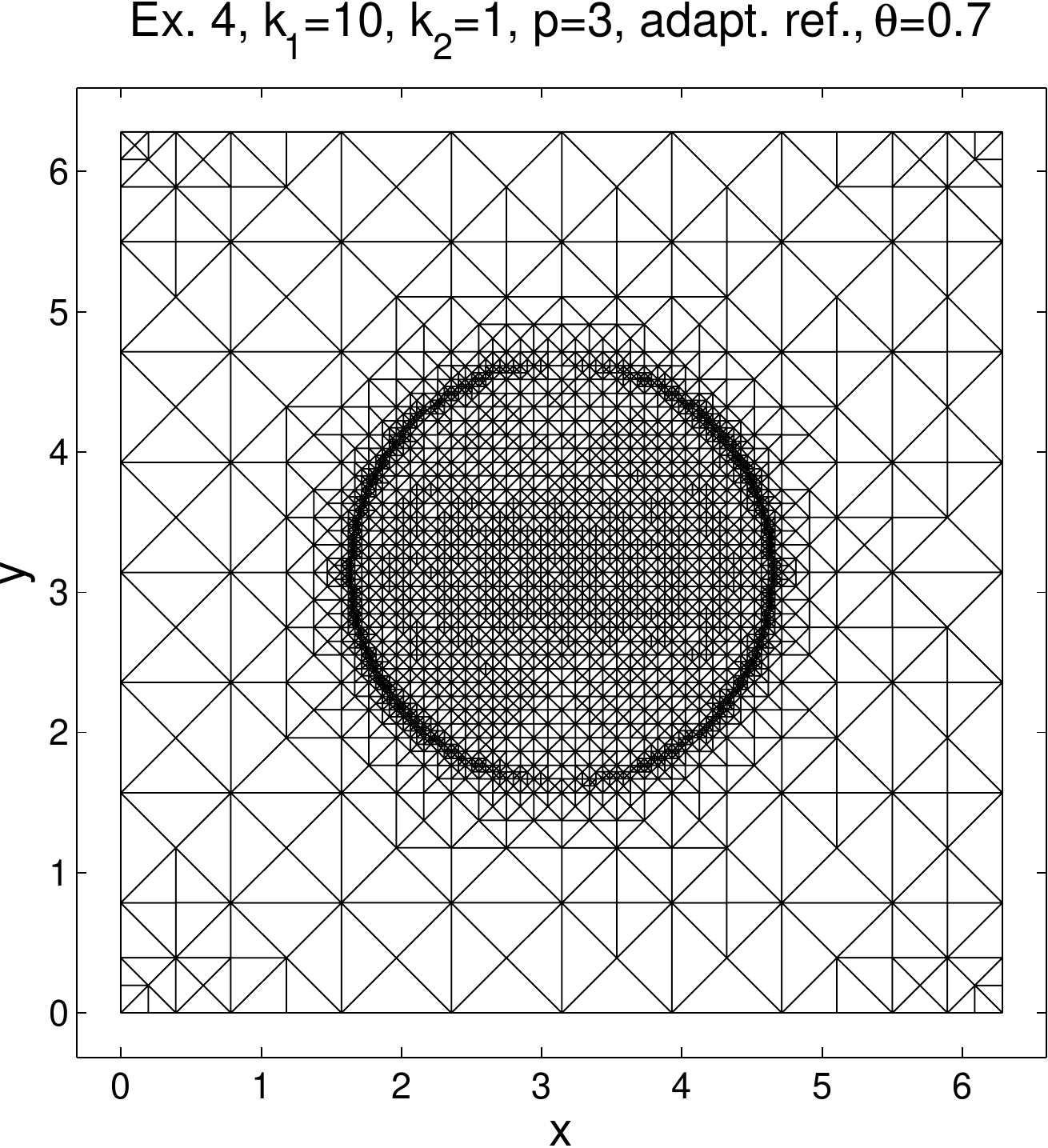}} \hfill
  \subfloat[$\Re (u_{\mathcal{T}}
  )$]{\includegraphics[width=0.521\textwidth]{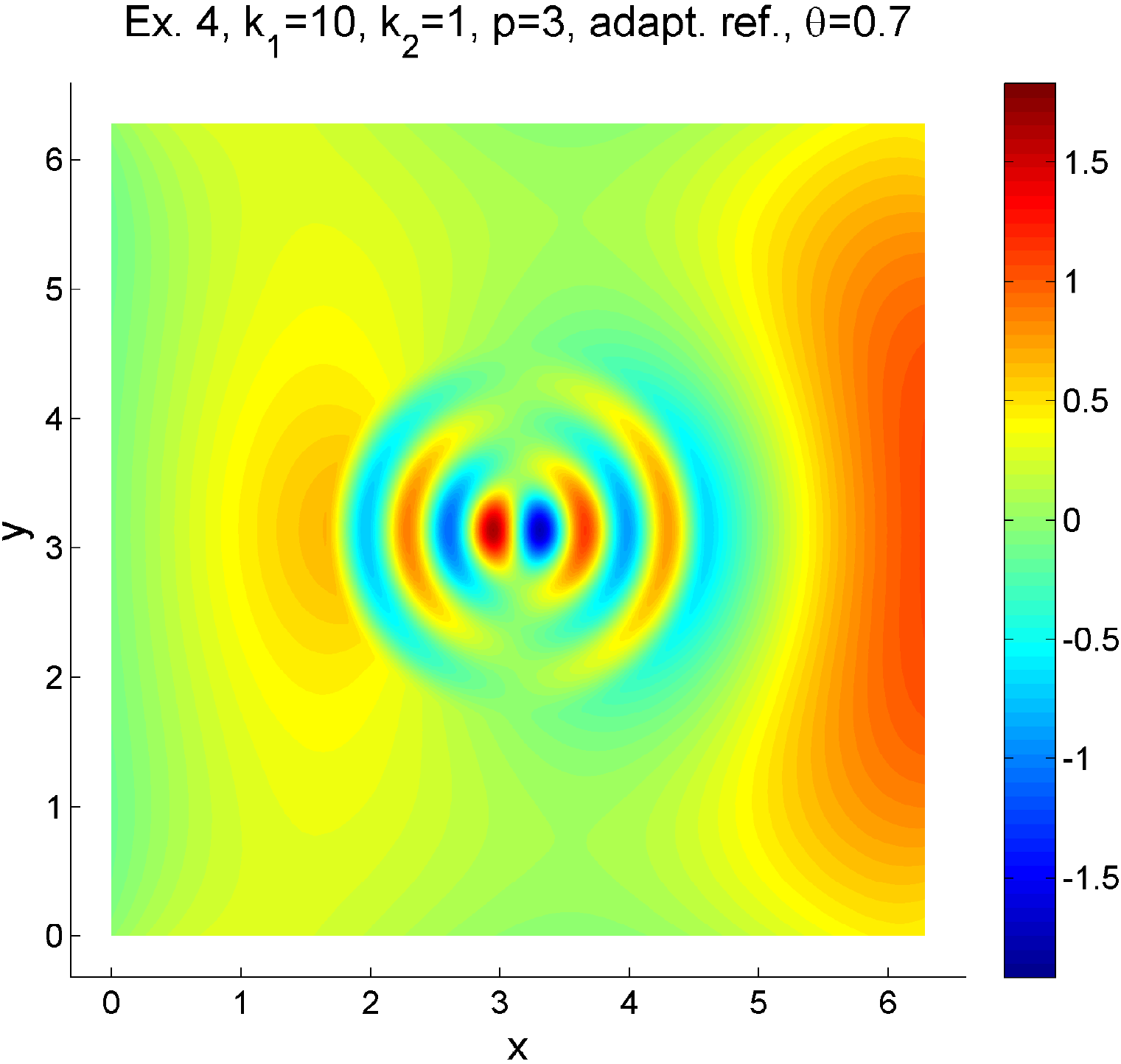}}
  \caption[Mesh and DGFEM solution, Ex. 4.1]{Adaptively refined mesh
    with $\theta=0.7$ and real part of the DGFEM solution on this mesh
    for Example 4 with $k_{1}=10$, $k_{2}=1$, and the boundary data
    $g_{1}$.}
  \label{fig:fig11}
\end{figure}

\begin{figure}
  \subfloat[Mesh after 29 refinements, 5892 elements, $\min_e h_e =
  0.034$] {\includegraphics[width=0.45\textwidth]{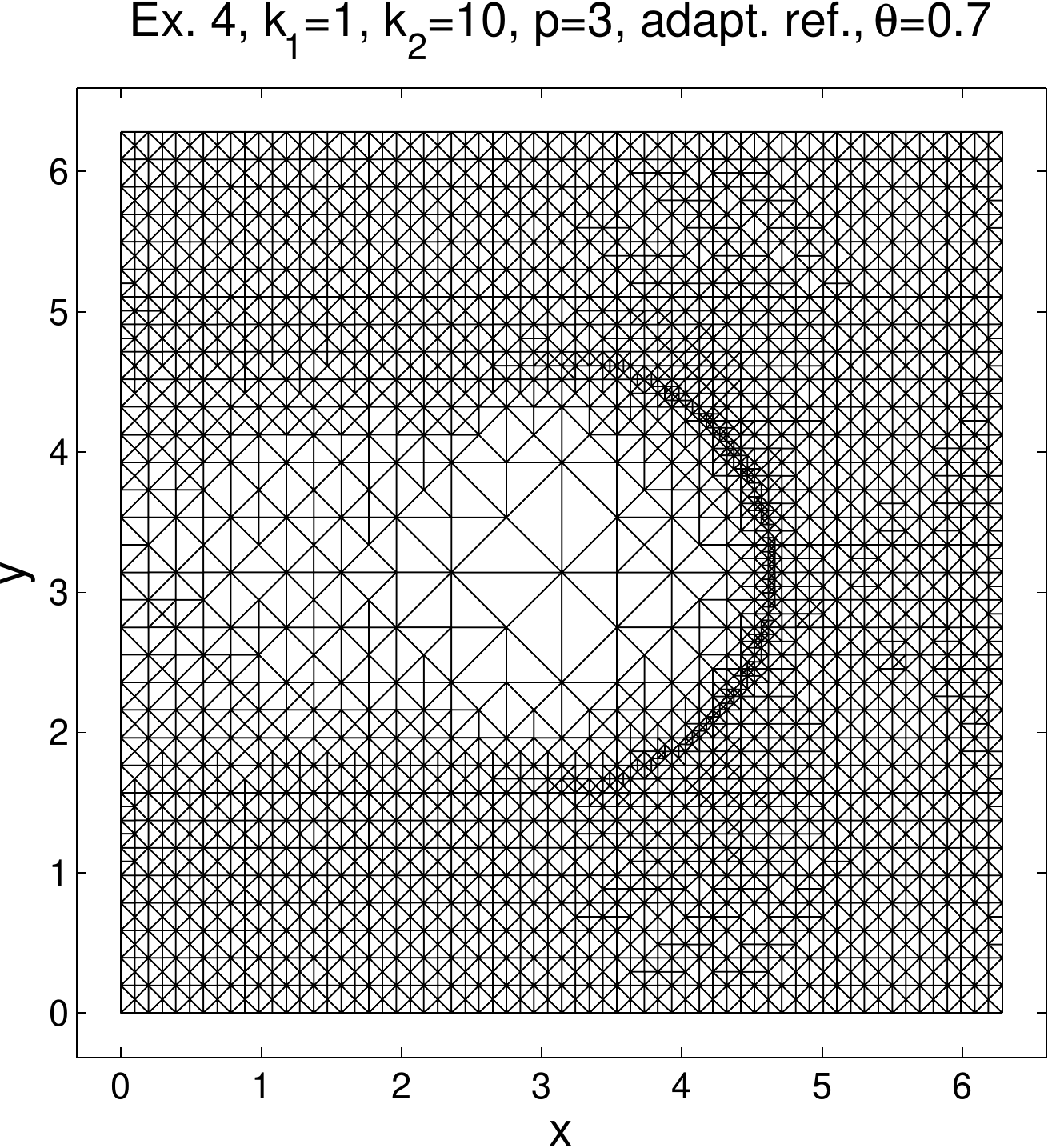}} \hfill
  \subfloat[$\Re (u_{\mathcal{T}}
  )$]{\includegraphics[width=0.521\textwidth]{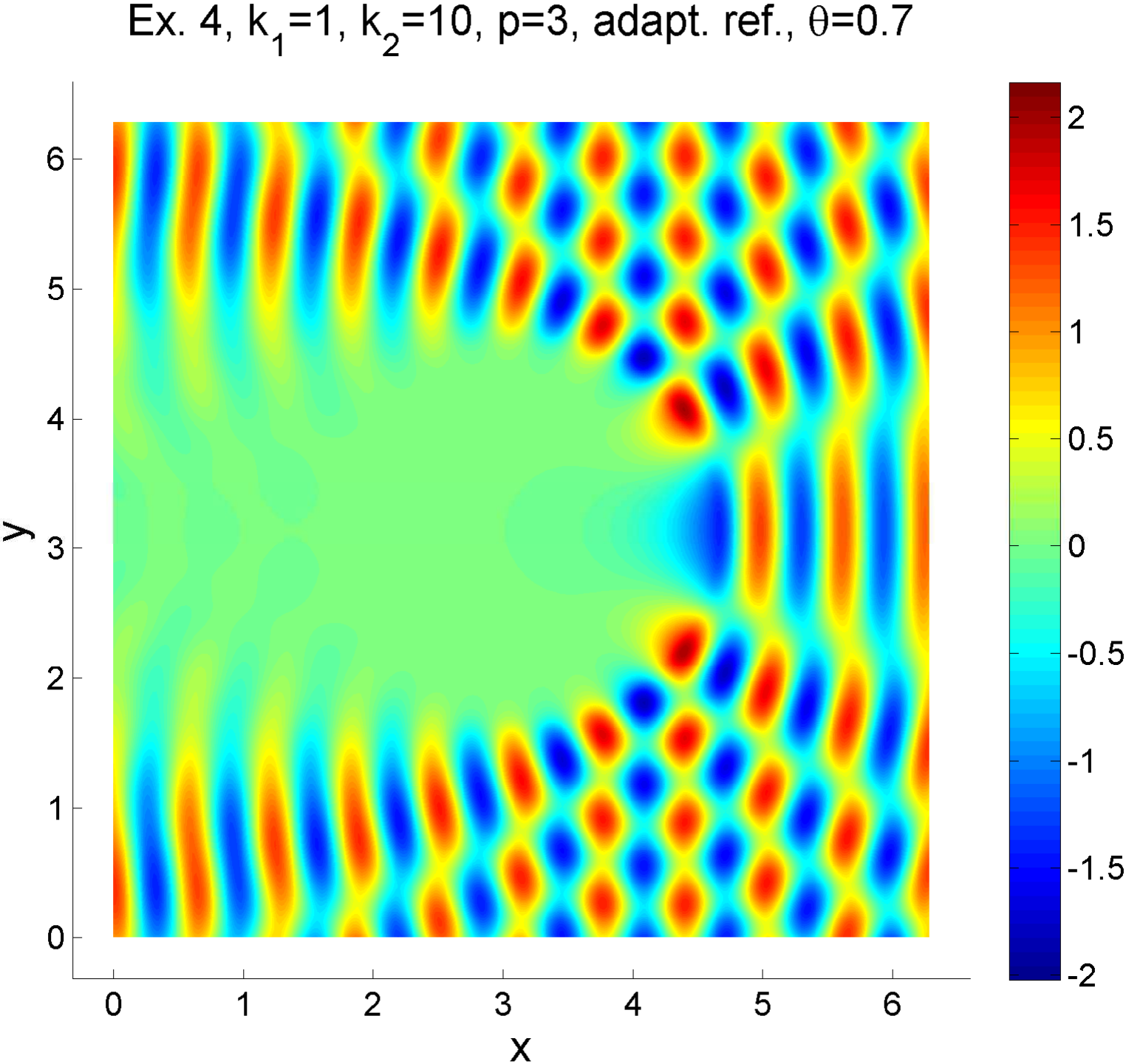}}
  \caption[Mesh and DGFEM solution, Ex. 4.2]{Adaptively refined mesh
    with $\theta=0.7$ and real part of the DGFEM solution on this mesh
    for Example 4 with $k_{1}=1$, $k_{2}=10$, and the boundary data
    $g_{2}$.}
  \label{fig:fig12}
\end{figure}

\begin{figure}[t]
  \subfloat[Boundary data $g_1$, $k_1=10$, $k_2=1$]
  {\includegraphics[width=0.49\textwidth]{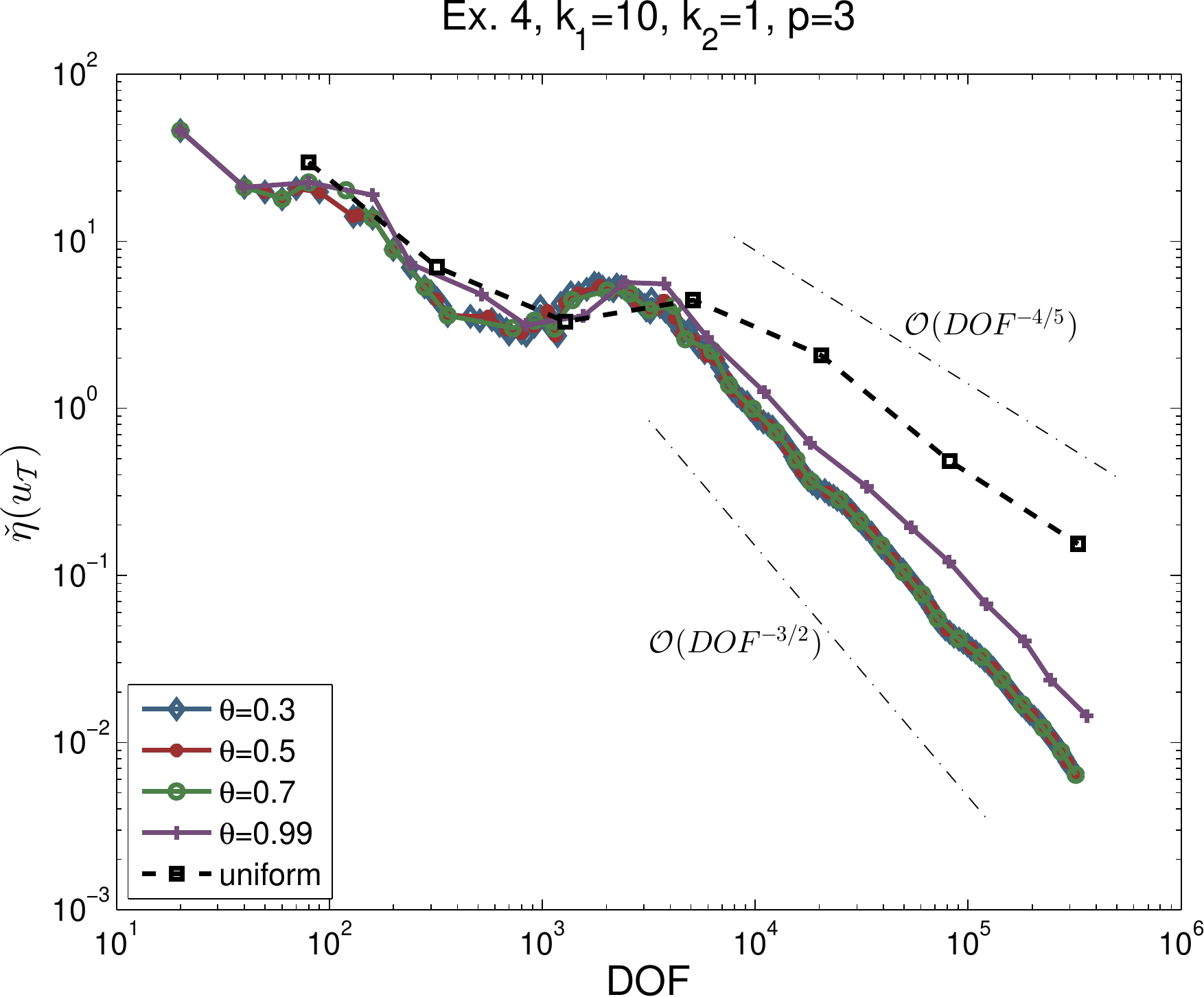}} \hfill
  \subfloat[Boundary data $g_2$, $k_1=1$, $k_2=10$]
  {\includegraphics[width=0.49\textwidth]{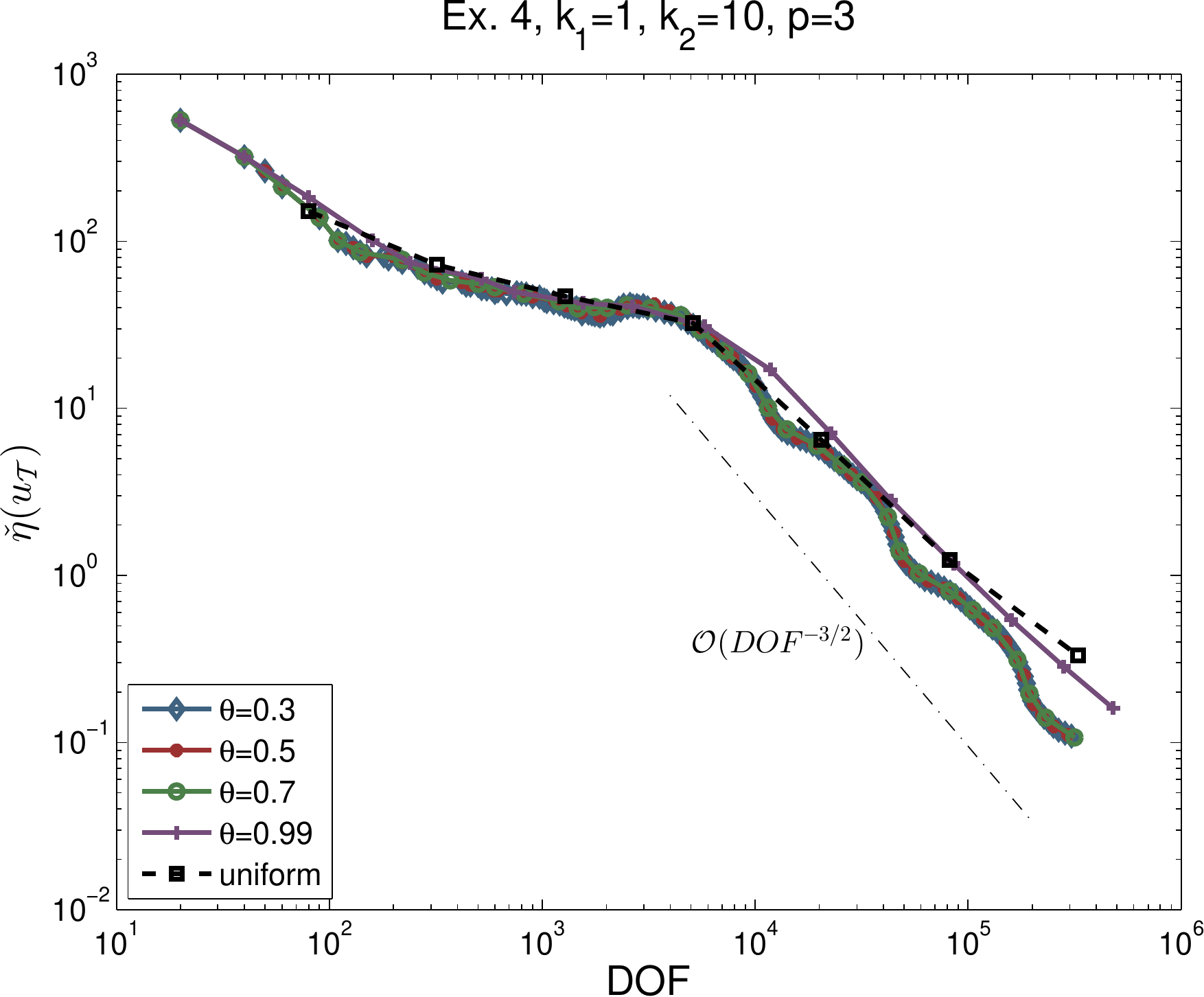}}
  \caption[Estimated convergence, Ex. 4] {The convergence of the error
    estimator $\check{\eta }(u_{\mathcal{T}} )$ in Example 4 for two
    non-constant functions $k(x,y)$ and the boundary data $g_{1}$,
    $g_{2}$, respectively.}
  \label{fig:fig15}
\end{figure}

These examples indicate that the adaptive algorithm, applied with the
error estimator $\check{\eta}_{K}$, properly accomplishes the task of
refining the mesh according to the properties of the solution:
Singularities and wave characteristics are recognized by the
estimator, and we observed optimal convergence rates.

\section{Conclusion and Outlook}

In this paper we have derived an a posteriori error estimator for an
$hp$-$\operatorname*{dG}$ method for highly indefinite Helmholtz
problems. In contrast to the discretization of the standard
variational formulation of the Helmholtz problem, the chosen
$hp$-$\operatorname*{dG}$ discretization always has a unique solution
(cf. Remark \ref{RemUniqueSolvability}). We have proved reliability
and efficiency estimates which are explicit in the discretization
parameters $h$, $p$, and the wave number $k$. Note that the adjoint
approximation property $\sigma_{k}^{\ast}\left( S_{\mathcal{T}}%
  ^{\mathfrak{p}}\right) $ enters the reliability estimate. In
\cite[Thm.  2.4.2]{ZechMaster} and \cite{MelenkSauterMathComp, MPS13}
it has been proved that for convex polygonal domains the conditions%
\begin{equation}
  p\geq C_{0}\log k\quad\text{and\quad}\forall K\in\mathcal{T}:\quad\frac
  {kh_{K}}{p_{K}}\leq C_{1}\label{rescond2}%
\end{equation}
imply $\sigma_{k}^{\ast}\left( S_{\mathcal{T}}^{\mathfrak{p}}\right)
\leq C_{2}$. We expect that general polygonal domains can be handled
by a) generalizing the \textquotedblleft decomposition
lemma\textquotedblright%
\ \cite[Theorem 4.10]{mm_stas_helm2} to a weighted $H^{2}\left(
  \Omega\right) $-regularity estimate for the non-analytic part of the
adjoint solution and b) performing an appropriate mesh grading towards
reentrant corners as is well known for elliptic boundary value
problem. Then, the error estimate for the non-analytic part can be
derived in a similar fashion as the estimate of $\eta_{\mathcal{A}}$
in the proof of \cite[Proposition 5.6]{mm_stas_helm2}. Again, we
expect that the resolution condition (\ref{rescond2}) remains
unchanged while the constant $C_{1}$ then, possibly, depends on the
angles at the reentrant corners of the polygon. Whereas the rigorous
derivation of such estimates is a topic of future research, we point
out that the use of adaptive methods is already justified in our model
setting of convex polygons, since higher polynomial degrees require
graded meshes also at convex corners in order to preserve optimal
convergence rates.

Our analysis is not sharp enough to give precise bounds for the
constants $C_{0}$, $C_{1}$, $C_{2}$. The numerical experiments show
that these estimates are \textit{qualitatively} sharp, i.e., if the
polynomial degree stays fixed independent of $k$, the error estimator
significantly overestimates the error while a mild, logarithmic
increase depending on $k$ cures this problem. It would be also
interesting to estimate the size of $\sigma_{k}^{\ast}\left(
  S_{\mathcal{T}}^{\mathfrak{p}}\right) $ by numerical
experiments. However, this task is far from being trivial because the
adjoint approximation property is defined as an infinite-dimensional
$\sup$-$\inf$ problem, and the dependence on the regularity of the
domain, step size $h$, polynomial degree $p$, wave number $k$ requires
extensive numerical tests which would increase the length of the paper
substantially. We are planning to investigate this question as a topic
of further research. Our numerical examples indicate that, as soon as
the resolution condition is satisfied with constants $C_{0}\sim
C_{1}=\mathcal{O}\left( 1\right) $, the a posteriori error estimator
becomes quite sharp.

Another interesting question is related to the mesh grading towards
the corners of the polygonal domain.  The results in
\cite{mm_stas_helm2} imply that if the initial, coarsest mesh and
polynomial degrees are chosen according to (\ref{rescond2}) and
\cite[Assumption 5.4]{mm_stas_helm2} then, $\sigma_{k}^{\ast}\left(
  S_{\mathcal{T}}^{\mathfrak{p}}\right) $ stays bounded by a constant
during the whole adaptive process and the geometric grading may not to
be incorporated in the adaptive refinement procedure.
Our numerical experiments show that after some refinements (as soon as
the resolution condition is satisfied) the convergence rate of the
adaptive solution becomes optimal and, in addition, the error
estimator nicely reflects the size and decay of the error. This
behaviour of the estimator, which is supported by our analysis only in
case $\sigma_{k}^{\ast}\left( S_{\mathcal{T}}^{\mathfrak{p}}\right) $
that is moderate, suggests that the adaptive algorithm achieves an
appropriate mesh grading on its own.

\appendix

\section{Approximation Properties\label{AApprox}}

\subsection{$C^{1}$-$hp$ Interpolant}

For residual-type a posteriori error estimation, typically, the subtle choice
of an interpolation operator for the approximation of the error along
$hp$-explicit error estimates plays an essential role. For our
\textit{non-conforming }$\operatorname{dG}$-formulation it turns out that a
$C^{1}$-interpolation operator has favorable properties, namely, the internal
jumps vanish while the approximation estimates are preserved. In
\cite{MelenkWohlmuth} a $C^{0}$-$hp$-Cl\'{e}ment-type interpolation operator
is constructed and $hp$-explicit error estimates are derived for
$W^{1,q}\left(  \Omega\right)  $ functions. In contrast, our estimate for the
$C^{1}$-hp Cl\'{e}ment interpolation operator allows for higher-order
convergence estimates for smoother functions as well as for estimates in norms
which are stronger than the $H^{1}$-norm. The proof follows the ideas in
\cite[Thm. 2.1]{MelenkWohlmuth} and employs a $C^{1}$-partition of unity by
the quintic Argyris finite element.

The construction is in two steps. First local (discontinuous) approximations
are constructed on local triangle patches. By multiplying with a $C^{1}%
$-partition of unity the resulting approximation is in $C^{1}\left(
\Omega\right)  $, while the approximation properties are preserved.

The first step is described by the following theorem. Its proof can be found
in \cite[Thm. 5.1]{melenk_nshpinterpolation_article} which is a generalization
of the one-dimensional construction (see, e.g., \cite[Chap. 7, eq.
(2.8)]{DEVore-Lorentz-93}).

\begin{theorem}
Let $d\in\mathbb{N}$ and $I:=%
\BIGOP{\times}%
_{j=1}^{d}I_{j}$ with $I_{i}$ being a bounded interval for every $i\in\left\{
1,\ldots,d\right\}  $. Let $n\in\mathbb{N}$. Then, for any $p\in\mathbb{N}$
with $p\geq n-1$, there exists a bounded linear operator $J_{n,p}:L^{1}\left(
I\right)  \rightarrow\bigotimes_{j=1}^{d}\mathbb{P}_{p}\left(  I_{j}\right)  $
with the following properties: For each $q\in\left[  1,\infty\right]  $, there
exists a constant $C>0$ depending only on $n$, $q$, and $I$ such that for all
$0\leq n\leq N$%
\begin{align*}
J_{n,p}u  &  =u\qquad\forall u\in%
{\displaystyle\bigotimes\limits_{j=1}^{d}}
\mathbb{P}_{n-1}\left(  I_{j}\right) \\
\left\Vert u-J_{n,p}\right\Vert _{W^{\ell,q}\left(  I\right)  }  &  \leq
C\left(  N+1\right)  ^{-\left(  r-\ell\right)  }\left\vert u\right\vert
_{W^{r,q}\left(  I\right)  },\qquad0\leq\ell\leq r\leq n.
\end{align*}
\end{theorem}

The proof of the following theorem is a generalization of the proof of
\cite[Thm. 2.1]{MelenkWohlmuth} and is carried out in detail in \cite[Thm.
3.1.10]{ZechMaster}. Here we skip it for brevity.

\begin{theorem}
[Cl\'{e}ment type quasi-interpolation]\label{TheoClQI}Let $\mathcal{T}$ be a
$\rho_{\mathcal{T}}$-shape regular, conforming simplicial finite element mesh
for the polygonal Lipschitz domain $\Omega\subseteq\mathbb{R}^{2}$. Let $\mathfrak{p}$
be a polynomial degree function for $\mathcal{T}$ satisfying (\ref{defrho2}).
Assume that $q\in\left[  1,\infty\right]  $ and let $n\in\mathbb{N}$.
\begin{subequations}
\label{C1Clement}
\end{subequations}%
\begin{enumerate}
\item[a.] Assume that $\left\lfloor \left(  p_{\mathcal{T}}-5\right)
/2\right\rfloor \geq n-1$. Then, there exists a bounded linear operator
$I_{n}^{\operatorname{hp}}:W^{n,q}\left(  \Omega\right)  \rightarrow
S_{\mathcal{T}}^{\mathfrak{p}}\cap C^{1}\left(  \Omega\right)  $ such that for
every $K\in\mathcal{T}$%
\begin{equation}
\left\vert u-I_{n}^{\operatorname{hp}}u\right\vert _{W^{m,q}\left(  K\right)
}\leq C\left(  \frac{h_{K}}{p_{K}}\right)  ^{n-m}\left\vert u\right\vert
_{W^{n,q}\left(  \omega_{K}\right)  } \qquad\forall m\in\left\{  0,\ldots,n\right\}  ,\tag{%
\ref{C1Clement}%
a}\label{C1Clementa}%
\end{equation}
and for every $e\in\mathcal{E}\left(  K\right)  $ and multiindex $\vartheta \in \mathbb{N}_0^2$ with $\vartheta_1+\vartheta_2=m\le n-1$
\begin{equation}
\left\Vert \frac{\partial^{m}}{\partial x^{\vartheta_1} \partial y^{\vartheta_2}} \left(  \left.  \left(  u-I_{n}^{\operatorname{hp}}u\right)
\right\vert _{K}\right)  \right\Vert _{L^{q}\left(  e\right)  }\leq C\left(
\frac{h_{e}}{p_{e}}\right)  ^{n-m-1/q}\left\vert u\right\vert _{W^{n,q}\left(
\omega_{e}\right)  },\tag{%
\ref{C1Clement}%
b}\label{C1Clementb}%
\end{equation}
where $C>0$ only depends on $n$, $q$, $\rho_{\mathcal{T}}$, and $\Omega$.
\item[b.] Assume that $\left\lfloor \left(  p_{\mathcal{T}}-1\right)
/2\right\rfloor \geq n-1$. Then, there exists a bounded linear operator
$I_{n}^{\operatorname{hp},0}:W^{n,q}\left(  \Omega\right)  \rightarrow
S_{\mathcal{T}}^{\mathfrak{p}}\cap C^{0}\left(  \Omega\right)  $ such that
(\ref{C1Clement}) holds with $I_{n}^{\operatorname{hp}}u$ replaced by
$I_{n}^{\operatorname{hp},0}u$ for a constant $C>0$ solely depending on $n$,
$q$, $\rho_{\mathcal{T}}$, and $\Omega$.
\end{enumerate}
\end{theorem}

\subsection{Conforming Approximation}

The a posteriori error analysis for our non-conforming
$\operatorname{dG}%
$-formulation requires the construction of conforming approximants of
non-conforming $hp$-finite element functions and this will be provided
next.

\begin{theorem}
[Conforming approximant]\label{TheoConfApprox}Let $\mathcal{T}$ be a
$\rho_{\mathcal{T}}$-shape regular, conforming simplicial finite
element mesh of the polygonal domain
$\Omega\subseteq\mathbb{R}^2$. Let $v\in
S_{\mathcal{T}}^{\mathfrak{p}}$, and let $\mathfrak{p}$ be a
polynomial degree function satisfying (\ref{defrho2}) and
$p_{\mathcal{T}}\geq1$. Then, there exists a constant $C>0$ which only
depends on the shape regularity and a function $v^{\ast}\in
S_{\mathcal{T}%
}^{\mathfrak{p}}\cap C^{0}\left( \Omega\right) $ such that%
\begin{subequations}
  \label{TheoNCE}
\end{subequations}%
\begin{align}
  \left\Vert v-v^{\ast}\right\Vert _{\partial \Omega } & \leq
  C\left\Vert
    \llbr
    v%
    \rrbr
  \right\Vert _{\mathfrak{S}^{I}},\tag{%
    \ref{TheoNCE}%
    a}\label{TheoNCEa}\\
  \left\Vert v-v^{\ast}\right\Vert & \leq C\left\Vert
    \mathfrak{h}^{1/2}
    \llbr
    v%
    \rrbr
  \right\Vert _{ \mathfrak{S}^{I} },\tag{%
    \ref{TheoNCE}%
    b}\label{TheoNCEb}\\
  \left\Vert \nabla\left( v-v^{\ast}\right) \right\Vert & \leq
  C\left\Vert \frac{\mathfrak{p}}{\mathfrak{h}^{1/2}}
    \llbr
    v%
    \rrbr
  \right\Vert _{ \mathfrak{S}^{I} }.\tag{%
    \ref{TheoNCE}%
    c}\label{TheoNCEc}%
\end{align}
\end{theorem}

For the proof of this theorem we refer to
\cite[Thm. 3.2.7]{ZechMaster} (see also, e.g.,
\cite{HoustonSchoetzWihl2007, Bonito2010}). \begin{corollary}
  [Conforming error]\label{CorConfApprox}Let the assumptions of
  Theorem \ref{TheoConfApprox} be satisfied. There exists a constant
  $C>0$ which only depends on the shape regularity constant
  $\rho_{\mathcal{T}}$ such that, for every $v\in
  S_{\mathcal{T}}^{p}$, there is a function $v^{\ast}\in
  S_{\mathcal{T}}^{\mathfrak{p}}\cap C^{0}\left( \Omega\right) $ with%
  \begin{align*}
    \left\Vert k\left( v-v^{\ast}\right) \right\Vert ^{2}+\left\Vert
      \nabla\left( v-v^{\ast}\right) \right\Vert ^{2}+&\left\Vert
      \sqrt{k}\left(
        v-v^{\ast}\right)  \right\Vert _{\partial\Omega}^{2}\\
    &\leq\frac{C}{\mathfrak{a}%
    }\left( 1+\frac{1}{p_{\mathcal{T}}}M_{\frac{\operatorname*{kh}}%
        {\operatorname*{p}}}+M_{\frac{\operatorname*{kh}}{\operatorname*{p}}}%
      ^{2}\right) \left\Vert
      \sqrt{{\mathfrak{a}\frac{\mathfrak{p}^2}{\mathfrak{h}}}}%
      \llbr
      v%
      \rrbr
    \right\Vert _{ \mathfrak{S}^{I} }^{2}.
  \end{align*}
\end{corollary}
\begin{proof}%
  The estimate follows by (\ref{TheoNCE}).
  \qquad\end{proof}%

\bibliographystyle{abbrv}
\bibliography{nlailu}

\newcommand{\noopsort}[1]{} \newcommand{\printfirst}[2]{#1}
  \newcommand{\singleletter}[1]{#1} \newcommand{\switchargs}[2]{#2#1}
  \def\cprime{$'$} \def\cprime{$'$}
\begin{thebibliography}{10}

\bibitem{Achenbach2005}
J.~D. Achenbach.
\newblock {\em Wave propagation in elastic solids}.
\newblock North Holland, Amsterdam, 2005.

\bibitem{AiOd3}
M.~Ainsworth and J.~T. Oden.
\newblock {\em A Posteriori Error Estimation in Finite Element Analysis}.
\newblock Wiley, 2000.

\bibitem{Strouboulis1}
I.~Babu{\v{s}}ka, F.~Ihlenburg, T.~Strouboulis, and S.~K. Gangaraj.
\newblock A posteriori error estimation for finite element solutions of
  {H}elmholtz' equation {I}. {T}he quality of local indicators and estimators.
\newblock {\em Internat. J. Numer. Methods Engrg.}, 40(18):3443--3462, 1997.

\bibitem{Strouboulis2}
I.~Babu{\v{s}}ka, F.~Ihlenburg, T.~Strouboulis, and S.~K. Gangaraj.
\newblock A posteriori error estimation for finite element solutions of
  {H}elmholtz' equation {II}. {E}stimation of the pollution error.
\newblock {\em Internat. J. Numer. Methods Engrg.}, 40(21):3883--3900, 1997.

\bibitem{BaRe1}
I.~Babu{\v s}ka and W.~C. Rheinboldt.
\newblock A-posteriori error estimates for the finite element method.
\newblock {\em Internat. J. Numer. Meth. Engrg.}, 12:1597--1615, 1978.

\bibitem{BaRe2}
I.~Babu{\v s}ka and W.~C. Rheinboldt.
\newblock Error estimates for adaptive finite element computations.
\newblock {\em SIAM J. Numer. Anal.}, 15:736--754, 1978.

\bibitem{Bonito2010}
A.~Bonito and R.~H. Nochetto.
\newblock Quasi-optimal convergence rate of an adaptive discontinuous
  {G}alerkin method.
\newblock {\em SIAM J. Numer. Anal.}, 48(2):734--771, 2010.

\bibitem{BuffaMonk2007}
A.~Buffa and P.~Monk.
\newblock Error estimates for the ultra weak variational formulation of the
  {H}elmholtz equation.
\newblock {\em M2AN Math. Model. Numer. Anal.}, 42(6):925--940, 2008.

\bibitem{CessenatDespres1998}
O.~Cessenat and B.~Despr{\'e}s.
\newblock Application of an ultra weak variational formulation of elliptic
  {PDE}s to the two-dimensional {H}elmholtz equation.
\newblock {\em SIAM J. Numer. Anal.}, 35:255--299, 1998.

\bibitem{CessenatDespres2003}
O.~Cessenat and B.~Despr{\'e}s.
\newblock Using plane waves as base functions for solving time harmonic
  equations with the ultra weak variational formulation.
\newblock {\em J. Computational Acoustics}, 11:227--238, 2003.

\bibitem{Graham_ActaNumerica_2012}
S.~N. Chandler-Wilde, I.~G. Graham, S.~Langdon, and E.~A. Spence.
\newblock Numerical-asymptotic boundary integral methods in high-frequency
  acoustic scattering.
\newblock {\em Acta Numer.}, 21:89--305, 2012.

\bibitem{cummings-feng06}
P.~Cummings and X.~Feng.
\newblock Sharp regularity coefficient estimates for complex-valued acoustic
  and elastic {H}elmholtz equations.
\newblock {\em Math. Models Methods Appl. Sci.}, 16(1):139--160, 2006.

\bibitem{D94}
B.~Despr{\'e}s.
\newblock Sur une formulation variationnelle de type ultra-faible.
\newblock {\em C. R. Acad. Sci. Paris S\'er. I Math.}, 318(10):939--944, 1994.

\bibitem{DEVore-Lorentz-93}
R.~A. DeVore and G.~G. Lorentz.
\newblock {\em Constructive Approximation}.
\newblock Springer-Verlag, New York, 1993.

\bibitem{PietroErn}
D.~A. Di~Pietro and A.~Ern.
\newblock {\em Mathematical aspects of discontinuous {G}alerkin methods}.
\newblock Springer, Heidelberg, 2012.

\bibitem{doerfler_sauter}
W.~D{\"o}rfler and S.~Sauter.
\newblock A posteriori error estimation for highly indefinite {H}elmholtz
  problems.
\newblock {\em Comput. Methods Appl. Math.}, 13(3):333--347, 2013.

\bibitem{MelenkHelmStab2010}
S.~Esterhazy and J.~M. Melenk.
\newblock On stability of discretizations of the {H}elmholtz equation.
\newblock In I.~Graham, T.~Hou, O.~Lakkis, and R.~Scheichl, editors, {\em
  Numerical {A}nalysis of {M}ultiscale {P}roblems}, volume~83 of {\em Lect.
  Notes Comput. Sci. Eng.}, pages 285--324. Springer, Berlin, 2012.

\bibitem{FW09}
X.~Feng and H.~Wu.
\newblock Discontinuous {G}alerkin methods for the {H}elmholtz equation with
  large wave number.
\newblock {\em SIAM J. Numer. Anal.}, 47(4):2872--2896, 2009.

\bibitem{FW}
X.~Feng and H.~Wu.
\newblock {$hp$}-discontinuous {G}alerkin methods for the {H}elmholtz equation
  with large wave number.
\newblock {\em Math. Comp.}, 80(276):1997--2024, 2011.

\bibitem{FX}
X.~Feng and Y.~Xing.
\newblock Absolutely stable local discontinuous {G}alerkin methods for the
  {H}elmholtz equation with large wave number.
\newblock {\em Math. Comp.}, 82(283):1269--1296, 2013.

\bibitem{GittelsonHiptmairPerugia2007}
C.~Gittelson, R.~Hiptmair, and I.~Perugia.
\newblock Plane wave discontinuous {G}alerkin methods: {A}nalysis of the
  h-version.
\newblock {\em Int. J. Numer. Meth. Engr.}, 43(2):297--331, 2009.

\bibitem{ACDG10}
M.~Grigoroscuta-Strugaru, M.~Amara, H.~Calandra, and R.~Djellouli.
\newblock A modified discontinuous {G}alerkin method for solving efficiently
  {H}elmholtz problems.
\newblock {\em Commun. Comput. Phys.}, 11(2):335--350, 2012.

\bibitem{IH06}
I.~Harari.
\newblock A survey of finite element methods for time-harmonic acoustics.
\newblock {\em Comput. Methods Appl. Mech. Engrg.}, 195(13-16):1594--1607,
  2006.

\bibitem{IH92}
I.~Harari and T.~J.~R. Hughes.
\newblock Galerkin/least-squares finite element methods for the reduced wave
  equation with nonreflecting boundary conditions in unbounded domains.
\newblock {\em Comput. Methods Appl. Mech. Engrg.}, 98(3):411--454, 1992.

\bibitem{hiptmair-moiola-perugia09b}
R.~Hiptmair, A.~Moiola, and I.~Perugia.
\newblock Plane wave discontinuous {G}alerkin methods for the 2{D} {H}elmholtz
  equation: analysis of the {$p$}-version.
\newblock {\em SIAM J. Numer. Anal.}, 49(1):264--284, 2011.

\bibitem{HoppeHelmholtz2013}
R.~H.~W. Hoppe and N.~Sharma.
\newblock Convergence analysis of an adaptive interior penalty discontinuous
  {G}alerkin method for the {H}elmholtz equation.
\newblock {\em IMA J. Numer. Anal.}, 33(3):898--921, 2013.

\bibitem{HoustonSchoetzWihl2007}
P.~Houston, D.~Sch{\"o}tzau, and T.~P. Wihler.
\newblock Energy norm a posteriori error estimation of {$hp$}-adaptive
  discontinuous {G}alerkin methods for elliptic problems.
\newblock {\em Math. Models Methods Appl. Sci.}, 17(1):33--62, 2007.

\bibitem{HuttungenMonk2007}
T.~Huttunen and P.~Monk.
\newblock The use of plane waves to approximate wave propagation in anisotropic
  media.
\newblock {\em J. Computational Mathematics}, 25:350--367, 2007.

\bibitem{Ihlenburgbook}
F.~Ihlenburg.
\newblock {\em Finite {E}lement {A}nalysis of {A}cousting {S}cattering}.
\newblock Springer, New York, 1998.

\bibitem{Frankp}
F.~Ihlenburg and I.~Babu{\v s}ka.
\newblock Finite {E}lement {S}olution to the {H}elmholtz {E}quation with {H}igh
  {W}ave {N}umber. {P}art {II}: {T}he h-p version of the {FEM}.
\newblock {\em SIAM J. Num. Anal.}, 34(1):315--358, 1997.

\bibitem{IrimieBuillard}
S.~Irimie and P.~Bouillard.
\newblock A residual a posteriori error estimator for the finite element
  solution of the {H}elmholtz equation.
\newblock {\em Comput. Methods Appl. Mech. Engrg.}, 190(31):4027--4042, 2001.

\bibitem{Jackson02}
J.~D. Jackson.
\newblock {\em Classical electrodynamics}.
\newblock John Wiley \& Sons, Inc., New York-London-Sydney, second edition,
  1975.

\bibitem{KellogReg}
R.~Kellogg.
\newblock Interpolation between subspaces of a {H}ilbert space.
\newblock Technical Report BN-719, Institute for {F}luid {D}ynamics and
  {A}pplied {M}athematics, University of Maryland at College Park, College
  Park, MD, 20742-2431,USA, 1971.

\bibitem{MelenkDiss}
J.~M. Melenk.
\newblock {\em On {G}eneralized {F}inite {E}lement {M}ethods}.
\newblock PhD thesis, University of Maryland at College Park, 1995.

\bibitem{melenk_nshpinterpolation_article}
J.~M. Melenk.
\newblock {$hp$}-interpolation of nonsmooth functions and an application to
  {$hp$}-a posteriori error estimation.
\newblock {\em SIAM J. Numer. Anal.}, 43(1):127--155, 2005.

\bibitem{MPS13}
J.~M. Melenk, A.~Parsania, and S.~Sauter.
\newblock Generalized {DG} {M}ethods for {H}ighly {I}ndefinite {H}elmholtz
  {P}roblems {B}ased on the {U}ltra-{W}eak {V}ariational {F}ormulation.
\newblock {\em J Sci Comput}, 57:536--581, 2013.

\bibitem{MelenkSauterMathComp}
J.~M. Melenk and S.~Sauter.
\newblock Convergence {A}nalysis for {F}inite {E}lement {D}iscretizations of
  the {H}elmholtz equation with {D}irichlet-to-{N}eumann boundary condition.
\newblock {\em Math. Comp}, 79:1871--1914, 2010.

\bibitem{mm_stas_helm2}
J.~M. Melenk and S.~Sauter.
\newblock Wave-{N}umber {E}xplicit {C}onvergence {A}nalysis for {G}alerkin
  {D}iscretizations of the {H}elmholtz {E}quation.
\newblock {\em SIAM J. Numer. Anal.}, 49(3):1210--1243, 2011.

\bibitem{MelenkWohlmuth}
J.~M. Melenk and B.~I. Wohlmuth.
\newblock On residual-based a posteriori error estimation in {$hp$}-{FEM}.
\newblock {\em Adv. Comput. Math.}, 15(1-4):311--331 (2002), 2001.

\bibitem{monk-wang99}
P.~Monk and D.-Q. Wang.
\newblock A least-squares method for the {H}elmholtz equation.
\newblock {\em Comput. Methods Appl. Mech. Engrg.}, 175(1-2):121--136, 1999.

\bibitem{NochSiebVees2009}
R.~H. Nochetto, K.~G. Siebert, and A.~Veeser.
\newblock Theory of adaptive finite element methods: an introduction.
\newblock In {\em Multiscale, nonlinear and adaptive approximation}, pages
  409--542. Springer, Berlin, 2009.

\bibitem{Oden2005}
J.~T. Oden, S.~Prudhomme, and L.~Demkowicz.
\newblock A posteriori error estimation for acoustic wave propagation problems.
\newblock {\em Arch. Comput. Methods Engrg.}, 12(4):343--389, 2005.

\bibitem{Schatz74}
A.~H. Schatz.
\newblock An observation concerning {R}itz-{G}alerkin methods with indefinite
  bilinear forms.
\newblock {\em Math. Comp.}, 28:959--962, 1974.

\bibitem{SchwabhpBook}
C.~Schwab.
\newblock {\em {$p$}- and {$hp$}-finite element methods}.
\newblock The Clarendon Press Oxford University Press, New York, 1998.

\bibitem{verfuerth_book_neu}
R.~Verf{\"u}rth.
\newblock {\em A posteriori error estimation techniques for finite element
  methods}.
\newblock Oxford University Press, Oxford, 2013.

\bibitem{wu12a}
H.~Wu.
\newblock Pre-asymptotic error analysis of cip-fem and fem for the helmholtz
  equation with high wave number. part i: linear version.
\newblock {\em IMA Journal of Numerical Analysis}, 2013.

\bibitem{ZechMaster}
J.~Zech.
\newblock A {P}osteriori {E}rror {E}stimation of {$hp$}-{DG} {F}inite {E}lement
  {M}ethods for {H}ighly {I}ndefinite {H}elmholtz {P}roblems.
\newblock Master's thesis, Inst. f. Mathematik, Unversit{\"{a}}t Z{\"{u}}rich,
  2014.
\newblock \url{https://www.math.uzh.ch/compmath/index.php?id=zech}.

\bibitem{zhu-wu12b}
L.~Zhu and H.~Wu.
\newblock Preasymptotic error analysis of {CIP}-{FEM} and {FEM} for {H}elmholtz
  equation with high wave number. {P}art {II}: {$hp$} version.
\newblock {\em SIAM J. Numer. Anal.}, 51(3):1828--1852, 2013.

\end{thebibliography}

\end{document}